\documentclass[11pt,oneside,reqno]{amsart}
\usepackage{amssymb}
\usepackage{}
\usepackage{amsmath}
\usepackage{graphicx}
\usepackage{mathrsfs}
\usepackage{amsfonts}
\usepackage{enumerate,amsmath,amssymb,amsthm}
\usepackage{comment}

\usepackage{fancyhdr}
\usepackage{hyperref}
\hypersetup{backref,
colorlinks=true,
linkcolor=blue,
anchorcolor=blue,
citecolor=blue}

\numberwithin{equation}{section}

\newcommand{\be}{\begin{eqnarray}}
\newcommand{\ee}{\end{eqnarray}}
\newcommand{\ce}{\begin{eqnarray*}}
\newcommand{\de}{\end{eqnarray*}}
\newtheorem{theorem}{Theorem}[section]
\newtheorem{lemma}[theorem]{Lemma}
\newtheorem{remark}[theorem]{Remark}
\newtheorem{definition}[theorem]{Definition}
\newtheorem{proposition}[theorem]{Proposition}
\newtheorem{Examples}[theorem]{Example}
\newtheorem{corollary}[theorem]{Corollary}
\newtheorem{assumption}{Assumption}[section]

\def\e{{\mathrm{e}}}
\def\eps{\varepsilon}

\def\e{\varepsilon}

\def\[{{\Big[}}
\def\]{{\Big]}}
\def\<{{\langle}}
\def\>{{\rangle}}
\def\({{\Big(}}
\def\){{\Big)}}

\def\bx{{\mathbf{x}}}

\def\dif{{\mathord{{\rm d}}}}

\def\min{{\mathord{{\rm min}}}}

\def\={&\!\!=\!\!&}
\def\bt{\begin{theorem}}
\def\et{\end{theorem}}
\def\bl{\begin{lemma}}
\def\el{\end{lemma}}
\def\br{\begin{remark}}
\def\er{\end{remark}}
\def\bas{\begin{assumption}}
\def\eas{\end{assumption}}
\def\bd{\begin{definition}}
\def\ed{\end{definition}}
\def\bp{\begin{proposition}}
\def\ep{\end{proposition}}
\def\bc{\begin{corollary}}
\def\ec{\end{corollary}}
\def\bx{\begin{Examples}}
\def\ex{\end{Examples}}

\def\cB{{\mathcal B}}

\def\cF{{\mathcal F}}

\def\mN{{\mathbb N}}

\def\geq{\geqslant}
\def\leq{\leqslant}

\allowdisplaybreaks
\usepackage{color}

\newcommand\dela[1]{}

\title[LDP  for  stochastic nonlinear Schr\"odinger equations driven by L\'evy noise]{{Large deviation principles for  stochastic nonlinear Schr\"odinger equations driven by L\'evy noise$^\dagger$}
}
  \thanks{$\dagger$  This work is partially supported by National Key R\&D Program of China(No. 2022YFA1006001), National Natural Science Foundation of China (Nos. 12131019, 11971456, 12071433, 12171208, 11831014, 12090011), and the Fundamental Research Funds for the Central Universities (No. WK3470000031).}
   \thanks{$*$ Corresponding author}
\begin{document}

\maketitle
\centerline{\scshape Jiahui Zhu$^a$, Wei Liu$^b$, Jianliang Zhai$^{c,*}$ }
\medskip
 {\footnotesize
\centerline{ $a.$   School of Science, Zhejiang University of Technology, Hangzhou 310019, China }
 \centerline{ $b.$ School of Mathematics and Statistics, Jiangsu Normal University, Xuzhou
221116, China}
 \centerline{$c.$ School of Mathematical Sciences, University of Science and Technology of China, Hefei 230026, China}
 }

\begin{abstract}
In this work we establish a Freidlin-Wentzell type large deviation principle for  stochastic nonlinear Schr\"{o}dinger equation, with either focusing or defocusing nonlinearity, driven by nonlinear multiplicative L\'evy noise in the Marcus canonical form. This task is challenging in the current setting due to the presence of the power-type
nonlinear term, the lack of regularization effect of the Schr\"odinger operator and the absence
of compactness of embeddings. To overcome these difficulties, we employ a regularization procedure based on Yosida approximations and implement techniques such as time discretization, cut-off arguments, and relative entropy estimates of  sequences of probability measures. Our innovative approach circumvents the need for compactness conditions, distinguishing our work from previous studies.

\hspace{1cm}

\textit{Keywords}: Nonlinear Schr\"odinger equation, multiplicative L\'evy noise,  Marcus canonical  form, weak convergence method, large deviation principle

\textit{Mathematics Subject Classification}: 35Q55, 35R60,  60H15,  60G51, 60F10
\end{abstract}

\section{Introduction}
Nonlinear Schr\"odinger equation (NLS) has become an important and influential subject since its inception, mainly due to its  numerous applications in physics.  NLS serves as a universal model for describing  the wave propagation in weakly nonlinear dispersive media.
Its principal applications include, but are not limited to, hydrodynamics,  nonlinear optics, nonlinear acoustic, Bose-Einstein condensate and plasma physics, see e.g. \cite{Tur-12} and the references therein.

To describe the propagation  of  nonlinear dispersive waves through a non-homogeneous or random media, more realistic mathematical models can be constructed by adding suitable noisy driving terms to the equation \cite{DC-01, FK-01, KV-94}. Starting with the works \cite{deB+De-99, deB+De-03}, stochastic NLS with Gaussian noise have been investigated by many researchers, including \cite{BRZ-14,  BRZ-17, BM-14, BHW-19} for the well-posedness, \cite{BRZ-18} for the optimal control, \cite{Gautier05,Gautier051,ZD-22} for Freidlin-Wentzell type large deviation principle (LDP), \cite{HRZ-19} for scattering behavior, \cite{BFZ-23} for ergodic results, \cite{Cao+Su+Zhang, Su+Zhang} for blow-up solutions.
However, Gaussian noise cannot incorporate the possibility of having sudden, sharp changes in the field that commonly occur in real world models.  Compared to the case of stochastic NLS with Gaussian noise, there has been considerably less results on the case of jump noise, with only a small number of papers available in the literature. 
In \cite{deB+Hau-19}, de Bouard and Hausenblas established the existence of a martingale solution of stochastic NLS with linear jump noise in $H^{1,2}\left(\mathbb{R}^n\right)$  and in \cite{deB+Hau-On-19} they proved pathwise uniqueness.
In \cite{BHM-20}, Brze\'zniak et al. constructed a martingale solution of the stochastic NLS with  multiplicative jump noise in $H^1$ on compact manifolds. 
Recently,  Brze\'zniak et al. in \cite{BLZ} established  the existence and uniqueness of global solutions to stochastic nonlinear Schr\"odinger equation with L\'evy noise of Marcus type in $L^2(\mathbb{R}^d)$. To the best of our knowledge, there is no result available such as  ergodicity, optimal control, scattering behavior for stochastic NLS with jump noise. The dearth of existing research on stochastic NLS with jump-type noise provides  a strong impetus for our investigation and this paper presents the first result on Freidlin-Wentzell type LDP for stochastic NLS with jump noise.  The large deviation principles mentioned in this paper are all of Freidlin-Wentzell type and will be referred as LDP for the sake of convenience.

The LDP describes the limiting behavior of  the laws of solutions when the noise in the equation converges to zero, in terms of a rate function.
 On one hand,  large deviation
techniques are precious tools to  study  the first exit time from a neighborhood of an
asymptotically stable equilibrium point, the exit place determination or the transition between two equilibrium points in randomly perturbed dynamical systems, which is
important in the fields of statistical and
quantum mechanics, chemical reactions etc, see e.g. \cite{Dembo-Zeitouni-1998, Freidlin-Wentzell-1998}.  In theoretical physics, the techniques used to study these exit problems are often called optimal
fluctuations or instanton formalism and are closely related to large deviations.
On the other hand, LDP allows researchers to identify and analyze rare but extreme events, and it is a crucial tool in interpreting rogue waves in the ocean \cite{DGOV,DGV,DP-11}, as well as other relevant physical systems where an understanding of extreme events is important. Rogue waves are a particularly dangerous hazard that can cause significant damage to marine and coastal environments \cite{Nik+Did}. Estimating the likelihood of appearance of rogue waves is  an important and challenging problem with practical implications for the safety and resilience of coastal infrastructure.
Recently, the NLS has been introduced to study  hydrodynamic rogue waves generated by nonlinear energy transfer in the open ocean \cite{HP-99}, as well as optical rogue waves in an optical system \cite{So+Ro+KJ}. It has emerged as a well-suited model for the evolution of unidirectional, narrow-banded surface wave fields in deep sea \cite{ DP-11}.  Besides, both experimental and numerical evidence indicate that rogue waves
 can be described as  hydrodynamic instantons within the framework of large deviation  \cite{DGOV,DGV}.
  These  further  motivate our study of the LDP for stochastic NLS and we hope that the theory developed in this work may contribute to  deeper understanding of the behavior of various other dynamical systems and their mechanism of creation of extreme events.
 \vskip 0.2cm

The aim of this paper is to establish LDP for the solutions of following stochastic nonlinear Schr\"{o}dinger equation with pure jump noise in Marcus canonical form
\begin{align}
\begin{split}\label{eq X0-1}
\mathrm{d} u^{\varepsilon}(t)&=\mathrm{i} \left[\Delta u^{\varepsilon}(t)-\lambda|u^{\varepsilon}(t)|^{2 \sigma} u^{\varepsilon}(t)\right] \mathrm{d} t-\varepsilon\, \mathrm{i} \sum_{j=1}^m g_{j}(u^\varepsilon(t-))\diamond \dif L_j^{\varepsilon^{-1}}(t),\ t\in[0,T], \\
 u^{\varepsilon}(0)& =u_0,
 \end{split}
\end{align}
where $\varepsilon>0$ is a small parameter converging to 0, $\lambda\in\mathbb{R}$, $(L_1^{\varepsilon^{-1}}(t),...,L_m^{\varepsilon^{-1}}(t))=L^{\varepsilon^{-1}}(t):=\int_0^t\int_B z\tilde{N}^{\varepsilon^{-1}}(\mathrm{d} z,\mathrm{d}s)$. Here $B=\{z\in\mathbb{R}^m:0<|z|\leq 1\}$ and $\tilde{N}^{\varepsilon^{-1}}$ is a compensated Poisson random measure on $[0, T] \times Z$ with a $\sigma$-finite mean measure $\varepsilon^{-1} \lambda_T \otimes \nu$, where $\lambda_T$ is the Lebesgue measure on $[0, T]$ and $\nu$ is a $\sigma$-finite measure on $B$. The precise definition of the noise and the Marcus product $\diamond$ is provided in Section \ref{sec-assp}. We remark that the nonlinearity with coefficient $\lambda>0$ is called defocusing case and $\lambda<0$  focusing case.

From the physical point of view, Stratonovich integration is preferred in Gaussian noise-driven systems \cite{Brz+Gol+Jeg, deB+De-99}. It allows us to treat stochastic integrals according to the conventional integration rules, and it can be interpreted as Wong-Zakai type approximation, which is useful as a numerical method for approximations of solutions of SPDEs.
 However, for systems with jump noises, those nice features are not exhibited by the Stratonovich prescription. The choice of Marcus canonical noise in our equation \eqref{eq X0-1} has the same physical bases and pertains remarkably similar properties as Stratonovich noise in the continuous case \cite{Brz+Man-19}.  Base on this perspective, the Marcus prescription is a more suitable tool for describing physical systems with jump noises; see Remark \ref{remark-marcus-int} for more details. Concerning the NLS, Marcus canonical type L\'evy noise guarantees the conservation of the $L^2(\mathbb{R}^d)$-norm of the solution which is the natural and consistent generalization of the deterministic NLS \cite{BLZ}.  It is noteworthy that stochastic partial differential equations (SPDEs) driven by L\'evy noise in the Marcus canonical form have only recently been explored, see e.g. \cite{Brz+Man-19-1,Brz+Man-19} for the existence result of the Landau-Lifshitz-Gilbert equation, \cite{BLZ, BHM-20} for the existence results of  nonlinear Schr\"odinger equations and \cite{MP-21} for well-posedness and LDP of the 2D constrained Navier-Stokes equations. Due to the appearance of the additional discontinuous component corresponding to the Marcus type noise, see equation \eqref{eq X0}, specialized techniques and tools are required, which distinguishes the current work from the previous studies on LDP for SPDEs driven by It\^{o}-type L\'evy noises, e.g. \cite{BPZ, LSZ}.  We believe that the present study may offer new insights into the behavior of SPDEs driven by L\'evy noise.
We would like to make a short remark that the methods to obtain LDP used in \cite{MP-21}, which heavily rely on the
compactness condition for the Gelfand triples, cannot be applied to prove our result. One of the main difference is that our setting does not have the compactness of embeddings.
We will provide further details about this difference later in the introduction.

%\vskip 0.2cm

In the last few decades, the LDP has been established for many important equations in mathematical physics (see e.g. \cite{BPZ, Brz+Gol+Jeg, BD-19, Budhiraja-Dupuis-Maroulas., DWZZ-20, HLL1, HLL2, LSZ, MSZ-21} and the references therein).  Concerning NLS,   the LDP for Gaussian noise has been studied in several works \cite{Gautier05, Gautier051, ZD-22}.
Gautier in \cite{Gautier05} obtained a  LDP for stochastic NLS with additive Gaussian noise, and continued his work in \cite{Gautier051} by proving a uniform LDP for linear multiplicative Gaussian noise. 
Zhang in \cite{ZD-22} also established  a LDP for stochastic NLS with  linear multiplicative Gaussian noise. The proofs of these works relied on the Varadhan contraction principle. 
An important point to mention here is that if the noise in stochastic NLS is additive or linear multiplicative, the Varadhan contraction principle can be implemented, which involves constructing a continuous function. However, in the case of nonlinear multiplicative noise, constructing a continuous function is either tedious or infeasible, and the problem also becomes much more complicated in the jump noise case. Moreover, there are significant differences between the Wiener case
and  L\'evy case, and applying the methods in  \cite{Gautier05, Gautier051, ZD-22} to deal with our problem  is extremely difficult,  if not impossible at all,  even for the cases of additive or linear multiplicative L\'evy  noises.
%We emphasize that in this paper we consider stochastic NLS with nonlinear multiplicative L\'evy noise which is inherently  not applicable by the rescaling approach. 
This paper presents the first result on LDP for stochastic NLS incorporating nonlinear noise. 
 New and effective ideas have been employed to obtain our
result.   We believe that  these ideas proposed in this paper could be used to obtain the LDP for stochastic NLS with nonlinear multiplicative Gaussian noise,
which appears also to be open problems.
%\vskip 0.2cm

The LDP is a consequence of the weak approach to large deviations
established in \cite{Budhiraja-Dupuis-Maroulas.} \cite{Budhiraja-Chen-Dupuis}, as well as an improved version \cite{LSZ}. The proof of the LDP relies on establishing two facts. First, we establish the well-posedness, some appropriate convergence of the solutions of the corresponding skeleton equation. Second, we require  the solutions of the  controlled stochastic NLS  converge to the solution of the skeleton equation; see Section \ref{sec-SE}, Propositions \ref{Yu-con} and \ref{lem LDP 2} in Section \ref{sec-proof-Th} respectively.

The primary difficulty of our proof stems from the power-nonlineary and the lack of regularization effect of the Schr\"odinger operator. Particularly, the Schr\"odinger operator generates a unitary group on $H^s$, $s\geq 0$, which is not good enough for the large deviation estimates. To tackle these issues, we employ a regularization procedure based on Yosida approximations to obtain the well-posedness and a priori estimates for the approximate skeleton equation by means of Strichartz estimates. Thanks to the characteristic of the Marcus-type noise, the $L^2$-norm of the solutions of the approximate skeleton equation is conserved and we are able to prove LDP in the same space of the solutions.

 The second main difficulty of our proof is that in our setting there is no compactness of embeddings. In this respect it differs from most of previous papers on LDP for SPDEs with L\'evy noise, notably \cite{BPZ,   LSZ, MP-21, YZZ}, in which their proofs  crucially depended on the compactness condition. The only work towards LDP for stochastic evolution equations with L\'evy noise without compactness conditions is a 
 recent paper  \cite{WZ-22}, in which the authors established a LDP for a class of stochastic porous media equations driven by It\^{o}-type L\'evy noise.
   However, due to the lack of regularization effect of the Schr\"odinger operator and the different natures of the nonlinearity in our equation compared to \cite{WZ-22} and those in previous studies, our problem is significantly harder. Therefore, new methods need to be employed to address these challenges. 
   
To demonstrate the convergence of the associated skeleton equation, we first approximate the nonlinear multiplicative noise coefficients by a family of nicer nonlinearities. To be more precise, we regularize the nonlinear noise coefficients by applying $\mu(\mu I -\Delta)^{-1}$ and introduce another auxiliary approximation processes $\Phi^n_{\mu}$, see \eqref{eq LDP 4}. We establish some uniform estimates of the sequence $\Phi^n_{\mu}$ and prove the convergence of  the corresponding skeleton equation with the help of $\Phi^n_{\mu}$. The proof of the convergence also involves a discretization argument and the use of relative entropy estimates of a sequence of probability measures.
To complete the proof of LDP, it is also necessary to verify the convergence in distribution of the controlled stochastic equation to the corresponding skeleton equation in  $L^{\infty}(0, T; L^2(\mathbb{R}^d)) \cap L^{p}(0, T; L^r(\mathbb{R}^d))$.  This is a non-trivial task as we need to find uniform estimates for the solutions of controlled stochastic equation in the space $L^{p}(0,T;L^r(\mathbb{R}^d))$.   In addition,  some localization arguments and the conservation laws  are employed to establish the convergence of the solutions of the controlled stochastic NLS. Due to the special features of the Marcus noise, some other technical difficulties have also to be overcome in the proof.
We point out that our results allow the full range of subcritical exponents $\sigma\in(0,\frac{2}{d})$, compared to \cite{deB+De-99}.

The rest of the paper is organized as follows. In Section \ref{sec-assp}, we describe our framework precisely and state the main result. In Section \ref{sec-SE}, we mainly analyze the solution to the skeleton equation. Section \ref{sec-proof-Th} is devoted to the proof of the large deviation principle.  In the end of this section, we state two main sufficient conditions for establishing the LDP, where their proofs are given in
 Section \ref{sect-prop-1} and \ref{sect-prop-2} respectively.

\section{Assumptions and statement of the main result}\label{sec-assp}

We begin with a brief introduction to some relevant notations used in this paper.
Set $\mathbb{N}:=\{1,2,3,\cdots\}$, $\mathbb{R}:=(-\infty,\infty)$ and $\mathbb{R}_+:=[0,\infty)$. For a metric space $\mathcal{H}$, the Borel $\sigma$-field on $\mathcal{H}$ will be written as
$\mathcal{B}(\mathcal{H})$. We denote by $C_c(\mathcal{H})$ the space of real-valued continuous functions with compact supports. Let $C([0,T];\mathcal{H})$ be the space of continuous functions $f:[0,T]\rightarrow \mathcal{H}$ endowed with the uniform convergence topology. Let $D([0,T];\mathcal{H})$ be the space of all c\`adl\`ag functions $f:[0,T]\rightarrow \mathcal{H}$ endowed with the Skorokhod topology. For an $\mathcal{H}$-valued measurable map
$X$ defined on some probability space $(\Omega,\mathcal{F},P)$ we use $Law(X)$ to denote the measure induced by $X$ on $(\mathcal{H},\mathcal{B}(\mathcal{H}))$.
 We use the symbol $``\Rightarrow"$
to denote the convergence in distribution.

For a locally compact Polish space $\mathcal{H}$, the space of all Borel measures on $\mathcal{H}$ is denoted by $M(\mathcal{H})$, and $M_{FC}(\mathcal{H})$ denotes the set of all  $\mu\in M(\mathcal{H})$ with $\mu(O)<\infty$
for each compact subset $O\subseteq \mathcal{H}$. We endow $M_{FC}(\mathcal{H})$ with the weakest topology such that for each $f\in C_c(\mathcal{H})$ the mapping
$\mu\in M_{FC}(\mathcal{H})\rightarrow \int_\mathcal{H}f(s)\mu(\dif s)\in\mathbb{R}$ is continuous. This topology is metrizable such that $M_{FC}(\mathcal{H})$ is a Polish space, see \cite{Budhiraja-Dupuis-Maroulas.} for more details.

We fix $T>0$ and $m\in\mathbb{N}$ throughout this paper.  Let $B=\{z\in\mathbb{R}^m:0<|z|\leq 1\}$, and $\nu$ a given $\sigma$-finite intensity measure of a L\'evy process on $B$, see \cite{Sato 1999}, that is, $\nu$ is a  $\sigma$-finite Borel measure  on $B$  satisfying
\begin{eqnarray}\label{Assump on nu}
    \int_B|z|^2\nu(\dif z)<\infty.
\end{eqnarray}
Clearly, $\nu\in M_{FC}(B)$.
 In this paper, the  probability space $(\Omega, \cF, {\mathbb F}:=\{\cF_t\}_{t\in [0,T]},P)$ is specified as follows
\begin{align*}
  \Omega:= M_{FC}\big([0,T]\times B
  \times \mathbb{R}_+\big),\qquad \cF:=\cB(\Omega).
\end{align*}
We introduce the coordinate mappings
\begin{align*}
 N\colon \Omega \rightarrow M_{FC}\big([0,T]\times B\times \mathbb{R}_+\big),\qquad  N(\omega)=\omega.
\end{align*}
Define for each $t\in [0,T]$ the $\sigma$-algebra
\begin{align*}
\mathcal{G}_{t}:=\sigma\left(\left\{N((0,s]\times A):\,
0\leq s\leq t,\,A\in \mathcal{B}\big(B\times \mathbb{R}_+\big)\right\}\right).
\end{align*}
For the given $\nu$, it follows from \cite[Sec.I.8]{Ikeda-Watanabe} that there exists a unique probability measure $P$
 on $(\Omega,\mathcal{F})$ such that 
 $N$ is a Poisson random measure (PRM) on $[0,T]\times B\times\mathbb{R}_+$ with intensity measure $\text{Leb}_T\otimes\nu\otimes \text{Leb}_\infty$, where
$\text{Leb}_T$ and $\text{Leb}_\infty$ stand for the Lebesgue measures on $[0,T]$ and $\mathbb{R}_+$ respectively.

We denote by $\mathbb{F}:=\{{\mathcal{F}}_{t}\}_{t\in[0,T]}$ the
$P$-completion of $\{\mathcal{G}_{t}\}_{t\in[0,T]}$ and
$\mathcal P$ the $\mathbb{F}$-predictable $\sigma$-field
on $[0,T]\times \Omega$. The  PRM $N$
is defined on  the (filtered) probability space $(\Omega, \cF, \mathbb{F}:=\{\cF_t\}_{t\in [0,T]},P)$.
and the corresponding compensated PRM is denoted by $\widetilde{N}$.

Denote
\begin{align*}
{\mathcal R_+}
=\left\{\varphi\colon [0,T]\times \Omega\times B\to \mathbb{R}_+: \varphi\ \text{is}
\, (\mathcal{P}\otimes\mathcal{B}(B))/\mathcal{B}(\mathbb{R}_+)\text{-measurable}\right\}.
\end{align*}
For any $\varphi\in{\mathcal R_+}$, $N^{\varphi}:\Omega\rightarrow M_{FC}([0,T]\times B)$ is  a
counting process  on $[0,T]\times {B}$ defined by
   \begin{align}\label{Jump-representation}
      N^\varphi((0,t]\times A)=\int_{(0,t]\times A\times \mathbb{R}_+}1_{[0,\varphi(s,z)]}(r)\, N(\dif s, \dif z, \dif r),\ 0\le t\le T,\ A\in\mathcal{B}(B).
   \end{align}
Here $N^\varphi$ can be viewed as a controlled random measure, with $\varphi$ selecting the intensity.

Analogously, $\widetilde{N}^\varphi$ is defined by replacing $N$ with $\widetilde{N}$ in (\ref{Jump-representation}). When $\varphi\equiv c>0$, we write $N^\varphi=N^c$ and $\widetilde{N}^\varphi=\widetilde{N}^c$.

\vskip 0.3cm

For each measurable function
$g\colon [0,T]\times B\to [0,\infty)$, define
\begin{align*}
Q(g):=\int_{[0,T]\times B}\ell\big(g(s,z)\big) \, \nu(\dif z)\dif s,
\end{align*}
where $\ell(x)=x\log x-x +1,\ \ell(0):=1$.
For each $N>0$, denote
\begin{align}\label{S_N}
     W^N:=\Big\{g:[0,T]\times B\rightarrow[0,\infty):\,Q(g)\leq N\Big\}.
\end{align}
Any measurable function $g\in W^N$ can be identified with a measure $\nu^g_T\in M_{FC}([0,T]\times B)$, defined by
   \begin{align}\label{eq.corres-func-meas}
      \nu^g_T(A)=\int_A g(s,z)\, \nu_T(\dif s,\dif z),\ \forall A\in\mathcal{B}([0,T]\times B),
   \end{align}
where we set $\nu_T=\text{Leb}_T\otimes\nu$.
This identification induces a topology on $W^N$ under which $W^N$ is a compact space (see the Appendix of \cite{Budhiraja-Chen-Dupuis}).
Throughout the paper, we use this topology on $W^N$.

Denote
\begin{eqnarray}\label{eq P22}
W:= \bigcup_{N\in\mathbb{N}}W^N.
\end{eqnarray}

Let $\{Z_n=\{z\in B:\frac{1}{n}\leq |z|\leq 1\}\}_{n\in\mN}$ then $Z_n\subseteq B$,
$ Z_n \nearrow B$ and $\nu(Z_n)<\infty$.  For each $n\in\mN$, let
\begin{align}\label{Eq Rn}
     \mathcal{R}_{b,n}
= \Big\{\psi\in \mathcal{R}_+:
\psi(t,z,\omega)\in \begin{cases}
 [\tfrac{1}{n},n], &\text{if }z\in Z_n\\
\{1\}, &\text{if }z\in Z_n^c
\end{cases}
\text{ for all }(t,\omega)\in [0,T]\times \Omega
\Big\},
\end{align}
and $\mathcal{R}_{b}=\bigcup _{n=1}^\infty \mathcal{R}_{b,n}$.
For any $N\in(0,\infty)$, let $\mathcal{W}^N$ be a space of stochastic processes on $\Omega$ defined by
\begin{align*}
\mathcal{W}^N:=\{\psi\in  \mathcal{R}_{b}:\, \psi(\cdot,\cdot,\omega)\in W^N
\text{ for $P$-a.e. $\omega\in \Omega$}\}.
\end{align*}

In this paper, we make the following assumption on the noise coefficients $g_j$, $j=1,\cdots,m$:
\begin{assumption}\label{assm-g}
For each $1 \leq j \leq m$, there exists a function $\tilde{g}_{j}:[0, \infty) \rightarrow \mathbb{R}$ of class $C^{1}$ such that $g_{j}$ is given by $g_{j}(y)=\tilde{g}_{j}\left(|y|^{2}\right) y, y \in \mathbb{C}$. We also assume that there exist constants $L_{1}, L_{2}>0$ such that for all $x, y \in \mathbb{C}$
\begin{align}
\max _{1 \leq j \leq m}\left|g_{j}(x)-g_{j}(y)\right| & \leq L_{1}|x-y|, \label{cond-g-1}\\
\max _{1 \leq j, k \leq m}\left|g_{j}^{\prime}(x) g_{k}(x)-g_{j}^{\prime}(y) g_{k}(y)\right| & \leq L_{2}|x-y| .\label{cond-g-2}
\end{align}
\end{assumption}

\begin{remark}
This kind of nonlinearities $g_j$ are known as saturation nonlinearities. In optics, such nonlinearities are often used to describe propagation of intense beams through saturable media, such as a photorefracctive crystal, where the refraction index changes to an upper bound and then ceases to increase with the increasing of the intensity. 

Assume that $\tilde{g}_j:[0, \infty) \rightarrow \mathbb{R}$ is twice continuously differentiable and satisfies
\begin{align}\label{Ass-g-boundedness-1}
\max_{1\leq j\leq m}\sup _{\theta>0} \big[\tilde{g}_j(\theta)+(1+\theta)\tilde{g}_j^{\prime}(\theta)+(1+\theta^{\frac32})\tilde{g}_j^{\prime\prime}(\theta)\big]<\infty.
\end{align} 
Consequently, the function $g_j$ meets the conditions outlined in  Assumption \ref{assm-g}. In particular, for each $j$, we can choose $\tilde{g}_j$ to be either the saturation of the intensity nonlinearity $\tilde{g}_j(\theta)=\frac{\theta}{1+\rho \theta}$, with $\rho>0$,  which characterizes photorefractive media \cite{Kelley}, or the square-root nonlinearity $\tilde{g}_j(\theta)=1-\frac1{(1+\theta)^{1/2}}$, which describes narrow-gap semiconductors \cite{SBB}. We can also choose $ \tilde{g}_j(\theta)=\frac{\theta(2+\rho \theta)}{(1+\rho \theta)^2}$ or  $\tilde{g}_j(\theta)=\frac{\log (1+\rho \theta)}{1+\log (1+\rho \theta)}$,  $\theta\geq 0$, with $\rho>0$. 

\end{remark}

Let $\Phi^{\varepsilon}(z, x)$ be the value at time $t=1$ of the solution of the following equation: for each $t\in[0,T]$, $z\in B$, $x\in\mathbb{C}$,
\begin{align}
\begin{split}
\frac{\partial \Phi^{\varepsilon}}{\partial t}(t, z, x)&=-\varepsilon\,\mathrm{i} \sum_{j=1}^{m} z_{j} g_{j}(\Phi^{\varepsilon}(t, z, x)), \\
 \Phi^{\varepsilon}(0, z, x)&=x.
  \end{split}
\end{align}
Then equation \eqref{eq X0-1} with notation $\diamond$ is defined in the integral form as follows
\begin{align}
\begin{split}\label{eq X0}
\mathrm{d} u^{\varepsilon}(t)=&\mathrm{i} \left[\Delta u^{\varepsilon}(t)-\lambda|u^{\varepsilon}(t)|^{2 \sigma} u^{\varepsilon}(t)\right] \mathrm{d} t+\int_{B}[\Phi^\varepsilon(z, u^{\varepsilon}(t-))-u^{\varepsilon}(t-)] \tilde{N}^{\varepsilon^{-1}}(\mathrm{d} t, \mathrm{d} z) \\
&+\varepsilon^{-1}\int_{B}\Big[\Phi^\varepsilon(z, u^\varepsilon(t))-u^\varepsilon(t)+\varepsilon\, \mathrm{i}  \sum_{j=1}^{m} z_{j} g_{j}(u^\varepsilon(t))\Big] \nu(\mathrm{d} z) \mathrm{d} t, \quad t\in[0,T],\\
 u^{\varepsilon}(0)=&u_0.
 \end{split}
\end{align}
In the following remark, we use a simple example to clarify the intuition and advantages of utilizing the Marcus canonical integrals.
\begin{remark}\label{remark-marcus-int}
The Marcus canonical integral $\diamond$ has several beneficial consequences and can be seen as extensions of Stratonovich
integrals for continuous case. Notably, it ensures the consistency with a Wong-Zakai-type approximation implying that the jump in the differential may be approximated sensibly by piecewise linear continuous functions. More precisely,  consider the following piecewise linear approximation of the L\'evy process $L(t)=(L_j(t))_{j=1}^m:=\int_0^t\int_B z\tilde{N}(\mathrm{d}z,\mathrm{d} s)$:
\begin{align*}
    L^n_j(t)=L_j(t_i^n)+\frac{ L_j(t_{i+1}^n)-L_j(t_{i}^n)    }{t_{i+1}^n- t_{i}^n}(t-t_{i}^n),\quad t_i^n\leq t\leq t_{i+1}^n
\end{align*}
where $0=t_0^n<t_1^n<\cdots<t_{N(n)}^n=T$ is a partition of the interval $[0,T]$. Then as the mesh $\max\limits_{1\leq j\leq N(n)}|t_{j}^n-t_{j-1}^n|\rightarrow0$, the solution of the following ordinary differential equation
\begin{align*}
   X^n(t)=X_0+\sum_{j=1}^m \int_0^t b_j(X^n(s))\mathrm{d} L_j^n(s)
\end{align*}
converges in probability to
\begin{align}
   X(t)&=X_0 + \int_0^t b(X(s-))\diamond \mathrm{d} L(s)\nonumber\\
   &=X_0 +\sum_{j=1}^m \int_0^t b_j(X(s-))\diamond \mathrm{d} L_j(s)\label{marcus-SDE-eq1}\\
   &:=X_0+ \int_0^t\int_B\big[ \phi(1,z,X(s-))-X(s-)\big] \tilde{N}(\mathrm{d} s, \mathrm{d} z )\nonumber\\
   &\ \ \ \ + \int_0^t\int_B\big[  \phi(1,z,X(s))-X(s) -\sum_{j=1}^mb_j(X(s))z_j \big]\nu(\mathrm{d} z)\mathrm{d} s,\nonumber
\end{align}
where $X_0\in\mathbb{R}^d$ is a deterministic initial, $b=(b_1,\cdots,b_m)$, $b_j\in C^1(\mathbb{R}^d;\mathbb{R}^d)$, $j=1,\cdots,m$ and $\phi:[0,1]\times \mathbb{R}^m \times \mathbb{R}^d  \ni (t,z,x) \rightarrow \phi(t,z,x)\in \mathbb{R}^d $ denotes the value at time $t$ of the solution of the following ordinary differential equation
\begin{align*}
\left\{\begin{array}{l}
y'(t)=\sum_{j=1}^m b_j(y(t))z_j, \\
y(0)=x.
\end{array}\right.
\end{align*}
We refer to \cite{Marcus-78} and \cite{KPP-95} for more details. 
The second benefit is the validity of the change of variables formula, that is, let $f\in C^2(\mathbb{R}^d;\mathbb{R}^d)$, then we have
\begin{align*}
    f(X(t))-f(X_0)=\int_0^t (f'(X(s-))b(X(s-))) \diamond\mathrm{d} L(s).
\end{align*}
Another benefit of the Marcus noise is its invariance under a change of coordinates, which is crucial for obtaining the conservation of the $L^2(\mathbb{R}^d)$-norm of solutions for stochastic NLS, see \cite[proof of Proposition 4.1]{BLZ}. Let $\psi:\mathbb{R}^d\rightarrow \mathbb{R}^d$ be a $C^1$-diffeomorphism. Define the change of coordinate
\begin{align}\label{change-of-coor}
y:=\psi(x)
\end{align}
with the inverse $x=\psi^{-1}(y)$. Under the change of coordinate \eqref{change-of-coor}, $X$ is a solution to  \eqref{marcus-SDE-eq1} if and only if the process $Y(t):=\psi(X(t))$ is a solution to
\begin{align*}
   \mathrm{d}Y(t)= \sum_{j=1}^m\hat{b}_j(Y(t-))\diamond \mathrm{d} L_j(t)
\end{align*}
where 
$$
      \hat{b}_j: \mathbb{R}^d \ni y \mapsto \psi^{\prime}(\psi^{-1}( y)) b_j(\psi^{-1} (y)) \in \mathbb{R}^d.
$$
This invariance result was first stated and proved in \cite{Marcus-78} and then generalized in \cite{Brz+Man-19} to infinite dimensional case for Marcus canonical integrals.
\end{remark}

\begin{remark}
Consider the linear operator $A_1$ in $L^2(\mathbb{R}^d)$ defined by
\begin{equation*}
\left\{\begin{array}{l}
\mathcal{D}(A_1)=\left\{u \in H^{1}(\mathbb{R}^d), \Delta u \in L^2(\mathbb{R}^d)\right\}, \\
A_1 u=\mathrm{i} \Delta u, \quad \forall u \in \mathcal{D}(A_1),
\end{array}\right.
\end{equation*}
and another linear operator $A_2$ in $H^{-1}(\mathbb{R}^d)$ defined by
\begin{equation*}
\left\{\begin{array}{l}
\mathcal{D}(A_2)=H^{1}(\mathbb{R}^d), \\
A_2 u=\mathrm{i} \Delta u, \quad \forall u \in \mathcal{D}(A_2) .
\end{array}\right.
\end{equation*}
Let $G(A_1)$ and $G(A_2)$ be the graph of $A_1$ and $A_2$, respectively, that is,
$$
G(A_1)=\{(u,f)\in L^2(\mathbb{R}^d)\times L^2(\mathbb{R}^d): u\in \mathcal{D}(A_1)\text{ and }f=A_1 u\}
$$
and
$$
G(A_2)=\{(u,f)\in H^{-1}(\mathbb{R}^d)\times H^{-1}(\mathbb{R}^d): u\in \mathcal{D}(A_2)\text{ and }f=A_2 u\}.
$$
Note that both $A_1$ and $A_2$ are skew-adjoint and $G(A_1)\subset G(A_2)$. Let $(S^1_t)_{t \in\mathbb{R}}$ and  $(S^2_t)_{t\in\mathbb{R}}$ be the isometry groups generated by $A_1$ and $A_2$ respectively. It follows from \cite[Lemma 3.5.12]{Caz+Har} that $S^1_tx=S^2_tx$ for all $x\in L^2(\mathbb{R}^d)$, $t\in\mathbb{R}$. This allows us to denote the isometry groups generated by $A_1$ and $A_2$ both by $(S_t)_{t \in\mathbb{R}}$.

\end{remark}

Recall that a pair $(p, r) \in [2,\infty]^2$ is admissible if
$$
\frac{2}{p}=\frac{d}{2}-\frac{d}{r},\quad (p,r,d)\neq (2,\infty,2).
$$
With this definition, for different dimensions, we have 
\begin{align*}
\begin{cases}2 \leq r \leq \infty, & \text { if } \quad d=1 \\ 2 \leq r<\infty, & \text { if } \quad d=2 \\ 2 \leq r \leq \frac{2 d}{d-2}, & \text { if } \quad d \geq 3.\end{cases}
\end{align*}

By \cite[Theorem 1.1]{BLZ}, we have the following result.
\begin{proposition}\label{Prop solution}
  Let $p\geq 2$, $\sigma\in(0,\frac{2}{d})$, $r=2\sigma+2$ such that $(p,r)$ is an admissible pair, and $u_0\in L^2(\mathbb{R}^d)$. Under Assumption \ref{assm-g}, there exists a unique global mild solution $u^\varepsilon=(u^\varepsilon(t))_{t\in[0,T]}$ of (\ref{eq X0-1}), that is, $u^\varepsilon=(u^\varepsilon(t))_{t\in[0,T]}$ is an $\mathbb{F}$-adapted, c\`adl\`ag, $L^2(\mathbb{R}^d)$-valued process satisfying 
  $$
  u^\varepsilon\in L^p(\Omega;L^\infty(0,T;L^2(\mathbb{R}^d))\cap L^p(0,T;L^r(\mathbb{R}^d))),
  $$
  and for $t\in[0,T]$, 
  \begin{align}
\begin{split}
u^\varepsilon(t)=& S_{t} u_{0}-\mathrm{i} \int_{0}^{t } S_{t -s}(\lambda|u^\varepsilon(s)|^{2 \sigma} u^\varepsilon(s)) \mathrm{d} s+\int_{0}^{t} \int_{B}  S_{t-s}\left[\Phi^\varepsilon(z, u^\varepsilon(s- ))-u^\varepsilon(s- )\right] \tilde{N}^{\varepsilon^{-1}}(\mathrm{d} s, \mathrm{d} z)\nonumber\\
&+\varepsilon^{-1} \int_{0}^{t } \int_{B} S_{t -s}\Big[\Phi^\varepsilon(z, u^\varepsilon(s)))-u^{\varepsilon}(s)+\varepsilon\, \mathrm{i} \sum_{j=1}^{m} z_{j} g_{j}(u^\varepsilon(s))\Big] \nu(\mathrm{d} z) \mathrm{d} s,\quad\mathbb{P}\text{-a.s.}
 \end{split}
\end{align}
  Moreover,  for all $t\in[0,T]$,
  $$
  \|u^\varepsilon(t)\|_{L^2(\mathbb{R}^d)}=\|u_0\|_{L^2(\mathbb{R}^d)},\ \mathbb{P}\text{-a.s.}
  $$
\end{proposition}

For $\psi\in W $, let
 $Y^\psi=\{Y^\psi(t),t\in[0,T]\}$ denote the solution (see Section \ref{sec-SE} for details) of the following deterministic nonlinear Schr\"odinger equation  (the so called skeleton equation)
\begin{align}
\begin{split}\label{eq rate LDP-new}
\mathrm{d}Y^\psi(t)&=\mathrm{i} \Delta Y^\psi(t) \mathrm{d}t  -\mathrm{i} \lambda|Y^\psi (t)|^{2 \sigma} Y^\psi(t) \mathrm{d} t -\int_{B}\sum_{j=1}^{m}\mathrm{i}  z_{j} g_{j}(Y^\psi(t))(\psi(t,z)-1)\nu(\dif z)\dif t,\\
Y^\psi(0)&=u_0.
\end{split}
\end{align}

Let us also recall the definition of a rate function and LDP. Let $\mathcal{E}$ be a Polish space with the Borel $\sigma$-field $\mathcal{B}(\mathcal{E}).$

\begin{definition}[Rate function] A function $I: \mathcal{E} \rightarrow[0, \infty]$ is called a rate function on $\mathcal{E},$ if for each $M<\infty,$ the level set $\{x \in \mathcal{E}: I(x) \leq M\}$ is a compact subset of $\mathcal{E}$.
\end{definition}

\begin{definition}[Large deviation principle] Let I be a rate function on $\mathcal{E}$. A family $\left\{\mathbb{X}^{\varepsilon}\right\}_{\varepsilon>0}$ of $\mathcal{E}$-valued random elements is said to satisfy a LDP on $\mathcal{E}$ with the rate function $I$ if the following two claims hold.
\begin{itemize}
  \item[(a)]
 (Upper bound) For each closed subset $C$ of $\mathcal{E},$
$$
\limsup_{\varepsilon \rightarrow 0} \varepsilon \log {P}\left(\mathbb{X}^{\varepsilon} \in C\right) \leq - \inf_{x \in C} I(x).
$$
 \item[(b)] (Lower bound) For each open subset $O$ of $\mathcal{E}$
$$
\liminf_{\varepsilon \rightarrow 0} \varepsilon \log {P}\left(\mathbb{X}^{\varepsilon} \in O\right) \geq - \inf_{x \in O} I(x).
$$
\end{itemize}
\end{definition}

We now state the main result in this paper.
\begin{theorem}\label{Th ex 1}
Under Assumption \ref{assm-g}, the solutions $\{u^\varepsilon,~\e>0\}$ to (\ref{eq X0-1}) satisfy a LDP on $D([0,T];L^2(\mathbb{R}^d)\cap L^p(0,T;L^r(\mathbb{R}^d))$ with the good
rate function $I$ given by
\begin{eqnarray}\label{rate ex LDP}
I(g):=\inf\{Q(\psi):\psi \in W,~Y^\psi=g\},~g\in D([0,T];L^2(\mathbb{R}^d))\cap L^p(0,T;L^r(\mathbb{R}^d)),
\end{eqnarray}
where $Y^\psi$ is the unique solution of (\ref{eq rate LDP-new}).
Here we use the convention that the infimum of an empty set is $\infty$.
\end{theorem}

\vskip 0.3cm

%-------------------------
\section{Skeleton equation}\label{sec-SE}
In this section we will prove the existence and uniqueness of global mild solutions of the skeleton equation \eqref{eq rate LDP-new}.  First we define the notion of a mild solution used in the present work.

\begin{definition}Given any $u_0\in L^2(\mathbb{R}^d)$, a mild solution to  \eqref{eq rate LDP-new} is a function $Y^\psi=\{Y^\psi(t),t\in[0,T]\}\in C([0,T];L^2(\mathbb{R}^d))\cap L^{p}(0, T; L^{r}(\mathbb{R}^{d}))$  satisfying
\begin{align}\label{eq rate LDP 1}
Y^\psi(t)=&\, S_{t} u_{0}-\mathrm{i} \int_{0}^{t } S_{t -s}(\lambda|Y^\psi (s)|^{2 \sigma} Y^\psi(s)) \mathrm{d} s \nonumber\\
&-\int_0^t\!\!\int_{B}S_{t -s}\Big[\sum_{j=1}^{m}\mathrm{i}  z_{j} g_{j}(Y^\psi(s))(\psi(s,z)-1)\Big]\nu(\dif z)\dif s,\ \ t\in[0,T].
\end{align}

\end{definition}

Now we state the following technical lemma for later use.
\begin{lemma}\label{lem-con-convo-1}
Let $\phi\in  L^1(0,T; H^{1}(\mathbb{R}^d))$. Then
\begin{align*}
     v(t):=\int_0^t S_{t-s}\phi(s) \dif s, \quad t\in[0,T],
\end{align*}
belongs to $C([0,T];H^{1}(\mathbb{R}^d))$ and
\begin{eqnarray}
     &&\|v\|_{L^\infty(0,T;H^1)}\leq \int_0^T\|\phi(t)\|_{H^1}\dif t;\label{lem-con-eq-1}\\
     &&\|v(t)\|^2_{L^2(\mathbb{R}^d)}=2\int_0^t \big(\phi(s),v(s)\big)_{L^2(\mathbb{R}^d)} \dif s, \quad t\in[0,T].\label{lem-con-eq-2}
\end{eqnarray}
Moreover, for any $0\leq s\leq t\leq T$,
\begin{align}\label{lemna-eq-v2}
     \|v(t)-v(s)\|^2_{L^2(\mathbb{R}^d)}
   &=
     -2\mathrm{i}\int_s^t\Big(\nabla v(l),\nabla\big(v(l)-v(s)\big)\Big)_{L^2(\mathbb{R}^d)}\dif l
      +
     2\int_s^t \big(\phi(l),v(l)-v(s)\big)_{L^2(\mathbb{R}^d)} \dif l\nonumber\\
   &=
     2\mathrm{i}\int_s^t\Big(\nabla v(l),\nabla v(s)\Big)_{L^2(\mathbb{R}^d)}\dif l
     +
     2 \int_s^t \big(\phi(l),v(l)-v(s)\big)_{L^2(\mathbb{R}^d)} \dif l.
\end{align}

\end{lemma}

\begin{proof}
 By \cite[Propositon 4.1.6]{Caz+Har}, if $\phi\in  C([0,T]; H^{1}(\mathbb{R}^d))$, then Lemma \ref{lem-con-convo-1} holds.
 Since $C([0,T]; H^{1}(\mathbb{R}^d))$ is dense in $L^1([0,T]; H^{1}(\mathbb{R}^d))$, by using a standard density arguments, Lemma \ref{lem-con-convo-1} holds in the general case.
\end{proof}

In what follows,  the letter $C$ stands for a positive constant which may depend on $A$ and whose exact value may change from line to line. We also use the notation $C_A$ if the dependence of the constant $C$ on $A$ should be highlighted.

Let $(p, r)$, $(\gamma, \rho)$ be admissible pairs {and let $\gamma',\rho'$ be conjugates of $\gamma,\rho$, respectively}. We recall the following deterministic Strichartz estimates from \cite{Caz} and the stochastic Strichartz estimates from \cite[Propositions 2.2 and 2.6]{BLZ}, which are crucial to the proof.
\begin{proposition}\label{prop-Stri-est}
\begin{enumerate}
\item[(1)] For every $\phi \in L^{2}\left(\mathbb{R}^{d}\right)$, the function $t \mapsto S_{t} \phi$ belongs to $L^{p}\left(\mathbb{R} ; L^{r}\left(\mathbb{R}^{d}\right)\right) \cap L^{\infty}\left(\mathbb{R} ; L^{2}\left(\mathbb{R}^{d}\right)\right)$ and there exists
a constant $C$ such that
\begin{align}\label{Stri-est-1}
\|S . \phi\|_{L^{p}\left(\mathbb{R} ; L^{r}\left(\mathbb{R}^{d}\right)\right)} \leq C\|\phi\|_{L^{2}\left(\mathbb{R}^{d}\right)}.
\end{align}
\item[(2)] Let $I$ be an interval of $\mathbb{R}$ and $J=\bar{I}$ with $0 \in J$. Then for every $f \in L^{\gamma^{\prime}}\left(I ; L^{\rho^{\prime}}\left(\mathbb{R}^{d}\right)\right)$, the function $t \mapsto \Phi_{f}(t)=\int_{0}^{t} S_{t-s} f(s) \mathrm{d} s$ belongs to $L^{p}\left(I ; L^{r}\left(\mathbb{R}^{d}\right)\right) \cap L^{\infty}\left(J ; L^{2}\left(\mathbb{R}^{d}\right)\right)$ and there exists a constant $C$
independent of $I$ such that
\begin{align}
&\left\|\Phi_{f}\right\|_{L^{\infty}\left(J ; L^{2}\left(\mathbb{R}^{d}\right)\right)} \leq C\|f\|_{L^{\gamma\prime}\left(I ; L^{\rho\prime}\left(\mathbb{R}^{d}\right)\right)} ; \label{Stri-est-2}\\
&\left\|\Phi_{f}\right\|_{L^{p}\left(I ; L^{r}\left(\mathbb{R}^{d}\right)\right)} \leq C\|f\|_{L^{\gamma \prime}\left(I ; L^{\rho{\prime}}\left(\mathbb{R}^{d}\right)\right) }.\label{Stri-est-3}
\end{align}
\item[(3)] Let $p,r\in[2,\infty)$. For all $q \geq 2$ and all $\mathbb{F}$-predictable process $\xi:[0, T] \times \Omega \times B \rightarrow L^{r}\left(\mathbb{R}^{d}\right)$ in
$L^{q}\left(\Omega ; L^{2}\left([0, T] \times B ; L^{2}\left(\mathbb{R}^{d}\right)\right) \cap L^{q}\left([0, T] \times B ; L^{2}\left(\mathbb{R}^{d}\right)\right)\right)$, we have
\begin{eqnarray}\label{sto-stri-est-1}
\mathbb{E}\left\|\int_{0}^{\cdot}\!\! \int_{B} S_{\cdot-s} \xi(s, z) \tilde{N}^1(\mathrm{d} s, \mathrm{d} z)\right\|_{L^{p}\left(0, T ; L^{r}\left(\mathbb{R}^{d}\right)\right)}^{q} \!\!\!\!\!\!\!\!
&\leq&\!\!\!\! C_q\;  \mathbb{E}\left(\int_{0}^{T}\!\!\!\! \int_{B}\!\|\xi(s, z)\|_{L^{2}\left(\mathbb{R}^{d}\right)}^{2} \nu(\mathrm{d} z) \mathrm{d} s\right)^{\frac{q}{2}} \nonumber\\
&&\!\!\!\!+ C_q\mathbb{E}\left(\int_{0}^{T}\!\!\!\! \int_{B}\!\|\xi(s, z)\|_{L^{2}\left(\mathbb{R}^{d}\right)}^{q} \nu(\mathrm{d} z) \mathrm{d} s\right).
\end{eqnarray}

\end{enumerate}
\end{proposition}

%----------------------

The subsequent arguments will repeatedly rely on the following crucial result.

\begin{lemma}\label{lemma-est-ass}
For every $N \in \mathbb{N}$, the following estimates hold
\begin{eqnarray}
&&\sup _{\hbar \in W^{N}}\int_0^T \int_{B}|z|^{2}(\hbar(t, z)+1) \nu(\dif z) \dif t<\infty, \label{eq lemma-est-ass 01}\\
&&\sup _{\hbar \in W^{N}}\int_0^T \int_{B}|z||\hbar (t, z)-1| \nu(\dif z) \dif t<\infty,\label{eq lemma-est-ass 02}\\
&&\lim_{\delta\rightarrow0}\sup _{\hbar \in W^{N}}\sup_{O\in\mathcal{B}([0,T]):{\rm Leb}_T(O)\leq \delta}\int_O \int_{B}|z||\hbar (t, z)-1| \nu(\dif z) \dif t=0,\label{eq lemma-est-ass 03}
\end{eqnarray}
\begin{eqnarray}\label{eq zhai 1}
\lim_{\delta\rightarrow0}\sup_{h\in W^N}\int_0^T\int_{0<|z|<\delta}|z||h(s,z)-1|\nu(dz)ds=0.
\end{eqnarray}
\end{lemma}

\begin{proof}
    Notice that $\int_B|z|^2\nu(\dif z)<\infty$, see \eqref{Assump on nu}, and for any $\delta>0$, $\mathcal{O}\in \mathcal{B}([0,T]\times B)$ with $\int_\mathcal{O}\nu(\dif z)\dif t<\infty$,
    $$
\int_\mathcal{O}e^{|z|^2}\nu(\dif z)\dif t\leq e\int_\mathcal{O}\nu(\dif z)\dif t<\infty,\ \ \ \ \int_\mathcal{O}e^{\delta|z|}\nu(\dif z)\dif t\leq e^\delta\int_\mathcal{O}\nu(\dif z)\dif t<\infty.
    $$
Applying \cite[Lemma 3.4]{Budhiraja-Chen-Dupuis} and \cite[Remark 2]{YZZ}, we yield \eqref{eq lemma-est-ass 01}, \eqref{eq lemma-est-ass 02} and \eqref{eq lemma-est-ass 03}. By the proof of \cite[Lemma 3.11]{Budhiraja-Chen-Dupuis} (cf. \cite[(3.23)]{Budhiraja-Chen-Dupuis}), we get \eqref{eq zhai 1}.

The proof of Lemma \ref{lemma-est-ass} is complete.
\end{proof}

\begin{remark}
 We remark that (\ref{eq lemma-est-ass 03}) is quoted from \cite[Remark 2]{YZZ}, which is a little more general than (3.5) of \cite[Lemma 3.4]{Budhiraja-Chen-Dupuis}, i.e.,
\begin{eqnarray}\label{Ineq SN}
\lim _{\delta \rightarrow 0} \sup _{\hbar \in W^{N}} \sup _{0\leq s\leq t\leq T: t-s \leq \delta} \int_s^t\int_B |z||\hbar(r, v)-1| \nu(\dif v) \dif r=0.
\end{eqnarray}
In this paper, (\ref{eq lemma-est-ass 03}) plays an important role to obtain the a priori estimates; see e.g. (\ref{eq LDP 5-05}).
\end{remark}

Given an admissible pair $(r, p)$ and $s$, $t$ such that $0 \leqslant s<t$, denote the space
$$
E^{(s, t, p)}:=L^{\infty}(s, t; L^2(\mathbb{R}^d)) \cap L^{p}(s, t; L^r(\mathbb{R}^d)).
$$
For the simplicity of notation, we write $E^{(T, p)}$ instead of $E^{(0, T, p)}$.
Before stating the main result in this subsection, we need the following result.
\vskip 0.3cm

%-----------------------------
\begin{proposition}\label{prop 5}
Let $r=2\sigma+2$ and $(p,r)$ be an admissible pair.  Under Assumption \ref{assm-g}, for any $\psi\in W$, there exists a unique mild solution $Y^\psi=\{Y^\psi(t),t\in[0,T]\}\in C([0,T];L^2(\mathbb{R}^d))\cap L^{p}(0, T; L^{r}(\mathbb{R}^{d}))$  to the skeleton equation \eqref{eq rate LDP-new}.
Moreover, $\|Y^{\psi}(t)\|_{L^2(\mathbb{R}^d)}=\|u_0\|_{L^2(\mathbb{R}^d)}$, for all $t\in[0,T]$, $\psi\in W$.

\end{proposition}

\begin{proof}
Without loss of generality, fix $N\in\mathbb{N}$ and let $\psi\in W^N$. We will divide the proof into four steps.

\textbf{Step 1 (Local existence of the solution)} For the existence part, we use a fixed point argument.
Set
\begin{align}
\Lambda_{{\widetilde{T}},M}=\left\{u \in E^{({\widetilde{T}}, p)} :\|u\|_{L^\infty(0, {\widetilde{T}}; L^{2}(\mathbb{R}^d))}+\|u\|_{L^{p}(0, {\widetilde{T}}; L^{r}(\mathbb{R}^d))} \leq M\right\},
\end{align}
where ${\widetilde{T}}\in(0,T]$ and $M>0$ will
be determined later.

For $Y\in \Lambda_{{\widetilde{T}},M}$, let $\Upsilon$ denote the mapping such that
\begin{align}\label{eq Up 01}
\Upsilon(Y,u_0,\psi)(t)=& S_{t} u_{0}-\mathrm{i} \int_{0}^{t } S_{t -s}(\lambda|Y (s)|^{2 \sigma} Y (s)) \mathrm{d} s \nonumber\\
&-{\rm{i}} \int_0^t\!\!\int_{B}S_{t -s}\Big[ \sum_{j=1}^{m} z_{j} g_{j}(Y(s) )(\psi(s,z)-1)\Big]\nu(\dif z)\dif s\nonumber\\
=& S_{t} u_{0}+I_1(Y )(t) +I_2(Y )(t),\ \ t\in[0,{\widetilde{T}}].
\end{align}
Let $p',r'$ be the conjugates of $p,r$. Recall that $E^{(\tilde{T}, p)}:=L^{\infty}(0, \tilde{T}; L^2(\mathbb{R}^d)) \cap L^{p}(0, \tilde{T}; L^r(\mathbb{R}^d))$. 
By applying Strichartz inequalities \eqref{Stri-est-2} and \eqref{Stri-est-3} with the admissible pair $(\gamma,\rho)=(p,r)$ and the H\"older inequality, we get
\begin{align}\label{eq Up 02}
\|I_1(Y )\|_{E^{({\widetilde{T}},p)}}&\leq C\|  \lambda|Y  |^{2 \sigma} Y   \|_{L^{p\prime}(0, {\widetilde{T}} ; L^{r^{\prime}}(\mathbb{R}^{d}))}\nonumber\\
&=C|\lambda|\left(\int_{0}^{{\widetilde{T}}}\left(\int_{\mathbb{R}^{d}}|Y (t, x)|^{r} \mathrm{~d} x\right)^{\frac{p^{\prime}}{r^{\prime}}} \mathrm{d} t\right)^{\frac{1}{p^{\prime}}} \nonumber\\
&\leq C|\lambda| {\widetilde{T}}^{\frac{p-r}{p-1} \frac{1}{p^{\prime}}}\left(\int_{0}^{{\widetilde{T}}}\left(\int_{\mathbb{R}^{d}}|Y (t, x)|^{r} \mathrm{~d} x\right)^{\frac{p}{r}} \mathrm{~d} t\right)^{\frac{1}{p} \cdot \frac{r}{r^{\prime}}} \nonumber\\
&= C|\lambda| {\widetilde{T}}^{1-\frac{d \sigma}{2}}\|Y \|_{L^{p}(0, {\widetilde{T}}; L^{r}(\mathbb{R}^{d}))}^{2 \sigma+1},
\end{align}
where we also used the fact that $r'(2\sigma+1)=r$ and $\frac2p=\frac{d}2-\frac{d}{r}$. 

Let $Y_1, Y_2 \in \Lambda_{{\widetilde{T}},M}$. 
It is easy to show that the map $F:L^{r}(\mathbb{R}^d)\ni u\rightarrow |u|^{2\sigma}u\in L^{r'}(\mathbb{R}^d)$
is continuously Fr\'echet differential with 
\begin{align*}
    F'(u)v=|u|^{2\sigma}v+2\sigma |u|^{2\sigma-2}u\mathrm{Re}(u\overline{v}),\quad u,v\in L^{r}(\mathbb{R}^d),
\end{align*}
and Lipschitz continuous on bounded sets
 \begin{align}\label{eq-est-nonlinear-1}
 \big{\|}\ |u_1|^{2\sigma}u_1 - |u_2|^{2\sigma}u_2\big{\|}_{L^{r'}(\mathbb{R}^d)}\leq C(\|u_1\|_{u^r(\mathbb{R}^d)}+\|u_2\|_{L^r(\mathbb{R}^d)})^{2\sigma}\|u_1-u_2\|_{L^r(\mathbb{R}^d)},\quad u_1,u_2\in L^r(\mathbb{R}^d).
 \end{align}
By using again Strichartz's estimates \eqref{Stri-est-2} and \eqref{Stri-est-3} with $(\gamma,\rho)=(p,r)$,  inequality \eqref{eq-est-nonlinear-1}, and the H\"older inequality, we obtain
\begin{align}\label{eq Up 03}
&\|I_1(Y _1)-I_1(Y _2)\|_{E^{({\widetilde{T}},p)}}\nonumber\\
&\leq C |  \lambda|\||Y _1 |^{2 \sigma} Y _1 - |Y _2 |^{2 \sigma} Y _2\|_{L^{p\prime}(0, {\widetilde{T}}; L^{r^{\prime}}(\mathbb{R}^{d}))}\nonumber\\
&\leq C |  \lambda|\big|(\|Y _1 \|_{L^r(\mathbb{R}^d)} +\|Y _2\|_{L^r(\mathbb{R}^d)})^{2 \sigma} \|Y _1-Y _2\|_{L^r(\mathbb{R}^d)}\big|_{L^{p\prime}(0, {\widetilde{T}})}\nonumber\\
&\leq C|  \lambda| {\widetilde{T}}^{1-\frac{d\sigma}2}(\|Y _{1}\|_{L^{p}(0, {\widetilde{T}}; L^{r}(\mathbb{R}^{d}))}+\|Y _{2}\|_{L^{p}(0, {\widetilde{T}}; L^{r}(\mathbb{R}^{d}))})^{2 \sigma}\|Y _{1}-Y _{2}\|_{L^{p}(0, {\widetilde{T}}; L^{r}(\mathbb{R}^{d}))}\nonumber\\
&\leq C|  \lambda| {\widetilde{T}}^{1-\frac{d\sigma}2}(2M)^{2\sigma}\|Y _{1}-Y _{2}\|_{L^{p}(0, {\widetilde{T}}; L^{r}(\mathbb{R}^{d}))}.
\end{align}
By Strichartz's estimate (\ref{Stri-est-1}), the Cauchy-Schwartz inequality, and Condition \eqref{cond-g-1}, we get
\begin{align}\label{eq Up 04}
\sup _{0 \leq t \leq {\widetilde{T}}}\big\|I_2(Y )(t)\big\|_{L^{2}\left(\mathbb{R}^{d}\right)} &=\sup _{t \in[0, {\widetilde{T}}]}\Big\| \int_{0}^{t} \int_{B}S_{t -s}\Big[{\rm{i}}\sum_{j=1}^{m}    z_{j} g_{j}(Y(s) )(\psi(s,z)-1)\Big] \nu(\mathrm{d} z) \mathrm{d} s\Big\|_{L^{2}\left(\mathbb{R}^{d}\right)} \nonumber\\
&\leq  C \int_0^{\widetilde{T}}\int_B \sum_{j=1}^{m}|z_j||\psi(s,z)-1|\| g_{j}(Y(s) ) \|_{L^2(\mathbb{R}^d)}\nu(\mathrm{d} z) \mathrm{d} s\nonumber\\
&\leq  CL_1m^{\frac12} \int_0^{\widetilde{T}}\int_B|z||\psi(s,z)-1|\|Y(s) \|_{L^2(\mathbb{R}^d)}\nu(\mathrm{d} z) \mathrm{d} s\nonumber\\
& \leq CL_1m^{\frac12}\|Y \|_{ E^{({\widetilde{T}},p)}}\sup _{\hbar \in W^{N}}\int_0^{\widetilde{T}} \int_{B}|z||\hbar (s, z)-1|\nu(\dif z) \dif s.
\end{align}
Similarly, using Strichartz's estimates (\ref{Stri-est-3}) with $(\gamma,\rho)=(p,r)$, the Cauchy-Schwartz inequality, and Condition \eqref{cond-g-1} gives
\begin{align}\label{eq Up 05}
\|I_2(Y )\|_{L^{p}(0, {\widetilde{T}} ; L^{r}(\mathbb{R}^{d}))}
&=\Big\|\int_{0}^{\cdot} S_{\cdot-s}\Big(\int_{B} \Big[\sum_{j=1}^{m} {\rm{i}} z_{j} g_{j}(Y(s) )(\psi(s,z)-1)\Big] \nu(\mathrm{d} z)\Big) \mathrm{d} s\Big\|_{L^{p}(0, {\widetilde{T}} ; L^{r}(\mathbb{R}^{d}))} \nonumber\\
& \leq C \int_{0}^{{\widetilde{T}}}\Big\|\int_{B} \Big[ \sum_{j=1}^{m} {\rm{i}} z_{j} g_{j}(Y(s) )(\psi(s,z)-1)\Big] \nu(\mathrm{d} z)\Big\|_{L^{2}(\mathbb{R}^{d})} \mathrm{d} s \nonumber\\
&\leq CL_1m^{\frac12}\|Y \|_{ E^{({\widetilde{T}},p)}}\sup _{\hbar \in W^{N}}\int_0^{\widetilde{T}} \int_{B}|z||\hbar (s, z)-1|\nu(\dif z) \dif s.
\end{align}
Using arguments similar to those that led up to \eqref{eq Up 04} and \eqref{eq Up 05}, we deduce that
\begin{align}\label{eq Up 06}
&\sup _{0 \leq t \leq {\widetilde{T}}}\big\|I_2(Y_1 )(t)-I_2(Y _2)(t)\big\|_{L^{2}(\mathbb{R}^{d})} \nonumber\\
&\leq  C \int_0^{\widetilde{T}}\int_B \sum_{j=1}^m|z_j||\psi(s,z)-1|\| g_{j}(Y _1(s))- g_{j}(Y _2(s))\|_{L^2(\mathbb{R}^d)}\nu(\mathrm{d} z) \mathrm{d} s\nonumber\\
& \leq CL_1m^{\frac12}\|Y _1-Y _2\|_{ E^{({\widetilde{T}},p)}}\sup _{\hbar \in W^{N}}\int_0^{\widetilde{T}} \int_{B}|z||\hbar (s, z)-1|\nu(\dif z) \dif s,
\end{align}
and
\begin{align}\label{eq Up 07}
&\|I_2(Y _1)-I_2(Y _2)\|_{L^{p}(0, {\widetilde{T}} ; L^{r}(\mathbb{R}^{d}))}\nonumber\\
 & \leq C \int_{0}^{{\widetilde{T}}}\Big\|\int_{B} \Big[ \sum_{j=1}^{m} {\rm{i}} z_{j}( g_{j}(Y _1(s))-g_j(Y _2(s)))(\psi(s,z)-1)\Big] \nu(\mathrm{d} z)\Big\|_{L^{2}(\mathbb{R}^{d})} \mathrm{d} s \nonumber\\
&\leq CL_1m^{\frac12}\|Y _1-Y _2\|_{ E^{({\widetilde{T}},p)}}\sup _{\hbar \in W^{N}}\int_0^{\widetilde{T}} \int_{B}|z||\hbar (s, z)-1|\nu(\dif z) \dif s.
\end{align}
Combining (\ref{eq Up 01}), (\ref{eq Up 02}), and (\ref{eq Up 03})--(\ref{eq Up 07}) gives
\begin{align}
\|\Upsilon(Y ,u_0,\psi)\|_{E^{({\widetilde{T}},p)}}\leq &C \|u_0\|_{L^2(\mathbb{R}^d)}+C|\lambda| {\widetilde{T}}^{1-\frac{d \sigma}{2}}\|Y \|_{E^{({\widetilde{T}},p)}}^{2 \sigma+1}\label{prop-proof-gamma-eq}\\
&+ 2CL_1m^{\frac12}\|Y \|_{ E^{({\widetilde{T}},p)}}\sup _{\hbar \in W^{N}}\int_0^{\widetilde{T}} \int_{B}|z||\hbar (s, z)-1|\nu(\dif z) \dif s,\nonumber
\end{align}
and
\begin{align}\label{prop-proof-gamma-eq 02}
&\|\Upsilon(Y _1,u_0,\psi)-\Upsilon(Y _2,u_0,\psi)\|_{E^{({\widetilde{T}},p)}}\nonumber\\
&\leq  C|  \lambda| {\widetilde{T}}^{1-\frac{d\sigma}2}(2M)^{2\sigma}\|Y _{1}-Y _{2}\|_{E^{({\widetilde{T}},p)}}\\
&\hspace{0.5cm}+2CL_1m^{\frac12}\|Y _1-Y _2\|_{ E^{({\widetilde{T}},p)}}\sup _{\hbar \in W^{N}}\int_0^{\widetilde{T}} \int_{B}|z||\hbar (s, z)-1|\nu(\dif z) \dif s.\nonumber
\end{align}

In the following of the proof, set $M=2 C \|u_0\|_{L^2(\mathbb{R}^d)}$ and let $\varepsilon =\min\{\frac{M}{4},\frac14\}$.
By Lemma \ref{lemma-est-ass}, there exists  $\delta_1>0$ such that
$$
\sup _{\hbar \in W^{N}} \sup _{0\leq s\leq t\leq T: t-s \leq \delta_1} \int_s^t\int_B |z||\hbar(r, v)-1| \nu(\dif v) \dif r<\frac1{2CL_1m^{\frac12}}\varepsilon.
$$
Now let us  choose ${\widetilde{T}}=T_0$, here
\begin{align}\label{T0-req-cond}
T_0:=\frac{1}{2}\min\Big\{1,\Big(\frac{1}{4(2M)^{2\sigma}C|\lambda|}\Big)^{1/(1-d\sigma/2)},\delta_1,T\Big\},
\end{align}
which implies that
\begin{align*}\label{T0-req-cond-01}
&C|\lambda|T_0^{1-\frac{d\sigma}2}M^{2\sigma+1}-\frac{M}4<0,\\
&C|  \lambda| T_0^{1-\frac{d\sigma}2}(2M)^{2\sigma}<\frac14, \\
&2CL_1m^{\frac12}\sup _{\hbar \in W^{N}}\int_0^{T_0} \int_{B}|z||\hbar (s, z)-1|\nu(\dif z) \dif s<\frac14.
\end{align*}
Then by \eqref{prop-proof-gamma-eq} we get, for all $Y\in \Lambda_{{T_0},M}$,
\begin{align}
\|\Upsilon(Y,u_0,\psi)\|_{E^{(T_0,p)}}\leq M.
\end{align}
Hence $\Upsilon(Y,u_0,\psi)\in \Lambda_{{T_0},M}$ for all $Y\in \Lambda_{{T_0},M}$. Furthermore, by (\ref{prop-proof-gamma-eq 02}), the mapping $\Upsilon(\cdot,u_0,\psi)$ is an $\frac{1}{2}$-contraction on $\Lambda_{{T_0},M}$, 
 that is, for any $Y_1,Y_2\in \Lambda_{{T_0},M}$,
 \begin{align*}
&\|\Upsilon(Y _1,u_0,\psi)-\Upsilon(Y _2,u_0,\psi)\|_{E^{(T_0,p)}}
\leq  
\frac{1}{2}\|Y _{1}-Y _{2}\|_{E^{(T_0,p)}}.
\end{align*}

The length of $T_0$ is determined only by $N$, $\sigma$, $M$, $\|u_0\|_{L^2(\mathbb{R}^d)}$, and $\lambda$. Thus $\Upsilon(\cdot,u_0,\psi)$ has a unique fixed point $Y_1^\psi\in \Lambda_{{T_0},M}$, denoted also by $\Gamma^{0}(u_{0}, \psi)$. And it is not difficult to see that $Y_1^\psi\in C([0,T_0];L^2(\mathbb{R}^d))$ and $Y_1^\psi$ is a solution to \eqref{eq rate LDP 1} on $[0,T_0]$.

\vskip 0.2cm
\textbf{Step 2 (Uniqueness of the solution)}
Let $Y_1,Y_2\in E^{(T,p)}\cap C([0,T];L^2(\mathbb{R}^d))$ be any two solutions satisfying \eqref{eq rate LDP 1} with the same initial datum $u_0$.  We define
$$t_{1}=\sup \left\{t \in[0, T] : Y_1(s)=Y_2(s)\text{ for all }s\in[0,t]\right\} .$$ If $t_{1}=T$, then the uniqueness follows. If $t_{1}<T$,  then we see that $\hat{Y}_{1}(\cdot)=Y_1(t_1+\cdot)$ and $\hat{Y}_{1}(\cdot)=Y_2(t_1+\cdot)$ are two solutions of \eqref{eq rate LDP 1} with $u_0$ replaced by $Y_1(t_1)=Y_2(t_1)$ on the interval $[0, T-t_1]$. Using arguments similar to proving (\ref{prop-proof-gamma-eq 02}), we have
\begin{align}
\begin{split}\label{prop-proof-contra-eq}
 \|\hat{Y}_1  &-\hat{Y}_2\|_{E^{(0,\tau,p)}}\\
   \leq  &\Big[C|  \lambda| \tau^{1-\frac{n\sigma}2}(\|\hat{Y}_{1}\|_{L^{p}([0, \tau]; L^{r}(\mathbb{R}^{d}))}+\|\hat{Y}_{2}\|_{L^{p}([0, \tau]; L^{r}(\mathbb{R}^{d}))})^{2 \sigma}\\
   &+2CL_1m^{\frac12}\sup _{\hbar \in W^{N}}\int_0^\tau \int_{B}|z||\hbar (t_1+s, z)-1|\nu(\dif z) \dif s\Big]\|\hat{Y}_{1}-\hat{Y}_{2}\|_{E^{(0,\tau,p)}}.
  \end{split}
\end{align}
Since (\ref{Ineq SN}) and $Y_1,Y_2\in E^{(T,p)}\cap C([0,T];L^2(\mathbb{R}^d))$, we can choose $\tau$ small enough such that  the coefficient of $\|\hat{Y}_{1}-\hat{Y}_{2}\|_{E^{(0,\tau,p)}}$ on the right hand side of \eqref{prop-proof-contra-eq} is smaller than $\frac12$, then
\begin{align*}
   \|\hat{Y}_1-\hat{Y}_2\|_{E^{(0,\tau,p)}}\leq\frac12\|\hat{Y}_{1}-\hat{Y}_{2}\|_{E^{(0,\tau,p)}}.
\end{align*}
This implies that $Y_1=Y_2$ on $[0,t_1+\tau]$ which contradicts the definition of $t_1$.
\vskip 0.2cm
\textbf{Step 3 (Conservation law)}
We first introduce the Yosida approximation operator and its properties. For the details, we refer the reader to  \cite[Section 1.5]{Caz}.

Fix $\mu>0$ and define an operator $J_{\mu}:H^{-1}(\mathbb{R}^d)\rightarrow H^{1}(\mathbb{R}^d)$ by $J_{\mu}=\mu(\mu I- \Delta)^{-1}$. Indeed for any $s\in\mathbb{R}$,  $J_{\mu}$ is a contraction of $H^s(\mathbb{R}^d)$ and $J_{\mu}\in \mathcal{L}(H^{s}(\mathbb{R}^{d}), H^{s+2}(\mathbb{R}^{d}))$.
Moreover, if $\mathbb{X}$ is either of the spaces $H^1(\mathbb{R}^d)$,  $H^{-1}(\mathbb{R}^d)$ or $L^p(\mathbb{R}^d)$ for $p\in(1,\infty)$, then
\begin{eqnarray}
&& \|J_{\mu}g\|_\mathbb{X}\leq \|g\|_\mathbb{X},\ \forall g\in\mathbb{X};\label{J mu 01}\\
&& _{\mathbb{X}}\langle J_{\mu}g,   h\rangle_{\mathbb{X}^*}=_{\mathbb{X}}\langle g,   J_{\mu}h\rangle_{\mathbb{X}^*},
       \ \forall g\in\mathbb{X}, h\in \mathbb{X}^*;\label{J mu 02}\\
&& \lim_{\mu\rightarrow\infty}\|J_{\mu}g-g\|_{\mathbb{X}}=0,\ \forall g\in\mathbb{X}.\label{J mu 03}
\end{eqnarray}
Here $\mathbb{X}^*$ is the dual space of $\mathbb{X}$. And
\begin{eqnarray}
\|J_{\mu}g\|_{L^1(\mathbb{R}^d)}\leq \|g\|_{L^1(\mathbb{R}^d)},\ \forall g\in L^1(\mathbb{R}^d).\label{J mu 04}
\end{eqnarray}

We also define the nonlinearity by
\begin{align*}
   &f_{\mu}(w):=J_{\mu}f(J_{\mu}w),\quad w\in L^2(\mathbb{R}^d),\\
   &h_{\mu}(w,z):=J_{\mu}h(J_{\mu}w,z)=\sum_{j=1}^mz_jJ_{\mu}g_j(J_{\mu}w),\quad w\in L^2(\mathbb{R}^d), z\in B,
\end{align*}
where $f(w)=\lambda  |w|^{2\sigma}w$ and  $h(w,z)=\sum_{j=1}^mz_jg_j(w).$
Observe that $J_{\mu}w\in H^1(\mathbb{R}^d)\subset L^2(\mathbb{R}^d)\cap L^r(\mathbb{R}^d)$. So $f(J_{\mu}w)\in L^{r'}(\mathbb{R}^d)$. Since $L^{r^{\prime}}(\mathbb{R}^d)  \hookrightarrow H^{-1}(\mathbb{R}^{d})$, we infer $f_{\mu}(w) \in H^{1}(\mathbb{R}^{d})$. Also we have $h_{\mu}(w,z) \in H^{1}(\mathbb{R}^{d})$.

Recall $T_0$ introduced in (\ref{T0-req-cond}). Applying (\ref{J mu 01}) and using arguments similar to that in the proofs of Step 1 and Step 2, there exists a unique solution
$Y^1_{\mu}\in \Lambda_{{T_0},M}$ of the following equation on $[0,T_0]$.
\begin{align}
\begin{split}\label{prop-proof-yosida-eq1}
&\mathrm{d} Y_{\mu}(t)={\rm{i}} \Delta Y_{\mu}(t)-{\rm{i}} f_{\mu}(Y_{\mu}(t)) \mathrm{d} t-\int_{B} {\rm{i}}h_{\mu}(Y_{\mu}(t),z)(\psi(t,z)-1) \nu(\mathrm{d} z) \mathrm{d} t, t\in[0,T], \\
&Y_{\mu}(0)= J_{\mu}u_0.
\end{split}
\end{align}
Indeed, for $Y\in \Lambda_{{\widetilde{T}},M}$, let $\Upsilon_\mu$ denote the mapping such that
\begin{align*}
\Upsilon_{\mu}(Y,u_0,\psi)(t)=& S_{t} J_{\mu}u_{0}-\mathrm{i} \int_{0}^{t } S_{t -s} f_{\mu}(Y(s)) \mathrm{d} s \nonumber\\
&-{\rm{i}} \int_0^t\!\!\int_{B}S_{t -s}h_{\mu}(Y(s),z)(\psi(s,z)-1)\nu(\dif z)\dif s,
\end{align*}
(\ref{prop-proof-gamma-eq}) and (\ref{prop-proof-gamma-eq 02}) hold with $\Upsilon$ replaced by $\Upsilon_{\mu}$.

Applying the properties of $J_\mu$ listed as above, it follows from
  Lemma \ref{lem-con-convo-1} and (\ref{eq lemma-est-ass 02}) that $Y^1_{\mu}\in C([0,T_0];H^1(\mathbb{R}^d))$. Taking the $H^{-1}-H^1$ duality product of equation \eqref{prop-proof-yosida-eq1} by $Y^1_{\mu}$ shows
\begin{align*}
   \frac12\frac{{\rm{d}}}{\rm{d }t}\|Y^1_{\mu}(t)\|_{L^2(\mathbb{R}^d)}^2
   &=\,_{H^{-1}}\langle{\rm{i}} \Delta Y^1_{\mu}(t),   Y^1_{\mu}(t)\rangle_{ H^{1}} - \,_{H^{-1}}\langle{\rm{i}} f_{\mu}\left(Y^1_{\mu}(t)\right),   Y^1_{\mu}(t)\rangle_{H^{1}}\\
    &\hspace{0.5cm}
    - \int_B\,_{H^{-1}}\langle {\rm{i}} h_{\mu}\left(Y^1_{\mu}(t),z\right),   Y^1_{\mu}(t)\rangle_{H^{1}}(\psi(t,z)-1)\nu(\mathrm{d}z)\\
   &=0.
\end{align*}
Hence
\begin{align*}
\|Y^1_{\mu}(t)\|_{L^2(\mathbb{R}^d)}^2=\|Y_{\mu}(0)\|_{L^2(\mathbb{R}^d)}^2=\|J_{\mu}u_0\|_{L^2(\mathbb{R}^d)}^2,\ t\in[0,T_0].
\end{align*}
Now consider the following equation
\begin{align}
\begin{split}\label{prop-proof-yosida-eq2}
&\mathrm{d} Y^2_{\mu}(t)={\rm{i}} \Delta Y^2_{\mu}(t)-{\rm{i}} f_{\mu}(Y^2_{\mu}(t)) \mathrm{d} t-\int_{B} {\rm{i}}h_{\mu}(Y^2_{\mu}(t),z)(\psi(T_0+t,z)-1) \nu(\mathrm{d} z) \mathrm{d} t, t\in[0,T_0], \\
&Y^2_{\mu}(0)=Y^1_{\mu}(T_0).
\end{split}
\end{align}
Using the same arguments as above, there exists a unique solution
$Y^2_{\mu}\in \Lambda_{{T_0},M}\cap C([0,T_0];H^1(\mathbb{R}^d))$ to (\ref{prop-proof-yosida-eq2}) and
\begin{align*}
\|Y^2_{\mu}(t)\|_{L^2(\mathbb{R}^d)}^2=\|Y^2_{\mu}(0)\|_{L^2(\mathbb{R}^d)}^2=\|Y^1_{\mu}(T_0)\|_{L^2(\mathbb{R}^d)}^2=\|J_{\mu}u_0\|_{L^2(\mathbb{R}^d)}^2,\ t\in[0,T_0].
\end{align*}
Following a recursive procedure we are able to prove that, for any $k=2,3,...$, there exists a unique solution
$Y^k_{\mu}\in \Lambda_{{T_0},M}\cap C([0,T_0];H^1(\mathbb{R}^d))$ to
\begin{align*}
\begin{split}
&\mathrm{d} Y^k_{\mu}(t)={\rm{i}} \Delta Y^k_{\mu}(t)-{\rm{i}} f_{\mu}(Y^k_{\mu}(t)) \mathrm{d} t-\int_{B} {\rm{i}}h_{\mu}(Y^k_{\mu}(t),z)(\psi((k-1)T_0+t,z)-1) \nu(\mathrm{d} z) \mathrm{d} t, t\in[0,T_0], \\
&Y^k_{\mu}(0)=Y^{k-1}_{\mu}(T_0),
\end{split}
\end{align*}
and
\begin{align*}
\|Y^k_{\mu}(t)\|_{L^2(\mathbb{R}^d)}^2=\|Y^k_{\mu}(0)\|_{L^2(\mathbb{R}^d)}^2=\|Y^{k-1}_{\mu}(T_0)\|_{L^2(\mathbb{R}^d)}^2=\|J_{\mu}u_0\|_{L^2(\mathbb{R}^d)}^2,\ t\in[0,T_0].
\end{align*}
For $j=1,2,3,...$, set
$$
Y^\psi_{\mu}(t)=Y^j_{\mu}(t-(j-1)T_0),\ t\in[(j-1)T_0,jT_0].
$$
Then it is easy to see that $Y^\psi_{\mu}=(Y^\psi_{\mu}(t))_{t\in[0,T]}\in E^{(T, p)}\cap C([0,T];H^1(\mathbb{R}^d))$ is the unique solution to (\ref{prop-proof-yosida-eq1}) on $[0,T]$, and
$$
\|Y^\psi_{\mu}(t)\|_{L^2(\mathbb{R}^d)}^2=\|J_{\mu}u_0\|_{L^2(\mathbb{R}^d)}^2,\ t\in[0,T].
$$
 Moreover,
\begin{eqnarray}\label{Eq Yosida 01}
\sup_{\psi\in W^N}\sup_{\mu>0}\|Y^\psi_{\mu}\|_{E^{(T, p)}}\leq ([\frac{T}{T_0}]+1)M.
\end{eqnarray}

Now let us prove the conservation law.

Assume that $Y^\psi=(Y^\psi(t))_{t\in[0,T]}\in E^{(T, p)}\cap C([0,T];L^2(\mathbb{R}^d))$ is the unique solution to (\ref{eq rate LDP 1}) on $[0,T]$.
The conservation law $\|Y^\psi(t)\|_{L^2(\mathbb{R}^d)}^2=\|u_0\|_{L^2(\mathbb{R}^d)}^2$ for $t\in[0,T]$ is straightforward once we prove
\begin{align}\label{eq-sec3-pro-100}
 \lim_{\mu\rightarrow\infty}\sup_{t\in[0,T]} \|Y_{\mu}^{\psi}(t)-Y^\psi(t)\|_{L^2}=0.
\end{align}
 To see this, we apply Strichartz's inequalities (\ref{Stri-est-2}) and (\ref{Stri-est-3}) to get, for $\widetilde{T}\in(0,T]$,
\begin{align*}
&\left\|  \int_{0}^{\cdot } S_{\cdot -s} f(Y^\psi(s))\mathrm{d} s - \int_{0}^{\cdot } S_{\cdot -s} f_\mu(Y_{\mu}^{\psi}(s))\mathrm{d} s\right\|_{E^{(\widetilde{T},p)}}
\\
&\leq C \| f(Y^\psi)-J_{\mu} f(J_{\mu}Y_{\mu}^{\psi}) \|_{L^{p\prime}(0, \widetilde{T}; L^{r^{\prime}}(\mathbb{R}^{d}))}\\
&\leq C \|J_{\mu} f(J_{\mu}Y_{\mu}^{\psi}) -J_{\mu} f(Y^\psi)\|_{L^{p\prime}(0, \widetilde{T} ; L^{r^{\prime}}(\mathbb{R}^{d}))}+C \|(J_{\mu}  -I)f(Y^\psi) \|_{L^{p\prime}(0, \widetilde{T} ; L^{r^{\prime}}(\mathbb{R}^{d}))}.
\end{align*}
By (\ref{J mu 01}) and using similar arguments as proving (\ref{eq Up 03}),
\begin{align*}
 &\|J_{\mu} f(J_{\mu}Y_{\mu}^{\psi}) -J_{\mu} f(Y^\psi)\|_{L^{p\prime}(0, \widetilde{T} ; L^{r^{\prime}}(\mathbb{R}^{d}))}\\
 &\leq  \| f(J_{\mu}Y_{\mu}^{\psi}) - f(Y^\psi)\|_{L^{p\prime}(0, \widetilde{T}; L^{r^{\prime}}(\mathbb{R}^{d}))}\\
 &\leq C|  \lambda| {\widetilde{T}}^{1-\frac{d\sigma}2}(\|Y_{\mu}^{\psi}\|_{L^{p}(0, {\widetilde{T}}; L^{r}(\mathbb{R}^{d}))}+\|Y^\psi\|_{L^{p}(0, {\widetilde{T}}; L^{r}(\mathbb{R}^{d}))})^{2 \sigma}\|J_{\mu}Y_{\mu}^{\psi}-Y^\psi\|_{L^p(0,{\widetilde{T}};L^r(\mathbb{R}^d))}\\
  &\leq C|  \lambda| {\widetilde{T}}^{1-\frac{d\sigma}2}(\|Y_{\mu}^{\psi}\|_{L^{p}(0, {\widetilde{T}}; L^{r}(\mathbb{R}^{d}))}+\|Y^\psi\|_{L^{p}(0, {\widetilde{T}}; L^{r}(\mathbb{R}^{d}))})^{2 \sigma}\\
  &\ \ \ \ \ \ \ \ \cdot\Big(\|J_{\mu}Y_{\mu}^{\psi}-J_{\mu}Y^\psi\|_{L^p(0,{\widetilde{T}};L^r(\mathbb{R}^d))}+\|J_{\mu}Y^\psi-Y^\psi\|_{L^p(0,{\widetilde{T}};L^r(\mathbb{R}^d))}\Big)\\
&\leq C|  \lambda| {\widetilde{T}}^{1-\frac{d\sigma}2}(\|Y_{\mu}^{\psi}\|_{E^{({\widetilde{T}},p)}}+\|Y^\psi\|_{E^{({\widetilde{T}},p)}})^{2 \sigma}\Big(\|Y_{\mu}^{\psi}-Y^\psi\|_{E^{({\widetilde{T}},p)}}+ \|J_{\mu}Y^\psi-Y^\psi\|_{E^{({\widetilde{T}},p)}}\Big).
\end{align*}
Applying the Minkowski inequality and the Strichartz inequality \eqref{Stri-est-1}, we find
\begin{align}\label{eq LDP1 pri 0}
&\left\|  \int_{0}^{\cdot }\int_{B} S_{\cdot -s} h_{\mu}(Y_{\mu}^{\psi}(s),z)(\psi(s,z)-1)\nu(\mathrm{d} z) \mathrm{d} s -  \int_{0}^{\cdot }\int_{B} S_{\cdot -s} h(Y^\psi(s),z)(\psi(s,z)-1)\nu(\mathrm{d} z) \mathrm{d} s  \right\|_{L^p(0,{\widetilde{T}};L^r(\mathbb{R}^d))}\nonumber\\
&\leq \int_{0}^{{\widetilde{T}} }\int_{B} \|S_{\cdot -s} h_{\mu}(Y_{\mu}^{\psi}(s),z)(\psi(s,z)-1)- S_{\cdot -s} h(Y^\psi(s),z)(\psi(s,z)-1) \|_{L^p(0,{\widetilde{T}};L^r(\mathbb{R}^d))}\nu(\mathrm{d} z) \mathrm{d} s \nonumber\\
&\leq C \int_0^{\widetilde{T}}\int_B \| h_{\mu}(Y_{\mu}^{\psi}(s)) -h(Y^\psi(s)) \|_{L^2(\mathbb{R}^d)}|\psi(s,z)-1|\nu(\mathrm{d} z) \mathrm{d} s,
\end{align}
and
\begin{align}\label{eq LDP1 pri 1}
&\sup _{0 \leq t \leq {\widetilde{T}}}\left\|  \int_{0}^{t }\int_{B} S_{t -s} h_{\mu}(Y_{\mu}^{\psi}(s),z)(\psi(s,z)-1)\nu(\mathrm{d} z) \mathrm{d} s -  \int_{0}^{t }\int_{B} S_{t -s} h(Y^\psi(s),z)(\psi(s,z)-1)\nu(\mathrm{d} z) \mathrm{d} s  \right\|_{L^2(\mathbb{R}^d)}
\nonumber\\
&\leq C \int_0^{\widetilde{T}}\int_B \| h_{\mu}(Y_{\mu}^{\psi}(s),z) -h(Y^\psi(s),z) \|_{L^2(\mathbb{R}^d)}|\psi(s,z)-1|\nu(\mathrm{d} z) \mathrm{d} s.
\end{align}
Combining the above two estimates with (\ref{J mu 01}) gives
\begin{align}\label{eq LDP1 pri 2}
&\left\|  \int_{0}^{\cdot }\int_{B} S_{\cdot -s} h_{\mu}(Y_{\mu}^{\psi}(s),z)(\psi(s,z)-1)\nu(\mathrm{d} z) \mathrm{d} s -  \int_{0}^{\cdot }\int_{B} S_{\cdot -s} h(Y^\psi(s),z)(\psi(s,z)-1)\nu(\mathrm{d} z) \mathrm{d} s  \right\|_{E^{({\widetilde{T}},p)}}
\nonumber\\
&\leq2 C \int_0^{\widetilde{T}}\int_B \| h_{\mu}(Y_{\mu}^{\psi}(s),z) -h(Y^\psi(s),z) \|_{L^2(\mathbb{R}^d)}|\psi(s,z)-1|\nu(\mathrm{d} z) \mathrm{d} s  \nonumber\\
&\leq 2C \int_0^{\widetilde{T}}\int_B \|J_{\mu} h(J_{\mu}Y_{\mu}^{\psi}(s),z) -J_{\mu} h(Y^\psi(s),z) \|_{L^2(\mathbb{R}^d)}|\psi(s,z)-1|\nu(\mathrm{d} z) \mathrm{d} s  \nonumber\\
&\hspace{1cm}+2C \int_0^{\widetilde{T}}\int_B \|J_{\mu} h(Y^\psi(s),z) -h(Y^\psi(s),z)\|_{L^2(\mathbb{R}^d)}|\psi(s,z)-1|\nu(\mathrm{d} z) \mathrm{d} s\nonumber\\
&\leq 2C \int_0^{\widetilde{T}}\int_B \|h(J_{\mu}Y_{\mu}^{\psi}(s),z) -h(Y^\psi(s),z) \|_{L^2(\mathbb{R}^d)}|\psi(s,z)-1|\nu(\mathrm{d} z) \mathrm{d} s  \nonumber\\
&\hspace{1cm}+2C \int_0^{\widetilde{T}}\int_B \|(J_{\mu}-I) h(Y^\psi(s),z)\|_{L^2(\mathbb{R}^d)}|\psi(s,z)-1|\nu(\mathrm{d} z) \mathrm{d} s\nonumber\\
&\leq 2C \int_0^{\widetilde{T}}\int_B \sum_{j=1}^m|z_j|\| g_{j}(J_{\mu}Y_{\mu}^{\psi}(s))-g_{j}(Y^\psi(s)) \|_{L^2(\mathbb{R}^d)}|\psi(s,z)-1|\nu(\mathrm{d} z) \mathrm{d} s  \nonumber\\
&\hspace{1cm}+2C \int_0^{\widetilde{T}}\int_B \|(J_{\mu}-I) h(Y^\psi(s),z)\|_{L^2(\mathbb{R}^d)}|\psi(s,z)-1|\nu(\mathrm{d} z) \mathrm{d} s\nonumber\\
&\leq  2CL_1m^{\frac12} \int_0^{\widetilde{T}}\int_B |z||\psi(s,z)-1| \| J_{\mu}Y_{\mu}^{\psi}(s)-Y^\psi(s) \|_{L^2(\mathbb{R}^d)}\nu(\mathrm{d} z) \mathrm{d} s  \nonumber\\
&\hspace{1cm}+2C \int_0^{\widetilde{T}}\int_B \|(J_{\mu}-I) h(Y^\psi(s),z) \|_{L^2(\mathbb{R}^d)}|\psi(s,z)-1|\nu(\mathrm{d} z) \mathrm{d} s\nonumber\\
&\leq  2CL_1m^{\frac12} \int_0^{\widetilde{T}} \Big( \| J_{\mu}Y_{\mu}^{\psi}-J_{\mu}Y^\psi \|_{L^{\infty}(0,s;L^2(\mathbb{R}^d))}+ \| J_{\mu}Y^\psi-Y^\psi \|_{L^{\infty}(0,s;L^2(\mathbb{R}^d))}\Big) \int_B |z||\psi(s,z)-1|\nu(\mathrm{d} z) \mathrm{d} s  \nonumber\\
&\hspace{1cm}+2C \int_0^{\widetilde{T}}\int_B \|(J_{\mu}-I) h(Y^\psi(s),z)  \|_{L^2(\mathbb{R}^d)}|\psi(s,z)-1|\nu(\mathrm{d} z) \mathrm{d} s\nonumber\\
&\leq  2CL_1m^{\frac12} \int_0^{\widetilde{T}} \Big( \|Y_{\mu}^{\psi}-Y^\psi \|_{E^{(s,p)}}+ \| J_{\mu}Y^\psi-Y^\psi \|_{E^{(s,p)}}\Big) \int_B |z||\psi(s,z)-1|\nu(\mathrm{d} z) \mathrm{d} s  \nonumber\\
&\hspace{1cm}+2C \int_0^{\widetilde{T}}\int_B \|(J_{\mu}-I) h(Y^\psi(s),z)  \|_{L^2(\mathbb{R}^d)}|\psi(s,z)-1|\nu(\mathrm{d} z) \mathrm{d} s
.
\end{align}
Therefore, we obtain
\begin{align*}
  & \| Y_{\mu}^{\psi}-Y^\psi\|_{E^{({\widetilde{T}},p)}}\\
  &\leq
     C\|Y_{\mu}^{\psi}(0)-Y^\psi(0)\|_{L^2(\mathbb{R}^d)}\\
     &\hspace{0.5cm}+
  C|  \lambda| {\widetilde{T}}^{1-\frac{d\sigma}2}(\|Y_{\mu}^{\psi}\|_{E^{({\widetilde{T}},p)}}+\|Y^\psi\|_{E^{({\widetilde{T}},p)}})^{2 \sigma}\Big(\|Y_{\mu}^{\psi}-Y^\psi\|_{E^{({\widetilde{T}},p)}}+ \|J_{\mu}Y^\psi-Y^\psi\|_{E^{({\widetilde{T}},p)}}\Big)\\
     &\hspace{0.5cm}+C \|(J_{\mu}  -I)f(Y^\psi) \|_{L^{p\prime}([0, {\widetilde{T}}] ; L^{r^{\prime}}(\mathbb{R}^{d}))}\\
     &\hspace{0.5cm}+2CL_1m^{\frac12}  \int_0^{\widetilde{T}} \Big( \| Y_{\mu}^{\psi}-Y^\psi \|_{E^{(s,p)}}+ \| J_{\mu}Y^\psi-Y^\psi \|_{E^{(s,p)}}\Big)\int_B |z||\psi(s,z)-1|\nu(\mathrm{d} z) \mathrm{d} s  \\
&\hspace{0.5cm}+2C \int_0^{\widetilde{T}}\int_B \|(J_{\mu}-I) h(Y^\psi(s),z)  \|_{L^2(\mathbb{R}^d)}|\psi(s,z)-1|\nu(\mathrm{d} z) \mathrm{d} s.
\end{align*}
It follows that
\begin{align}\label{eq Appro 01}
  &   \limsup_{\mu\rightarrow \infty}\| Y_{\mu}^{\psi}-Y^\psi\|_{E^{({\widetilde{T}},p)}}\nonumber\\
  &\leq
     \limsup_{\mu\rightarrow \infty}C\|Y_{\mu}^{\psi}(0)-Y^\psi(0)\|_{L^2(\mathbb{R}^d)}\nonumber\\
     &\hspace{0.5cm}+
   C|  \lambda| {\widetilde{T}}^{1-\frac{d\sigma}2}(\sup_{\hbar\in W^N}\sup_{\mu\in[1,\infty)}\|Y_{\mu}^{\hbar}\|_{E^{(T,p)}}+\|Y^\psi\|_{E^{({\widetilde{T}},p)}})^{2 \sigma} \limsup_{\mu\rightarrow \infty}\|Y_{\mu}^{\psi}-Y^\psi\|_{E^{({\widetilde{T}},p)}}\nonumber\\
     &\hspace{0.5cm}+2CL_1m^{\frac12} \int_0^{\widetilde{T}}  \limsup_{\mu\rightarrow \infty} \| Y_{\mu}^{\psi}-Y^\psi \|_{E^{(s,p)}} \int_B |z||\psi(s,z)-1|\nu(\mathrm{d} z) \mathrm{d} s .
\end{align}
Here (\ref{J mu 03}) and the dominated convergence theorem have been used. Indeed, (\ref{J mu 03}) implies that $\limsup_{\mu\rightarrow \infty}C\|Y_{\mu}^{\psi}(0)-Y^\psi(0)\|_{L^2(\mathbb{R}^d)}=0$.

By (\ref{Eq Yosida 01}), taking ${\widetilde{T}}=T_1$ small enough such that
\begin{eqnarray}\label{eq T1}
C|\lambda| T_1^{1-\frac{d\sigma}2}(\sup_{\hbar\in W^N}\sup_{\mu\in[1,\infty)}\|Y_{\mu}^{\hbar}\|_{E^{(T,p)}}+\|Y^\psi\|_{E^{(T,p)}})^{2 \sigma}<\frac12,
\end{eqnarray}
  it follows that
\begin{align*}
  & \limsup_{\mu\rightarrow \infty}\| Y_{\mu}^{\psi}-Y^\psi\|_{E^{(T_1,p)}}\\
    &\leq 4CL_1m^{\frac12} \int_0^{T_1}   \limsup_{\mu\rightarrow \infty} \| Y_{\mu}^{\psi}-Y^\psi \|_{E^{(s,p)}}   \int_B |z||\psi(s,z)-1|\nu(\mathrm{d} z) \mathrm{d} s.
\end{align*}
Since $\sup_{\hbar\in W^N}\int_0^{T_1}\int_B |z||\hbar(s,z)-1|\nu(\mathrm{d} z)\mathrm{d}t<\infty$ (see Lemma \ref{lemma-est-ass}), by Gronwall's inequality
\begin{align}\label{eq Appro 02}
\limsup _{\mu \rightarrow \infty} \left\|Y_{\mu}^{\psi}-Y^\psi\right\|_{E^{(T_1,p)}}
=0.
\end{align}
Using similar arguments as proving (\ref{eq Appro 01}) and by (\ref{eq T1}), we can get
\begin{align}\label{eq Appro 03}
  &   \limsup_{\mu\rightarrow \infty}\| Y_{\mu}^{\psi}-Y^\psi\|_{E^{(T_1,2T_1,p)}}\nonumber\\
  &\leq
     \limsup_{\mu\rightarrow \infty}2C\|Y_{\mu}^{\psi}(T_1)-Y^\psi(T_1)\|_{L^2(\mathbb{R}^d)}\nonumber\\
     &
     +4CL_1m^{\frac12} \int_{T_1}^{2T_1}  \limsup_{\mu\rightarrow \infty} \| Y_{\mu}^{\psi}-Y^\psi \|_{E^{(T_1,s,p)}} \int_B |z||\psi(s,z)-1|\nu(\mathrm{d} z) \mathrm{d} s .
\end{align}
Applying (\ref{eq Appro 02}) and Gronwall's inequality, we have
\begin{align*}
\limsup _{\mu \rightarrow \infty} \left\|Y_{\mu}^{\psi}-Y^\psi\right\|_{E^{(T_1,2T_1,p)}}
=0.
\end{align*}
Repeating this procedure, we can prove
\begin{align*}
\limsup _{\mu \rightarrow \infty} \left\|Y_{\mu}^{\psi}-Y^\psi\right\|_{E^{(T,p)}}=0,
\end{align*}
which verifies \eqref{eq-sec3-pro-100}.

\textbf{Step 4 (Global existence)} Recall $Y_1^\psi=\Upsilon^{0}(u_{0}, \psi)$ be the unique solution of \eqref{eq rate LDP 1} on $[0,T_0]$. By the $L^2$-norm conservation law, we see that $\|Y_1^\psi(t)\|_{L^2}$ is constant for all $t\in[0,T_0]$.  Using arguments similar to that in Step 3.1, we can construct a function
$Y^\psi=(Y^\psi(t))_{t\in[0,T]}\in E^{(T, p)}\cap C([0,T];L^2(\mathbb{R}^d))$, which is the unique solution to (\ref{eq rate LDP 1}) on $[0,T]$, and
$$
\|Y^\psi(t)\|_{L^2(\mathbb{R}^d)}^2=\|u_0\|_{L^2(\mathbb{R}^d)}^2,\ t\in[0,T].
$$
 Moreover,
\begin{eqnarray}\label{eq sko 01}
\sup_{\psi\in W^N}\|Y^\psi\|_{E^{(T, p)}}\leq ([\frac{T}{T_0}]+1)M.
\end{eqnarray}
The proof of this proposition is complete.

\end{proof}

From the proof of Proposition \ref{prop 5}, we have
\begin{corollary}\label{coro-uni-bou}
For $\psi\in W$ and $\mu>0$, let $Y^\psi$ and $Y^\psi_\mu$ be the unique solutions to (\ref{eq rate LDP 1}) and (\ref{prop-proof-yosida-eq1}), respectively. Then,

(1) For any $\psi\in W$,
\begin{eqnarray*}
\lim_{\mu\rightarrow\infty}\|Y^\psi-Y^\psi_\mu\|_{E^{(T,p)}}=0.
\end{eqnarray*}

(2) For any $N\in\mathbb{N}$, there exists a constant $C_{N}$ such that
\begin{eqnarray}\label{est uniform 1}
\sup_{\psi\in W^N}\|Y^\psi\|_{E^{(T,p)}}+\sup_{\psi\in W^N}\sup_{\mu>0}\|Y^\psi_\mu\|_{E^{(T,p)}}\leq C_{N}<\infty.
\end{eqnarray}

(3) For any $\psi\in W$, $\mu>0$ and $t\in[0,T]$, $\|Y^\psi_\mu(t)\|_{L^2(\mathbb{R}^d)}=\|J_\mu u_0\|_{L^2(\mathbb{R}^d)}$.

\end{corollary}

\section{Proof of the main result: Theorem \ref{Th ex 1}}\label{sec-proof-Th}

After the preparations in Section \ref{sec-SE} , we are ready to prove  the large deviation result in Theorem \ref{Th ex 1}. Proposition \ref{prop 5} allows us to define a map
\begin{eqnarray}\label{def G0}
\Gamma^0:W\ni \psi \mapsto Y^{\psi}\in C([0,T];L^2(\mathbb{R}^d))\cap L^{p}([0, T] ; L^{r}(\mathbb{R}^{d})),
\end{eqnarray}
here $Y^\psi$ is the unique solution of (\ref{eq rate LDP 1}).

By Proposition \ref{Prop solution} and the Yamada-Watanabe theorem, 
there exists a map
$$\Gamma^\varepsilon:M_{FC}\big([0,T]\times B\big)\rightarrow D([0,T];L^2(\mathbb{R}^d))\cap L^p(0,T;L^r(\mathbb{R}^d))$$
 such that if $\eta$ is a Poisson random measure on $[0,T]\times B$ with intensity $\text{Leb}_T \otimes{\eps}^{-1}\nu(dz)$, on a stochastic basis $(\Omega^1,\mathcal{F}^1,\mathbb{P}^1,\mathbb{F}^1)$ with $\mathbb{F}^1 =\{\mathcal{F}^1_t,t\in[0,T]\}$ satisfying the usual conditions, then the process $Y^{\varepsilon}$ defined by
$$
Y^{\varepsilon}:=\Gamma^\varepsilon(\eta)
$$
is the unique solution of following equation
 \begin{align}
\begin{split}\label{eq Y0}
\mathrm{d} Y^{\varepsilon}(t)=&\mathrm{i} \left[\Delta Y^{\varepsilon}(t)-\lambda|Y^{\varepsilon}(t)|^{2 \sigma} Y^{\varepsilon}(t)\right] \mathrm{d} t+\int_{B}[\Phi^\varepsilon(z, Y^{\varepsilon}(t-))-Y^{\varepsilon}(t-)] \Big(\eta(\mathrm{d} z, \mathrm{d} t)-{\eps}^{-1}\nu(\mathrm{d} z) \mathrm{d} t\Big) \\
&+\varepsilon^{-1}\int_{B}\Big[\Phi^\varepsilon(z, Y^\varepsilon(t))-Y^\varepsilon(t)+\varepsilon\, \mathrm{i}  \sum_{j=1}^{m} z_{j} g_{j}(Y^\varepsilon(t))\Big] \nu(\mathrm{d} z) \mathrm{d} t, \quad t\in[0,T],\\
 Y^{\varepsilon}(0)=&u_0,
 \end{split}
\end{align}
or in the mild form
\begin{align}
\begin{split}
Y^\varepsilon(t)=& S_{t} u_{0}-\mathrm{i} \int_{0}^{t } S_{t -s}(\lambda|Y^\varepsilon(s)|^{2 \sigma} Y^\varepsilon(s)) \mathrm{d} s\\
&+\int_{0}^{t} \int_{B}  S_{t-s}\left[\Phi^\varepsilon(z, Y^\varepsilon(s- ))-Y^\varepsilon(s- )\right] \Big(\eta(\mathrm{d} z, \mathrm{d} t)-{\eps}^{-1}\nu(\mathrm{d} z) \mathrm{d} t\Big)\nonumber\\
&+\varepsilon^{-1} \int_{0}^{t } \int_{B} S_{t -s}\Big[\Phi^\varepsilon(z, Y^\varepsilon(s)))-Y^{\varepsilon}(s)+\varepsilon\, \mathrm{i} \sum_{j=1}^{m} z_{j} g_{j}(Y^\varepsilon(s))\Big] \nu(\mathrm{d} z) \mathrm{d} s.
 \end{split}
\end{align}
Moreover, for all $t\in[0,T]$,
  $$
  \|Y^\varepsilon(t)\|_{L^2(\mathbb{R}^d)}=\|u_0\|_{L^2(\mathbb{R}^d)},\ \mathbb{P}^1\text{-a.s.}
  $$

 By the definition of $\Gamma^\varepsilon$, it is easy to see that
 $u^\varepsilon=\Gamma^\varepsilon(N^{\varepsilon^{-1}})$ is the unique solution to \eqref{eq X0-1}.

Now, for any fixed $m\in(0,\infty)$ and $\psi_\varepsilon\in \mathcal{W}^m$, by \cite[Theorem 6.1]{BPZ} (i.e. a Girsanov theorem for Poisson random measures),  set $\varphi_\varepsilon=1/\psi_\varepsilon$ we have
\begin{itemize}
\item[(S1)] The process $\mathcal{M}^{\eps}_t(\varphi_{\eps})$, $t\in[0,T]$, defined by
\begin{eqnarray*}
\label{eqn-M^eps}
&&\hspace{-1truecm}\lefteqn{ \mathcal{M}^{\eps}_t(\varphi_{\eps})=\exp\Big(
                     \int_{(0,t]\times{B}\times[0,{\eps}^{-1}\psi_{\eps}(s,z)]}\log(\varphi_{\eps}(s,z))N(ds,dz,dr)}
                     \\&&\ \ \ \ \ \ \ +
                     \int_{(0,t]\times{B}\times[0,{\eps}^{-1}\psi_{\eps}(s,z)]}\Big(-\varphi_{\eps}(s,z)+1\Big)\nu(dz)\,dsdr
                       \Big), \;\; t\in [0,T],
\end{eqnarray*}
is an $\mathbb{F}$-martingale on $(\Omega,\mathcal{F},\mathbb{F},P)$.
\item[(S2)] The formula
$$
\mathbb{P}^{\eps}_T({O})=\int_{O}\mathcal{M}^{\eps}_T(\varphi_{\eps})\,dP,\ \ \forall {O}\in \mathcal{F}
$$
defines a probability measure on $(\Omega,\mathcal{F})$.
\item[(S3)] The measures $P$ and $\mathbb{P}^{\eps}_T$ are equivalent.
\item[(S4)] The laws on $M_{FC}\big([0,T]\times B\big)$ of  the following two random variables are equal: (i) $N^{{\eps}^{-1}\psi_{\eps}}$ defined on probability space $(\Omega,\mathcal{F},\mathbb{F},\mathbb{P}^{\eps}_T)$
    and (ii) $N^{{\eps}^{-1}}$ defined on probability space $(\Omega,\mathcal{F},\mathbb{F},P)$.
 \end{itemize}

By the definition of $\Gamma^\varepsilon$ and (S4), the process
\begin{equation}\label{sol-control}
X^{\psi_\varepsilon}:=\Gamma^\varepsilon(N^{\varepsilon^{-1}\psi_\varepsilon})
\end{equation}
is the unique solution of (\ref{eq Y0}) with $(\Omega^1,\mathcal{F}^1,\mathbb{P}^1,\mathbb{F}^1,\eta)$ replaced by $(\Omega,\mathcal{F},\mathbb{F},\mathbb{P}^{\eps}_T,N^{{\eps}^{-1}\psi_{\eps}})$. Moreover,  for all $t\in[0,T]$,
  $$
  \|X^{\psi_\varepsilon}(t)\|_{L^2(\mathbb{R}^d)}=\|u_0\|_{L^2(\mathbb{R}^d)},\ \mathbb{P}^{\eps}_T\text{-a.s.}
  $$
Combining the above equality with (S3), we obtain  for all $t\in[0,T]$,
  \begin{eqnarray}\label{eq lem uniform 01}
\|X^{\psi_\varepsilon}(t)\|_{L^2(\mathbb{R}^d)}=\|u_0\|_{L^2(\mathbb{R}^d)},\ P\text{-a.s.}
  \end{eqnarray}
Moreover, (S1)--(S4) implies that
$X^{\psi_\varepsilon}$ is the unique solution of the integral equation on the probability space $(\Omega,\mathcal{F},\mathbb{F},P)$:
 \begin{align}
X^{\psi_\varepsilon}(t)&= S_{t} u_{0}-\mathrm{i} \int_{0}^{t } S_{t -s}(\lambda|X^{\psi_\varepsilon}(s)|^{2 \sigma} X^{\psi_\varepsilon}(s)) \mathrm{d} s\nonumber\\
&\hspace{0.5cm}+ \int_{0}^{t} \int_{B}  S_{t-s}\left[\Phi^\varepsilon(z, X^{\psi_\varepsilon}(s-))-X^{\psi_\varepsilon}(s- )\right] \Big(N^{\varepsilon^{-1}\psi_\varepsilon}(\mathrm{d} s, \mathrm{d} z)-\varepsilon^{-1}\nu(\dif z)\dif s\Big)\nonumber\\
&\hspace{0.5cm}+\varepsilon^{-1}\int_{0}^{t } \int_{B} S_{t -s}\Big[\Phi^\varepsilon(z, X^{\psi_\varepsilon}(s)))-X^{\psi_\varepsilon}(s)+\varepsilon\;\mathrm{i} \sum_{j=1}^{m} z_{j} g_{j}(X^{\psi_\varepsilon}(s))\Big] \nu(\mathrm{d} z) \mathrm{d} s\label{EQ4 LDP 2}\\
&= S_{t} u_{0}-\mathrm{i} \int_{0}^{t } S_{t -s}(\lambda|X^{\psi_\varepsilon}(s)|^{2 \sigma} X^{\psi_\varepsilon}(s)) \mathrm{d} s\nonumber\\
&\hspace{0.5cm}+ \int_{0}^{t} \int_{B}  S_{t-s}\left[\Phi^\varepsilon(z, X^{\psi_\varepsilon}(s-))-X^{\psi_\varepsilon}(s- )\right] \tilde{N}^{\varepsilon^{-1}\psi_\varepsilon}(\mathrm{d} s, \mathrm{d} z)\nonumber\\
&\hspace{0.5cm}+ \varepsilon^{-1}\int_{0}^{t} \int_{B}  S_{t-s}\left[\Phi^\varepsilon(z, X^{\psi_\varepsilon}(s))-X^{\psi_\varepsilon}(s )\right] \Big(\psi_\varepsilon(s,z)-1\Big)\nu(\dif z)\dif s\nonumber\\
&\hspace{0.5cm}+\varepsilon^{-1}\int_{0}^{t } \int_{B} S_{t -s}\Big[\Phi^\varepsilon(z, X^{\psi_\varepsilon}(s)))-X^{\psi_\varepsilon}(s)+\varepsilon\;\mathrm{i} \sum_{j=1}^{m} z_{j} g_{j}(X^{\psi_\varepsilon}(s))\Big] \nu(\mathrm{d} z) \mathrm{d} s\label{EQ4 LDP 2-01}\\
&= S_{t} u_{0}-\mathrm{i} \int_{0}^{t } S_{t -s}(\lambda|X^{\psi_\varepsilon}(s)|^{2 \sigma} X^{\psi_\varepsilon}(s)) \mathrm{d} s\nonumber\\
&\hspace{0.5cm}+ \int_{0}^{t} \int_{B}  S_{t-s}\left[\Phi^\varepsilon(z, X^{\psi_\varepsilon}(s-))-X^{\psi_\varepsilon}(s- )\right] \tilde{N}^{\varepsilon^{-1}\psi_\varepsilon}(\mathrm{d} s, \mathrm{d} z)\nonumber\\
&\hspace{0.5cm}+\varepsilon^{-1}\int_{0}^{t } \int_{B} S_{t -s}\Big[\Phi^\varepsilon(z, X^{\psi_\varepsilon}(s)))-X^{\psi_\varepsilon}(s)+\varepsilon\;\mathrm{i} \sum_{j=1}^{m} z_{j} g_{j}(X^{\psi_\varepsilon}(s))\Big]\psi_\varepsilon(s,z) \nu(\mathrm{d} z) \mathrm{d} s\nonumber\\
&\hspace{0.5cm}-
\int_{0}^{t} \int_{B}  S_{t-s} \mathrm{i} \sum_{j=1}^{m} z_{j} g_{j}(X^{\psi_\varepsilon}(s))\Big(\psi_\varepsilon(s,z)-1\Big)\nu(\dif z)\dif s.
\label{EQ4 LDP 2-02}
\end{align}
For the details of the proof of the above result, we refer \cite[Lemma 7.1]{BPZ}.

According to \cite[Theorem 4.4]{LSZ},  in which the conditions are an adaption of the original
conditions given in \cite{Budhiraja-Dupuis-Maroulas.} \cite{Budhiraja-Chen-Dupuis}, to complete the proof of the theorem,
it is sufficient to verify the following two Conditions \eqref{Prop-LDP1-eq} and \eqref{prop-LDP-cond-2}.

\bp\label{Yu-con}Let Assumption \ref{assm-g} hold.
For any given $N\in\mathbb{N}$, let $\psi_0,\psi_n \in W^N$, $n\in\mathbb{N}$ be such that
$\psi_n\rightarrow \psi_0$ in $W^N$ as $n\rightarrow\infty$. Then
\begin{align}\label{Prop-LDP1-eq}
\lim_{n\rightarrow\infty}\|{\Gamma}^0(\psi_n)-{\Gamma}^0(\psi_0)\|_{E^{(T,p)}}=0.
\end{align}

\ep

\begin{proposition}\label{lem LDP 2} Let Assumption \ref{assm-g} hold.
For any given $N\in\mathbb{N}$, let $\{\psi_\e,~\e>0\}\subset \mathcal{W}^N$. Then, for the solution $X^{\psi_\e}:=\Gamma^\varepsilon(\varepsilon N^{\varepsilon^{-1}\psi_\varepsilon})$ to \eqref{EQ4 LDP 2},
\begin{align}\label{prop-LDP-cond-2}
\lim_{\e\rightarrow0}P\left(\|X^{\psi_\e}-{\Gamma}^0(\psi_{\varepsilon})\|_{ E^{(T,p)}}\geq \delta\right)=0,\quad\quad\text{for any } \delta>0.
\end{align}

\end{proposition}
The proofs of Propositions \ref{Yu-con} and \ref{lem LDP 2} are given in Sections \ref{sect-prop-1}  and \ref{sect-prop-2} respectively.

\section{Proof of Proposition \ref{Yu-con}}\label{sect-prop-1}

Consider the following equation $n=0,1,2,\cdots$
\begin{align}\label{prop-eq}
Y^{n}(t)= S_{t} u_{0}-\mathrm{i} \int_{0}^{t } S_{t -s} f(Y^n(s))\mathrm{d} s -{\rm{i}}\int_0^t\!\!\int_{B}S_{t -s}h(Y^n(s),z)(\psi_n(s,z)-1) \nu(\dif z)\dif s,\ \ t\in[0,T],
\end{align}
where $f(Y)= \lambda|Y|^{2 \sigma} Y $ and $h(Y,z)=\sum_{j=1}^m z_{j} g_{j}(Y)$ as before.  We need to show
\begin{align}\label{eq LDP aim1}
\lim_{n\rightarrow\infty}\|Y^{n}-Y^{0}\|_{E^{(T,p)}}=0.
\end{align}
By (\ref{prop-eq}), for any $t\in[0,T]$, (in the following the initial data will be used)
\begin{align}\label{eq LDP 0}
&Y^{n}(t)-Y^{0}(t)\nonumber\\
&=S_{t}\Big(Y^{n}(0)-Y^{0}(0)\Big)
-\mathrm{i} \int_{0}^{t } S_{t -s} \Big(f(Y^n(s))-f(Y^0(s))\Big)\mathrm{d} s \nonumber\\
&\hspace{0.5cm}-{\rm{i}}\int_0^t\!\!\int_{B}S_{t -s}\Big(h(Y^n(s),z)(\psi_n(s,z)-1)-h(Y^0(s),z)(\psi_0(s,z)-1)\Big) \nu(\dif z)\dif s\nonumber\\
&=S_{t}\Big(Y^{n}(0)-Y^{0}(0)\Big)
-\mathrm{i} \int_{0}^{t } S_{t -s} \Big(f(Y^n(s))-f(Y^0(s))\Big)\mathrm{d} s \nonumber\\
&\hspace{0.5cm}-{\rm{i}}\int_0^t\!\!\int_{B}S_{t -s}\Big(h(Y^n(s),z)-h(Y^0(s),z)\Big)(\psi_n(s,z)-1) \nu(\dif z)\dif s\nonumber\\
&\hspace{0.5cm}-{\rm{i}}\int_0^t\!\!\int_{B}S_{t -s}h(Y^0(s),z)\Big((\psi_n(s,z)-1)-(\psi_0(s,z)-1)\Big) \nu(\dif z)\dif s.
\end{align}
By (\ref{est uniform 1}), using arguments similar to that used to obtain (\ref{eq Up 03}), (\ref{eq Up 06}) and (\ref{eq Up 07}), there exists constant  $C$ dependent on $N,\lambda,\sigma,\|u_0\|_{L^2(\mathbb{R}^d)},m,L_1$ but independent on $n$ and $\widetilde{T}$, such that for any $\widetilde{T}\in[0,T]$,
\begin{align}\label{eq LDP 1}
&\Big\|\mathrm{i} \int_{0}^{\cdot } S_{\cdot -s} \Big(f(Y^n(s))-f(Y^0(s))\Big)\mathrm{d} s\Big\|_{E^{\widetilde{T},p}}\nonumber\\
&+
\Big\|{\rm{i}}\int_{0}^{\cdot } S_{\cdot -s} \Big(h(Y^n(s),z)-h(Y^0(s),z)\Big)(\psi_n(s,z)-1) \nu(\dif z)\dif s\Big\|_{E^{\widetilde{T},p}}\nonumber\\
&\leq
C\|Y^n-Y^0\|_{E^{\widetilde{T},p}}
    \Big(\widetilde{T}^{1-\frac{d\sigma}{2}}+\sup _{\hbar \in W^{N}}\int_0^{\widetilde{T}} \int_{B}|z||\hbar (s, z)-1|\nu(\dif z) \dif s\Big).
\end{align}
Applying the Strichartz inequality (\ref{Stri-est-1}), (\ref{eq LDP 0}) and (\ref{eq LDP 1}), we arrive at
\begin{align}\label{eq LDP 2}
&\|Y^{n}-Y^{0}\|_{E^{(\widetilde{T},p)}}\nonumber\\
&\leq
C\|Y^{n}(0)-Y^{0}(0)\|_{L^2(\mathbb{R}^d)}
+
C\|Y^n-Y^0\|_{E^{\widetilde{T},p}}
    \Big(\widetilde{T}^{1-\frac{d\sigma}{2}}+\sup _{\hbar \in W^{N}}\int_0^{\widetilde{T}} \int_{B}|z||\hbar (s, z)-1|\nu(\dif z) \dif s\Big)\nonumber\\
    &\hspace{0.5cm}+
   \Big{ \|}{\rm{i}}\int_0^\cdot\!\!\int_{B}S_{\cdot -s}h(Y^0(s),z)\Big((\psi_n(s,z)-1)-(\psi_0(s,z)-1)\Big) \nu(\dif z)\dif s\Big{\|}_{E^{(\widetilde{T},p)}}.
\end{align}
Applying Lemma \ref{lemma-est-ass}, there exists $T_0\in(0,T]$ such that
\begin{eqnarray}\label{Choosing T0}
C\Big(T_0^{1-\frac{d\sigma}{2}}+\sup _{\hbar \in W^{N}}\sup_{l\in[0,T]}\int_{l \wedge T}^{(l+T_0) \wedge T} \int_{B}|z||\hbar (s, z)-1|\nu(\dif z) \dif s\Big)
\leq 1/2.
\end{eqnarray}
Therefore, the above two inequalities imply that
\begin{align}\label{eq LDP 2-1}
\|Y^{n}-Y^{0}\|_{E^{(T_0,p)}}
&\leq
C\|Y^{n}(0)-Y^{0}(0)\|_{L^2(\mathbb{R}^d)}\\
    &\hspace{0.5cm}+
    C\|\int_0^\cdot\!\!\int_{B}S_{\cdot -s}h(Y^0(s),z)\Big((\psi_n(s,z)-1)-(\psi_0(s,z)-1)\Big) \nu(\dif z)\dif s\|_{E^{(T_0,p)}}.\nonumber
\end{align}
Consider the following equation $n=0,1,2,\cdots$ and $\mu>0$,
\begin{align}\label{eq LDP 3}
\Phi^{n}(t)= \mathrm{i} \int_0^t\!\!\int_{B}S_{t -s}h(Y^0(s),z)(\psi_n(s,z)-1) \nu(\dif z)\dif s,\ \ t\in[0,T];
\end{align}
\begin{align}\label{eq LDP 4}
\Phi^{n}_{\mu}(t)= \mathrm{i} \int_0^t\!\!\int_{B}S_{t -s}h_{\mu}(Y^0(s),z)(\psi_n(s,z)-1) \nu(\dif z)\dif s,\ \ t\in[0,T],
\end{align}
where  $h_{\mu}(Y,z)= \sum_{j=1}^mz_{j}  J_{\mu}   g_{j}(J_{\mu}Y)=J_{\mu}h(J_\mu Y,z)$ as Step 3 in the proof of Proposition \ref{prop 5}.

We will prove that for any $\epsilon>0$, there exists $\mu_\epsilon>0$ such that
\begin{align}\label{eq LDP 5}
\sup_{n=0,1,2...}\|\Phi^{n}_{\mu_\epsilon}-\Phi^{n}\|_{E^{(T,p)}}\leq \epsilon.
\end{align}

Using arguments similar to that used to prove (\ref{eq LDP1 pri 0}), (\ref{eq LDP1 pri 1}) and (\ref{eq LDP1 pri 2}), we have
\begin{align}\label{eq LDP 5-01}
&\|\Phi^{n}_{\mu}-\Phi^{n}\|_{E^{(T,p)}}\nonumber\\
&\leq
C \int_0^{T}\int_B \| h_{\mu}(Y^0(s)) -h(Y^0(s)) \|_{L^2(\mathbb{R}^d)}|\psi_n(s,z)-1|\nu(\mathrm{d} z) \mathrm{d} s\nonumber\\
&\leq
C \int_0^{T}\int_B \| J_\mu h(J_\mu Y^0(s)) -J_\mu h(Y^0(s)) \|_{L^2(\mathbb{R}^d)}|\psi_n(s,z)-1|\nu(\mathrm{d} z) \mathrm{d} s\nonumber\\
&\hspace{0.5cm}+
C \int_0^{T}\int_B \| (J_\mu-I) h(Y^0(s))\|_{L^2(\mathbb{R}^d)}|\psi_n(s,z)-1|\nu(\mathrm{d} z) \mathrm{d} s\nonumber\\
&\leq
C_m \int_0^{T}\int_B \|(J_\mu-I) Y^0(s)\|_{L^2(\mathbb{R}^d)}|z||\psi_n(s,z)-1|\nu(\mathrm{d} z) \mathrm{d} s\nonumber\\
&\hspace{0.5cm}+
C_m\int_0^{T}\int_B \| (J_\mu-I) g_j(Y^0(s))\|_{L^2(\mathbb{R}^d)}|z||\psi_n(s,z)-1|\nu(\mathrm{d} z) \mathrm{d} s.
\end{align}

By (\ref{J mu 01}), (\ref{J mu 03}), (\ref{est uniform 1}), Assumption \ref{assm-g}, and the dominated convergence theorem, we have
\begin{eqnarray}
&&\lim_{\mu\rightarrow\infty}\int_0^{T}\|(J_\mu-I) Y^0(s)\|_{L^2(\mathbb{R}^d)}ds=0,\label{eq LDP 5-02}\\
&&\lim_{\mu\rightarrow\infty}\sum_{j=1}^m\int_0^T \| (J_\mu-I) g_j(Y^0(s))\|_{L^2(\mathbb{R}^d)}ds=0.\label{eq LDP 5-03}
\end{eqnarray}
For any $\kappa>0$, set
$$
O_{\mu,\kappa}:=\{s\in[0,T]:\|(J_\mu-I) Y^0(s)\|_{L^2(\mathbb{R}^d)}>\kappa\}\text{ and }O_{\mu,\kappa}^c=[0,T]\setminus O_{\mu,\kappa}.
$$
Then (\ref{eq LDP 5-02}) implies that
\begin{eqnarray}\label{eq LDP 5-04}
\lim_{\mu\rightarrow\infty}{\rm Leb}_T(O_{\mu,\kappa})=0.
\end{eqnarray}
Hence, by (\ref{est uniform 1}) again and Lemma \ref{lemma-est-ass}, for any $\kappa>0$,
\begin{align}\label{eq LDP 5-05}
&\limsup_{\mu\rightarrow\infty}\sup_{n=0,1,2...}\int_0^{T}\int_B \|(J_\mu-I) Y^0(s)\|_{L^2(\mathbb{R}^d)}|z||\psi_n(s,z)-1|\nu(\mathrm{d} z) \mathrm{d} s\nonumber\\
&=
\limsup_{\mu\rightarrow\infty}\sup_{n=0,1,2...}\int_{O_{\mu,\kappa}^c}\int_B \|(J_\mu-I) Y^0(s)\|_{L^2(\mathbb{R}^d)}|z||\psi_n(s,z)-1|\nu(\mathrm{d} z) \mathrm{d} s\nonumber\\
&\hspace{0.5cm}+
\limsup_{\mu\rightarrow\infty}\sup_{n=0,1,2...}\int_{O_{\mu,\kappa}}\int_B \|(J_\mu-I) Y^0(s)\|_{L^2(\mathbb{R}^d)}|z||\psi_n(s,z)-1|\nu(\mathrm{d} z) \mathrm{d} s\nonumber\\
&\leq
\kappa\sup_{\hbar\in W^N}\int_0^T\int_B |z||\hbar(s,z)-1|\nu(\mathrm{d} z) \mathrm{d} s\nonumber\\
&\hspace{0.5cm}+
\limsup_{\mu\rightarrow\infty}\Big(2\sup_{s\in[0,T]}\|Y^0(s)\|_{L^2(\mathbb{R}^d)}\sup_{\hbar\in W^N}\int_{O_{\mu,\kappa}}\int_B |z||\hbar(s,z)-1|\nu(\mathrm{d} z) \mathrm{d} s\Big)\nonumber\\
&\leq C_N\kappa.
\end{align}
By (\ref{eq LDP 5-03}), a similar
argument as above shows that
\begin{eqnarray}\label{eq LDP 5-06}
\limsup_{\mu\rightarrow\infty}\sup_{n=0,1,2...}\sum_{j=1}^m\int_0^{T}\int_B\!\!\! \| (J_\mu-I) g_j(Y^0(s))\|_{L^2(\mathbb{R}^d)}|z||\psi_n(s,z)-1|\nu(\mathrm{d} z) \mathrm{d} s
\leq
C_N\kappa.
\end{eqnarray}
Combining (\ref{eq LDP 5-01}), (\ref{eq LDP 5-05}) and (\ref{eq LDP 5-06}), there exists a constant $C_N$ such that, for any $\kappa>0$,
\begin{align*}%\label{eq LDP 5-07}
\limsup_{\mu\rightarrow\infty}\sup_{n=0,1,2...}\|\Phi^{n}_{\mu}-\Phi^{n}\|_{E^{(T,p)}}\leq C_N\kappa,
\end{align*}
which implies (\ref{eq LDP 5}).

By (\ref{eq LDP 2-1}) and (\ref{eq LDP 5}), for any $\epsilon>0$ and $n=1,2,3...$,
\begin{eqnarray}\label{eq LDP 2-2}
\|Y^{n}-Y^{0}\|_{E^{(T_0,p)}}
\leq
C\|Y^{n}(0)-Y^{0}(0)\|_{L^2(\mathbb{R}^d)}+C\epsilon
    +
    C\|\Phi^{n}_{\mu_\epsilon}-\Phi^{0}_{\mu_\epsilon}\|_{E^{(T,p)}}.
\end{eqnarray}
Therefore, since $Y^{n}(0)-Y^{0}(0)=0$, once we show (to be proved later in Lemma \ref{Lemma-M-n-appro}) that for any fixed $\mu>0$
\begin{eqnarray}\label{eq LDP 6}
\lim_{n\rightarrow\infty}\|\Phi^{n}_{\mu}-\Phi^{0}_{\mu}\|_{E^{(T,p)}}=0,
\end{eqnarray}
then by the arbitrariness of $\epsilon$, we conclude
\begin{eqnarray}\label{eq LDP 2-3}
\lim_{n\rightarrow\infty}\|Y^{n}-Y^{0}\|_{E^{(T_0,p)}}=0.
\end{eqnarray}
Now consider $t\in[T_0,2T_0\wedge T]$,  we have
\begin{align*}
&Y^{n}(t)-Y^{0}(t)\nonumber\\
&=S_{t}\Big(Y^{n}(T_0)-Y^{0}(T_0)\Big)
-\mathrm{i} \int_{T_0}^{t } S_{t -s} \Big(f(Y^n(s))-f(Y^0(s))\Big)\mathrm{d} s \nonumber\\
&\hspace{0.5cm}-{\rm{i}}\int_{T_0}^t\!\!\int_{B}S_{t -s}\Big(h(Y^n(s),z)(\psi_n(s,z)-1)-h(Y^0(s),z)(\psi_0(s,z)-1)\Big) \nu(\dif z)\dif s.
\end{align*}
By the definition of $T_0$(see (\ref{Choosing T0})), using arguments similar to getting (\ref{eq LDP 2-2}), we can get,
for any $\epsilon>0$ and $n=1,2,3...$,
\begin{align*}
\|Y^{n}-Y^{0}\|_{E^{(T_0,2T_0\wedge T,p)}}
&\leq
C\|Y^{n}(T_0)-Y^{0}(T_0)\|_{L^2(\mathbb{R}^d)}+C\epsilon
    +
    C\|\Phi^{n}_{\mu_\epsilon}-\Phi^{0}_{\mu_\epsilon}\|_{E^{(T_0,2T_0\wedge T,p)}}\nonumber\\
&\leq
\|Y^{n}-Y^{0}\|_{E^{(T_0,p)}}
+
C\epsilon
+
\|\Phi^{n}_{\mu_\epsilon}-\Phi^{0}_{\mu_\epsilon}\|_{E^{(T,p)}}.
\end{align*}
Combining the above inequality with (\ref{eq LDP 6}), (\ref{eq LDP 2-3}) and the arbitrariness of $\epsilon$, we arrive at
\begin{eqnarray*}
\lim_{n\rightarrow\infty}\|Y^{n}-Y^{0}\|_{E^{(2T_0 \wedge T,p)}}=0.
\end{eqnarray*}
Then following a standard recursive procedure, we obtain
\begin{eqnarray*}
\lim_{n\rightarrow\infty}\|Y^{n}-Y^{0}\|_{E^{( T,p)}}=0.
\end{eqnarray*}
which verifies \eqref{Prop-LDP1-eq}.

To complete the proof of Proposition \ref{Yu-con}, in the rest of this subsection, we give the proof of \eqref{eq LDP 6} as we promised before.
 To do this, we present first a priori  results.

\begin{lemma}\label{lem-est-M-n-mu}
  Fix $\mu>0$. we have
  \begin{itemize}
    \item[(1)] For any $n=0,1,2,...$, $\Phi^{n}_{\mu}\in C([0,T];H^1(\mathbb{R}^d))$.

    \item[(2)]Since $\nabla$ and $S_t$ commute, for any $n=0,1,2,...$, set
  \begin{eqnarray*}
 \nabla \Phi_{\mu}^n(t)={\rm{i}} \int_0^t\!\!\int_{B}S_{t -s} \nabla  \big(\sum_{j=1}^m z_{j} J_{\mu} (g_{j}(J_{\mu}Y^{0}(s)))(\psi_n(s,z)-1\big) \nu(\dif z)\dif s,\ t\in[0,T],
\end{eqnarray*}
then $\nabla \Phi^{n}_{\mu}\in C([0,T];H^1(\mathbb{R}^d))$.

    \item[(3)] There exists a constant $C_{\mu,N}$ such that
    \begin{eqnarray*}
    \sup_{n=0,1,2,...}\Big(\sup_{t\in[0,T]}\|\Phi_{\mu}^n(t)\|_{H^1(\mathbb{R}^d)}+\sup_{t\in[0,T]}\|\nabla \Phi_{\mu}^n(t)\|_{H^1(\mathbb{R}^d)}\Big)\leq C_{\mu,N}<\infty.
    \end{eqnarray*}

    \item[(4)] Let $0=t_0^k<t_1^k<\cdots<t_{2^k}^k=T$, where $t_i^k=\frac{iT}{ 2^{k}}$, be a partition of $[0,T]$ and let $\bar{s}_{k}(s)=t_i^k$ for $s \in\left[t_i^k,t_{i+1}^{k}\right)$, $i=0,\cdots,2^k-1$. Then
            \begin{eqnarray}
 \lim_{k\rightarrow\infty}\sup_{n=0,1,2,...}
       \Big(
          \int_0^T \|\Phi_{\mu}^n(s)-\Phi_{\mu}^n(\bar{s}_k)\|^2_{L^2(\mathbb{R}^d)} \dif s
            +
          \int_0^T \|\nabla\Phi_{\mu}^n(s)-\nabla\Phi_{\mu}^n(\bar{s}_k)\|^2_{L^2(\mathbb{R}^d)} \dif s
       \Big)
          =0.\label{eq-M_n-2}
  \end{eqnarray}
  \end{itemize}
\end{lemma}

\begin{proof}[Proof of Lemma \ref{lem-est-M-n-mu}]

  For $n=0,1,2,...$, set
\begin{eqnarray*}
\phi_{\mu}^n(s)=\int_{B}\big({\rm{i}}\sum_{j=1}^m z_{j}J_{\mu}  (g_{j}(J_{\mu}Y^{0}(s)))(\psi_n(s,z)-1\big) \nu(\dif z),\ s\in[0,T].
\end{eqnarray*}
Notice that, see e.g.,  \cite[Proposition 1.5.3]{Caz},
\begin{eqnarray}\label{eq J-mu 00}
\|J_{\mu}\|_{\mathcal{L}(L^2;H^2)}\leq C_\mu<\infty.
\end{eqnarray}
 By (\ref{eq lemma-est-ass 02}), (\ref{J mu 01}) and (\ref{est uniform 1}), using arguments similar to get (\ref{eq Up 04}), we deduce
\begin{align}\label{eq J-mu 01}
\sup_{n=0,1,2,...}\int_0^T\|\nabla \phi_{\mu}^n(s)\|_{H^1}ds
&\leq
\sup_{n=0,1,2,...}\int_0^T\|\phi_{\mu}^n(s)\|_{H^2}ds\nonumber\\
&\leq
\sup_{n=0,1,2,...}C_{\mu}\int_0^T\int_{B}\sum_{j=1}^m |z_{j}|\|g_{j}(J_{\mu}Y^{0}(s))\|_{L^2}|\psi_n(s,z)-1| \nu(\dif z)ds\nonumber\\
&\leq
 C_{m,\mu}\|Y^0\|_{L^\infty(0,T;L^2(\mathbb{R}^d))}\sup _{\hbar \in W^{N}}\int_0^{T} \int_{B}|z||\hbar (s, z)-1|\nu(\dif z) \dif s\nonumber\\
&\leq
C_{m,\mu,N}<\infty,
\end{align}
which implies that
\begin{eqnarray*}\label{eq Phi mu 00}
\text{for }n=0,1,2,...,\ \phi_{\mu}^n\in L^1(0,T;H^2(\mathbb{R}^d))\text{ and }\nabla \phi_{\mu}^n\in L^1(0,T;H^1(\mathbb{R}^d)).
\end{eqnarray*}
Applying Lemma \ref{lem-con-convo-1} leads us to statements (1) and (2) of this Lemma. Furthermore, in view of \eqref{lem-con-eq-2}, we have
  \begin{eqnarray}\label{lem-est-Phi-n-proof-1}
    \sup_{n=0,1,2,...}\Big(\sup_{t\in[0,T]}\|\Phi_{\mu}^n(t)\|_{H^1(\mathbb{R}^d)}+\sup_{t\in[0,T]}\|\nabla \Phi_{\mu}^n(t)\|_{H^1(\mathbb{R}^d)}\Big)
     \leq C_{m,\mu,N}<\infty,
    \end{eqnarray}
    which yields statement (3) of this Lemma.

On account of \eqref{lemna-eq-v2},  for any $0\leq s\leq t\leq T$, we have
\begin{align}\label{eq Phi time 01}
     &\|\Phi_{\mu}^n(t)-\Phi_{\mu}^n(s)\|^2_{L^2(\mathbb{R}^d)}\\
   &=
     2\mathrm{i}\int_s^t\Big(\nabla \Phi_{\mu}^n(l),\nabla \Phi_{\mu}^n(s)\Big)_{L^2(\mathbb{R}^d)}\dif l
     +
     2 \int_s^t \big(\phi_{\mu}^n(l),\Phi_{\mu}^n(l)-\Phi_{\mu}^n(s)\big)_{L^2(\mathbb{R}^d)} \dif l,\nonumber
\end{align}
and
\begin{align}\label{eq Phi time 02}
     &\|\nabla \Phi_{\mu}^n(t)-\nabla \Phi_{\mu}^n(s)\|^2_{L^2(\mathbb{R}^d)}\\
   &=
     2\mathrm{i}\int_s^t\Big(\nabla \big(\nabla \Phi_{\mu}^n(l)\big),\nabla \big(\nabla \Phi_{\mu}^n(s)\big)\Big)_{L^2(\mathbb{R}^d)}\dif l
     +
     2 \int_s^t \Big(\nabla\phi_{\mu}^n(l),\nabla \Phi_{\mu}^n(l)-\nabla \Phi_{\mu}^n(s)\Big)_{L^2(\mathbb{R}^d)} \dif l.\nonumber
\end{align}
Using \eqref{lem-est-Phi-n-proof-1},  with very similar arguments as the one in (\ref{eq J-mu 01}) we deduce
\begin{align*}
     &\|\Phi_{\mu}^n(t)-\Phi_{\mu}^n(s)\|^2_{L^2(\mathbb{R}^d)}
     +
     \|\nabla \Phi_{\mu}^n(t)-\nabla \Phi_{\mu}^n(s)\|^2_{L^2(\mathbb{R}^d)}\\
   &\leq
     C_{m,\mu,N}(t-s)
     +
     C_{m,\mu,N} \sup _{\hbar \in W^{N}}\int_s^{t} \int_{B}|z||\hbar (l, z)-1|\nu(\dif z) \dif l,
\end{align*}
and statement (4) of this lemma follows immediately from (\ref{eq lemma-est-ass 03}) (or (\ref{Ineq SN})).

\end{proof}

The following result plays a crucial role in proving \eqref{eq LDP 6}.
\begin{lemma}\label{lemma-D-U}
For every fixed $\mu>0$,
\begin{eqnarray*}
\lim_{n\rightarrow\infty}\sup_{t\in[0,T]}\|\Phi^{n}_{\mu}(t)-\Phi^{0}_{\mu}(t)\|_{H^1(\mathbb{R}^d)}=0.
\end{eqnarray*}

\end{lemma}
\begin{proof}[Proof of Lemma \ref{lemma-D-U}]
Now fix $\mu>0$.
It is enough to prove that
\begin{eqnarray}
\lim_{n\rightarrow\infty}\sup_{t\in[0,T]}\|\nabla\Phi^{n}_{\mu}(t)-\nabla\Phi^{0}_{\mu}(t)\|_{L^2(\mathbb{R}^d)}=0,\label{eq H1 01}
\end{eqnarray}
as the following convergnce
\begin{eqnarray}
\lim_{n\rightarrow\infty}\sup_{t\in[0,T]}\|\Phi^{n}_{\mu}(t)-\Phi^{0}_{\mu}(t)\|_{L^2(\mathbb{R}^d)}=0\label{eq H1 02}
\end{eqnarray}
follows similarly.  Clearly, combining (\ref{eq H1 01}) and (\ref{eq H1 02}) implies this lemma.

\vskip 0.2cm
Now let us prove (\ref{eq H1 01}). By Lemmas \ref{lem-con-convo-1}, we have %and \ref{lem-est-M-n-mu},
\begin{align}\label{eq App 01}
    &\|\nabla\Phi^{n}_{\mu}(t)-\nabla\Phi^{0}_{\mu}(t)\|^2_{L^2(\mathbb{R}^d)}\nonumber\\
&=
   \int_0^t\int_{B} \Big(  \sum_{j=1}^m{\rm{i}} z_{j} \nabla J_{\mu} (g_{j}(J_{\mu}Y^{0}(s))(\psi_n(s,z)-\psi_0(s,z)), \nabla\Phi^{n}_{\mu}(s)-\nabla\Phi^{0}_{\mu}(s)\Big)_{L^2(\mathbb{R}^d)}\nu(\mathrm{d} z) \dif s\nonumber\\
&=
   \int_0^t\int_{0<|z|<\delta} \Big(  \sum_{j=1}^m{\rm{i}} z_{j} \nabla J_{\mu} (g_{j}(J_{\mu}Y^{0}(s))(\psi_n(s,z)-\psi_0(s,z)), \nabla\Phi^{n}_{\mu}(s)-\nabla\Phi^{0}_{\mu}(s)\Big)_{L^2(\mathbb{R}^d)}\nu(\mathrm{d} z) \dif s\nonumber\\
&\hspace{0.5cm}+
   \int_0^t\int_{\delta\leq |z|\leq 1} \Big(  \sum_{j=1}^m{\rm{i}} z_{j} \nabla J_{\mu} (g_{j}(J_{\mu}Y^{0}(s))(\psi_n(s,z)-\psi_0(s,z)), \nabla\Phi^{n}_{\mu}(s)-\nabla\Phi^{0}_{\mu}(s)\Big)_{L^2(\mathbb{R}^d)}\nu(\mathrm{d} z) \dif s\nonumber\\
&:=I^\delta_{n, 1}(t)+I^\delta_{n, 2}(t).
\end{align}
Here $\delta\in (0,1]$ is an arbitrary constant.

Using arguments similar to those that led up to \eqref{eq Up 04} and \eqref{eq J-mu 01},  recalling Assumption \ref{assm-g}, and applying the Cauchy-Schwartz inequality and Lemma \ref{lem-est-M-n-mu}, we deduce
\begin{eqnarray}\label{I-n-2}
&&\hspace{-0.7cm}\sup_{t\in[0,T]}|I^\delta_{n, 1}(t)|\nonumber\\
&&\hspace{-0.7cm}\leq C_{\mu} L_1 m^{\frac12} \int_{0}^{T}\int_{0<|z|<\delta} |z|\|Y^{0}(s)\|_{L^2(\mathbb{R}^d)}|\psi_n(s,z)-\psi_0(s,z)|\|\nabla\Phi^{n}_{\mu}(s)-\nabla\Phi^{0}_{\mu}(s)\|_{L^2(\mathbb{R}^d)}\nu(\dif z)\dif s\nonumber\\
&&\hspace{-0.7cm}\leq C_{m,\mu,N}\Big[ \int_{0}^{T} \int_{0<|z|<\delta} |z||\psi_n(s,z)-1| \nu(\dif z) \dif s
     +
     \int_{0}^{T} \int_{0<|z|<\delta} |z||\psi_0(s,z)-1| \nu(\dif z) \dif s \Big]\nonumber\\
&&\hspace{-0.7cm}\leq C_{m,\mu,N}\sup_{h\in W^N}\int_0^T\int_{0<|z|<\delta}|z||h(s,z)-1|\nu(dz)ds.
\end{eqnarray}
Here the constant $C_\mu$ appears in (\ref{eq J-mu 00}).

To estimate the second term $I^\delta_{n, 2}$, fix $k\in\mathbb{N}$, let $0=t_0^k<t_1^k<\cdots<t_{2^k}^k=T$, where $t_i^k=\frac{iT}{ 2^{k}}$, be a partition of $[0,T]$ and set $\bar{s}_{k}(s)=t_i^k$ for $s \in\left[t_i^k,t_{i+1}^{k}\right)$, $i=0,\cdots,2^k-1$. We find
\begin{eqnarray}\label{I-n-1}
        &&\hspace{-0.7cm} \sup_{0\leq t\leq T} |I^\delta_{n, 2}(t)|\nonumber\\
          &&\hspace{-0.7cm}
          \leq
          \sup _{0 \leq t \leq T}\Big|\int_{0}^{t} \int_{\delta\leq |z|\leq 1}\Big(  \sum_{j=1}^m{\rm{i}}z_{j} \nabla J_{\mu} (g_{j}(J_{\mu}Y^{0}(s))(\psi_n(s,z)-\psi_0(s,z)), \nonumber\\
           &&\hspace{4cm}\big(\nabla\Phi^{n}_{\mu}(s)-\nabla\Phi^{0}_{\mu}(s)\big)
             -
             \big(\nabla\Phi^{n}_{\mu}(\bar{s}_k)-\nabla\Phi^{0}_{\mu}(\bar{s}_k)\big)\Big)_{L^2(\mathbb{R}^d)}\nu(\mathrm{d} z)\mathrm{d} s\Big|\nonumber\\
          &&\hspace{-0.4cm}+  \sup _{0 \leqslant t \leqslant T}\Big|\int_{0}^{t} \int_{\delta\leq |z|\leq 1} \Big(  \sum_{j=1}^m{\rm{i}}z_{j} \nabla J_{\mu} (g_{j}(J_{\mu}Y^{0}(s))- \sum_{j=1}^m{\rm{i}}z_{j} \nabla J_{\mu} (g_{j}(J_{\mu}Y^{0}(\bar{s}_k)),
          \nonumber\\&&\hspace{4.5cm}
          \nabla\Phi^{n}_{\mu}(\bar{s}_k)-\nabla\Phi^{0}_{\mu}(\bar{s}_k)\Big)_{L^2(\mathbb{R}^d)}
                \cdot(\psi_n(s,z)-\psi_0(s,z))\nu(\mathrm{d} z)\mathrm{d}s\Big|\nonumber\\
          &&\hspace{-0.4cm}+\sup _{1 \leqslant i \leqslant 2^{k}} \sup_{t_{i-1}^k \leqslant t \leqslant t_{i}^k}\Big|\int_{t_{i-1}^k}^{t} \int_{\delta\leq |z|\leq 1}\Big(   \sum_{j=1}^m{\rm{i}}z_{j} \nabla J_{\mu} (g_{j}(J_{\mu}Y^{0}(\bar{s}_k)),\nabla\Phi^{n}_{\mu}(\bar{s}_k)-\nabla\Phi^{0}_{\mu}(\bar{s}_k)\Big)_{L^2(\mathbb{R}^d)}\nonumber\\
          &&\hspace{9cm}\cdot(\psi_n(s,z)-\psi_0(s,z))\nu(\mathrm{d} z)\mathrm{d} s\Big|\nonumber\\
          &&\hspace{-0.4cm}+\sum_{i=1}^{2^{k}} \Big|\int_{t_{i-1}^k}^{t_i^k} \int_{\delta\leq |z|\leq 1}\Big(   \sum_{j=1}^m{\rm{i}}z_{j} \nabla J_{\mu} (g_{j}(J_{\mu}Y^{0}(\bar{s}_k)),\nabla\Phi^{n}_{\mu}(\bar{s}_k)-\nabla\Phi^{0}_{\mu}(\bar{s}_k)\Big)_{L^2(\mathbb{R}^d)}\nonumber\\
          &&\hspace{9cm}(\psi_n(s,z)-\psi_0(s,z))\nu(\mathrm{d} z)\mathrm{d} s\Big|\nonumber\\
                   &&\hspace{-0.7cm}:=\tilde{I}^{\delta,n,k}_{1}+\tilde{I}^{\delta,n,k}_{2}+\tilde{I}^{\delta,n,k}_{3}+\tilde{I}^{\delta,n,k}_{4}.
\end{eqnarray}

For the term $\tilde{I}^{\delta,n,k}_{1}$, using similar argument with which we derived \eqref{I-n-2} and then applying the inequality
\begin{equation}
a b \leq e^{K a}+\frac{1}{K}(b \log b-b+1)=e^{K a}+\frac{1}{K} l(b), \quad \text { for any } a, b \in(0, \infty), K \in[1, \infty),
\end{equation}
we infer, for any $K\in[1, \infty)$,
\begin{eqnarray}\label{I-n-1-1}
&&\hspace{-0.5cm}\tilde{I}^{\delta,n,k}_{1}\nonumber\\
&&\hspace{-0.5cm}\leq C_{m,\mu,N}\int_0^T \int_{\delta\leq |z|\leq 1} \big(|\psi_n(s,z)|+|\psi_0(s,z)|\big)\|\nabla\Phi^{n}_{\mu}(s)-\nabla\Phi^{n}_{\mu}(\bar{s}_k)\|_{L^2(\mathbb{R}^d)}\nu(dz)ds\nonumber\\
&&+
C_{m,\mu,N}\int_0^T \int_{\delta\leq |z|\leq 1} \big(|\psi_n(s,z)|+|\psi_0(s,z)|\big)\|\nabla\Phi^{0}_{\mu}(s)-\nabla\Phi^{0}_{\mu}(\bar{s}_k)\|_{L^2(\mathbb{R}^d)}\nu(dz)ds\nonumber\\
&&\hspace{-0.5cm}\leq C_{m,\mu,N}\int_0^T \int_{\delta\leq |z|\leq 1}\|\nabla\Phi^{n}_{\mu}(s)-\nabla\Phi^{n}_{\mu}(\bar{s}_k)\|_{L^2(\mathbb{R}^d)}
\left[2 e^K+\frac{l\left(\psi_{n}(s, z)\right)+l(\psi_0(s, z))}{K}\right] \nu(dz) d s\nonumber\\
&&+
C_{m,\mu,N}\int_0^T \int_{\delta\leq |z|\leq 1}\|\nabla\Phi^{0}_{\mu}(s)-\nabla\Phi^{0}_{\mu}(\bar{s}_k)\|_{L^2(\mathbb{R}^d)}
\left[ 2e^K+ \frac{l\left(\psi_{n}(s, z)\right)+ l(\psi_0(s, z))}{K}\right] \nu(dz) d s\nonumber\\
&&\hspace{-0.5cm}\leq e^KC_{m,\mu,N}\nu(\delta\leq |z|\leq 1)
\Big(
\int_0^T \|\nabla\Phi^{n}_{\mu}(s)-\nabla\Phi^{n}_{\mu}(\bar{s}_k)\|_{L^2(\mathbb{R}^d)}d s
+
\int_0^T\|\nabla\Phi^{0}_{\mu}(s)-\nabla\Phi^{0}_{\mu}(\bar{s}_k)\|_{L^2(\mathbb{R}^d)} d s
\Big)\nonumber\\
&&
+\frac{C_{m,\mu,N}\sup_{j=0,1,2...}\sup_{t\in[0,T]}\|\nabla\Phi^{j}_{\mu}(t)\|_{L^2(\mathbb{R}^d)}}{K}\nonumber\\
&&\hspace{-0.5cm}\leq e^KT^{\frac12}C_{m,\mu,N}\nu(\delta\leq |z|\leq 1)\sup_{j=0,1,2,...}\Big(\int_0^T \|\nabla\Phi^{j}_{\mu}(s)-\nabla\Phi^{j}_{\mu}(\bar{s}_k)\|^2_{L^2(\mathbb{R}^d)}d s\Big)^{\frac12}
+
\frac{C_{m,\mu,N}}{K},
\end{eqnarray}
where we haved used the fact that $\delta\leq |z|\leq 1$ and $\psi_n,\psi\in W^N$ in the third inequality.

Similarly, we have
\begin{align}\label{I-n-1-2}
\tilde{I}^{\delta,n,k}_{2}&\leq C_{m,\mu,N}\sup_{j=0,1,2...}\sup_{t\in[0,T]}\|\nabla\Phi^{j}_{\mu}(t)\|_{L^2(\mathbb{R}^d)}\nonumber\\
&\hspace{1cm}\times\int_0^T  \int_{\delta\leq |z|\leq 1} \|Y^{0}(s)-Y^{0}(\bar{s}_k)\|_{L^2(\mathbb{R}^d)}\left[2 e^K+\frac{l\left(\psi_{n}(s, z)\right)+l(\psi(s, z))}{K}\right] \nu(dz) d s\nonumber\\
&\leq e^KTC_{m,\mu,N}\nu(\delta\leq |z|\leq 1)\sup_{j=0,1,2...}\sup_{t\in[0,T]}\|\nabla\Phi^{j}_{\mu}(t)\|_{L^2(\mathbb{R}^d)}\sup_{s\in[0,T]}\|Y^{0}(s)-Y^{0}(\bar{s}_k)\|_{L^2(\mathbb{R}^d)}\nonumber\\
&\hspace{0.5cm}+\frac{C_{m,\mu,N}}{K},\nonumber\\
&\leq e^KTC_{m,\mu,N}\nu(\delta\leq |z|\leq 1)\sup_{j=0,1,2...}\sup_{s\in[0,T]}\|Y^{0}(s)-Y^{0}(\bar{s}_k)\|_{L^2(\mathbb{R}^d)}
+\frac{C_{m,\mu,N}}{K},
\end{align}
where we have used statement (3) of Lemma \ref{lem-est-M-n-mu} and \eqref{est uniform 1}.

We again follow the same line of argument as used in deducing  \eqref{I-n-2}, to find that
\begin{align}\label{I-n-1-3}
\tilde{I}^{\delta,n,k}_{3}
&\leq C_{m,\mu,N} \sup_{h\in W^N} \sup _{1 \leqslant i \leqslant 2^{k}}\int_{t_{i-1}}^{t_i} \int_{B}|z| |h(s,z)-1|\nu(dz)ds.
\end{align}
Now consider $\tilde{I}^{\delta,n,k}_{4}$. We first fix $k\in\mathbb{N}$ and $i=0,\cdots,2^k-1$. Due to the statement (3) of Lemma \ref{lem-est-M-n-mu}, it follows that
\begin{eqnarray}\label{eq star}
\sup_{n\in\mathbb{N}}\|\nabla\Phi^{n}_{\mu}(t_i^k)-\nabla\Phi^{0}_{\mu}(t_i^k)\|_{L^2(\mathbb{R}^d)}
\leq C_{\mu,N}<\infty,
\end{eqnarray}
which implies that
 $\{\nabla\Phi^{n}_{\mu}(t_i^k)-\nabla\Phi^{0}_{\mu}(t_i^k)\}_{n\in\mathbb{N}}$ is weakly tight in $L^2(\mathbb{R}^d)$, and hence there exists a subsequence (still denoted by itself) and $U_i^k\in L^2(\mathbb{R}^d)$ such that, for any $v\in L^2(\mathbb{R}^d)$,
$$
 \lim_{n\rightarrow\infty} \big( v, \nabla\Phi^{n}_{\mu}(t_i^k)-\nabla\Phi^{0}_{\mu}(t_i^k) \big)_{L^2(\mathbb{R}^d)}=( v, U_i^k)_{L^2(\mathbb{R}^d)}.
$$
Since $J_{\mu}$ is a contraction of $L^2(\mathbb{R}^d)$ and $\|J_{\mu}\|_{\mathcal{L}(L^2;H^2)}\leq C_\mu$,  by Assumption \ref{assm-g} and Corollary \ref{coro-uni-bou} we infer
\begin{eqnarray}\label{eq star}
\!\!\!\sup_{\delta\leq |z|\leq 1}\!\|\sum_{j=1}^m z_{j} \nabla J_{\mu} (g_{j}(J_{\mu}Y^{0}(t_i^k))\|_{L^2(\mathbb{R}^d)}\leq C_{m,\mu}\max_{j}\|g_j(J_{\mu}(Y^0(t_i^k)))\|_{L^2(\mathbb{R}^d)}
\leq C_{m,\mu,N}<\infty,
\end{eqnarray}
which yields
\begin{eqnarray}\label{eq App Gujia 01}
 \lim_{n\rightarrow\infty} V^n_{k,i}(z)
 =
 ( \sum_{j=1}^m{\rm{i}} z_{j} \nabla J_{\mu} (g_{j}(J_{\mu}Y^{0}(t_i^k)), U_i^k)_{L^2(\mathbb{R}^d)}:=V_{k,i}(z),\ \ \forall z\in \{\delta\leq |z|\leq 1\},
\end{eqnarray}
and
\begin{eqnarray}\label{eq App Gujia 02}
 \sup_{n\in\mathbb{N}}\sup_{\delta\leq |z|\leq 1} \Big|V^n_{k,i}(z)\Big|
 \leq C_{m,\mu,N}<\infty.
\end{eqnarray}
Here
$$
V^n_{k,i}(z):=\Big(   \sum_{j=1}^m{\rm{i}} z_{j} \nabla J_{\mu} (g_{j}(J_{\mu}Y^{0}(t_i^k)) ,\nabla\Phi^{n}_{\mu}(t_i^k)-\nabla\Phi^{0}_{\mu}(t_i^k)\Big)_{L^2(\mathbb{R}^d)}.
$$
Now, we may follow proofs similar in spirit to \cite[A.6]{Budhiraja-Chen-Dupuis} and assume without loss of generality that
$$
m_{n, \delta}:=\int_{t_{i}^k}^{t_{i+1}^k} \int_{\delta\leq |z|\leq 1} \psi_{n}(s, z) \nu(d z) d s \neq 0, \text { and } m_{\delta}:=\int_{t_{i}^k}^{t_{i+1}^k} \int_{\delta\leq |z|\leq 1} \psi(s, z) \nu(d z) d s \neq 0.
$$
Define probability measures $\tilde{\nu}_{n, \delta}, \tilde{\nu}_{\delta}$ and $\theta_{\delta}$ on
 $$\left([t_{i}^k, t_{i+1}^k) \times \{\delta\leq |z|\leq 1\}, \mathcal{B}\left([t_{i}^k, t_{i+1}^k)\right) \otimes \mathcal{B}\left(\{\delta\leq |z|\leq 1\}\right)\right)$$ as follows:
\begin{align*}
\tilde{\nu}_{n, \delta}(\cdot) &=\frac{ \nu_{T}^{\psi_{n}}\left(\cdot \cap([t_{i}^k, t_{i+1}^k) \times \{\delta\leq |z|\leq 1\})\right)}{m_{n, \delta}}, \\
\tilde{\nu}_{\delta}(\cdot) &=\frac{ \nu_{T}^{\psi}\left(\cdot \cap\left([t_{i}^k, t_{i+1}^k) \times \{\delta\leq |z|\leq 1\}\right)\right) }{m_{\delta}},\\
\theta_{\delta}(\cdot) &=\frac{\nu_{T}\left(\cdot \cap\left([t_{i}^k, t_{i+1}^k)\times \{\delta\leq |z|\leq 1\}\right)\right)}{\left.\nu_{T}\left([t_{i}^k, t_{i+1}^k) \times \{\delta\leq |z|\leq 1\}\right)\right)},
\end{align*}
where we employ the notations $\nu_T$ and $\nu^g_T$ introduced in (\ref{eq.corres-func-meas}).

 Denote
$A^{\delta}:= [t_{i}^k, t_{i+1}^k) \times \{\delta\leq |z|\leq 1\}$ for simplicity. In 
 the following, we will prove (\ref{prop-2-proof-I-4}). To do this, 
using similar arguments as in the proof of \cite[Lemma 10.9]{BD-19}, we may assume without loss of generality  that ${\rm Leb}_T\otimes\nu\big{(}\partial A^{\delta}\big{)} =0$. Recall the topology $W^N$ defined in \eqref{S_N}. 
Since $\psi_{n}\rightarrow \psi$ as $n\rightarrow\infty$ in $W^N$ and ${\rm Leb}_T \otimes\nu\big{(}\partial A^{\delta}\big{)} =0$, we infer
\begin{eqnarray}\label{eq App Gujia 03}
\lim _{n \rightarrow \infty} m_{n, \delta}=m_{\delta},
\end{eqnarray}
and
\begin{eqnarray}\label{eq App Gujia 04}
\tilde{\nu}_{n, \delta}\text{  converges weakly to } \tilde{\nu}_{\delta}\text{ as }n \rightarrow \infty.
\end{eqnarray}
Moreover, there exists a constant $\alpha_{\delta}$ such that 
\begin{align}\label{eq entropy}
 \sup _{n \geq 1} R\left(\tilde{\nu}_{n, \delta} \| \theta_{\delta}\right)=& \sup _{n \geq 1} \int_{t_{i}^k}^{ t_{i+1}^k} \int_{\delta\leq |z|\leq 1} \log \left(\frac{\nu_{T}\left([t_{i}^k, t_{i+1}^k) \times \{\delta\leq |z|\leq 1\}\right)}{m_{n, \delta}} \psi_{n}(s, z)\right) \frac{1}{m_{n, \delta}} \psi_{n}(s, z) \nu(d z) d s \nonumber\\
 =& \sup _{n \geq 1} \frac{1}{m_{n, \delta}} \int_{t_{i}^k}^{ t_{i+1}^k} \int_{\delta\leq |z|\leq 1}\left(l\left(\psi_{n}(s, z)\right)+\psi_{n}(s, z)-1\right) \nu(d z) d s \nonumber\\ &+\log \frac{\nu_{T}\left(\left[t_{i}^k, t_{i+1}^k\right) \times \{\delta\leq |z|\leq 1\}\right)}{m_{n, \delta}} \nonumber\\
  \leq & \sup _{n \geq 1}\left(\frac{N}{m_{n, \delta}}+1-\frac{\nu_{T}\left([t_{i}^k, t_{i+1}^k) \times \{\delta\leq |z|\leq 1\}\right)}{m_{n, \delta}}+\log \frac{\nu_{T}\left([t_{i}^k, t_{i+1}^k) \times \{\delta\leq |z|\leq 1\}\right)}{m_{n, \delta}}\right) \nonumber\\ \leq & \alpha_{\delta}<\infty .
 \end{align}
Based on (\ref{eq App Gujia 01})--(\ref{eq entropy}) and using \cite[Lemma 2.8]{BD 1998}, we are able to deduce that
\begin{align*}
&\lim _{n \rightarrow \infty} \int_{t_{i}^k}^{ t_{i+1}^k} \int_{\delta\leq |z|\leq 1} V^n_{k,i}(s, z) \tilde{\nu}_{n, \delta}(d z d s)=\int_{t_{i}^k}^{ t_{i+1}^k} \int_{\delta\leq |z|\leq 1} V_{k,i}(s, z) \tilde{\nu}_{\delta}(d z d s),\\
&\lim _{n \rightarrow \infty} \int_{t_{i}^k}^{ t_{i+1}^k} \int_{\delta\leq |z|\leq 1} V^n_{k,i}(s, z) \tilde{\nu}_{\delta}(d z d s)=\int_{t_{i}^k}^{ t_{i+1}^k} \int_{\delta\leq |z|\leq 1} V_{k,i}(s, z) \tilde{\nu}_{\delta}(d z d s).
\end{align*}
It follows that
\begin{align*}
&\lim _{n \rightarrow \infty} \int_{t_{i}^k}^{ t_{i+1}^k} \int_{\delta\leq |z|\leq 1} V^n_{k,i}(s, z) \psi_n(s,z) \nu(\mathrm{d} z)\mathrm{d} s=\int_{t_{i}^k}^{ t_{i+1}^k} \int_{\delta\leq |z|\leq 1} V_{k,i}(s, z) \psi(s,z) \nu(\mathrm{d} z)\mathrm{d} s,\\
&\lim _{n \rightarrow \infty} \int_{t_{i}^k}^{ t_{i+1}^k} \int_{\delta\leq |z|\leq 1} V^n_{k,i}(s, z) \psi(s,z) \nu(\mathrm{d} z)\mathrm{d} s=\int_{t_{i}^k}^{ t_{i+1}^k} \int_{\delta\leq |z|\leq 1} V_{k,i}(s, z) \psi(s,z) \nu(\mathrm{d} z)\mathrm{d} s.
\end{align*}
Hence
\begin{align}
\begin{split}\label{prop-2-proof-I-4}
      &\lim_{n\rightarrow\infty}  \Big|\int_{t_{i}^k}^{t_{i+1}^k} \int_{\delta\leq |z|\leq 1}\Big(   \sum_{j=1}^m{\rm{i}}z_{j} \nabla J_{\mu} (g_{j}(J_{\mu}Y^{0}(\bar{s}_k)),\nabla\Phi^{n}_{\mu}(\bar{s}_k)-\nabla\Phi^{0}_{\mu}(\bar{s}_k)\Big)_{L^2(\mathbb{R}^d)}\\
          &\hspace{9cm}(\psi_n(s,z)-\psi_0(s,z))\nu(\mathrm{d} z)\mathrm{d} s\Big|\\
        &=\lim_{n\rightarrow\infty}   \left|\int_{t_{i}^k}^{ t_{i+1}^k} \int_{\delta\leq |z|\leq 1} V^n_{k,i}(s, z) \psi_n(s,z) \nu(\mathrm{d} z)\mathrm{d} s
        -
         \int_{t_{i}^k}^{ t_{i+1}^k} \int_{\delta\leq |z|\leq 1} V^n_{k,i}(s, z) \psi(s,z) \nu(\mathrm{d} z)\mathrm{d} s \right|
        \\
        &=0.
        \end{split}
\end{align}
Therefore, for any fixed $k\in\mathbb{N}$ and $\delta\in(0,1]$, we have
\begin{eqnarray}\label{I-n-1-4}
\lim_{n\rightarrow\infty}\tilde{I}^{\delta,n,k}_{4}=0.
\end{eqnarray}
Now we collect together all the terms \eqref{eq App 01}, \eqref{I-n-2}, \eqref{I-n-1}, \eqref{I-n-1-1}, \eqref{I-n-1-2}, \eqref{I-n-1-3}, and \eqref{I-n-1-4}, to deduce that for any fixed $k\in\mathbb{N}$, $K\in[1,\infty)$ and $\delta\in(0,1]$,
\begin{align}\label{eq App 02}
&\limsup_{n\rightarrow\infty}\sup_{t\in[0,T]}\|\nabla\Phi^{n}_{\mu}(t)-\nabla\Phi^{0}_{\mu}(t)\|^2_{L^2(\mathbb{R}^d)}\nonumber\\
&\leq
C_{m,\mu,N}\sup_{h\in W^N}\int_0^T\int_{0<|z|<\delta}|z||h(s,z)-1|\nu(dz)ds\nonumber\\
&\hspace{0.5cm}+
e^KT^{\frac12}C_{m,\mu,N}\nu(\delta\leq |z|\leq 1)\sup_{j=0,1,2,...}\Big(\int_0^T \|\nabla\Phi^{j}_{\mu}(s)-\nabla\Phi^{j}_{\mu}(\bar{s}_k)\|^2_{L^2(\mathbb{R}^d)}d s\Big)^{\frac12}\nonumber\\
&\hspace{0.5cm}+
e^KTC_{m,\mu,N}\nu(\delta\leq |z|\leq 1)\sup_{s\in[0,T]}\|Y^{0}(s)-Y^{0}(\bar{s}_k)\|_{L^2(\mathbb{R}^d)}\nonumber\\
&\hspace{0.5cm}+\frac{C_{m,\mu,N}}{K}\nonumber\\
&\hspace{0.5cm}+
C_{m,\mu,N} \sup_{h\in W^N} \sup _{1 \leqslant i \leqslant 2^{k}}\int_{t_{i-1}}^{t_i} \int_{B}|z| |h(s,z)-1|\nu(dz)ds.
\end{align}
In light of statement (4) of Lemma \ref{lem-est-M-n-mu}, Lemma \ref{lemma-est-ass} and the fact $Y^0\in C([0,T],L^2(\mathbb{R}^d))$, letting $k\rightarrow\infty$, we find for any fixed $K\in[1,\infty)$ and $\delta\in(0,1]$,
\begin{eqnarray}\label{eq App 025}
&&\limsup_{n\rightarrow\infty}\sup_{t\in[0,T]}\|\nabla\Phi^{n}_{\mu}(t)-\nabla\Phi^{0}_{\mu}(t)\|^2_{L^2(\mathbb{R}^d)}\nonumber\\
&&\leq
C_{m,\mu,N}\sup_{h\in W^N}\int_0^T\int_{0<|z|<\delta}|z||h(s,z)-1|\nu(dz)ds+\frac{C_{m,\mu,N}}{K}.
\end{eqnarray}
Consequently,  on account of \eqref{eq zhai 1}, it is not difficulty to conclude that
\begin{eqnarray}\label{eq App 03}
\limsup_{n\rightarrow\infty}\sup_{t\in[0,T]}\|\nabla\Phi^{n}_{\mu}(t)-\nabla\Phi^{0}_{\mu}(t)\|^2_{L^2(\mathbb{R}^d)}=0,
\end{eqnarray}
which gives (\ref{eq H1 01}) and hence the proof of this lemma is complete.

\end{proof}
%-------------------
At last we are ready to prove \eqref{eq LDP 6}. 

\begin{lemma}\label{Lemma-M-n-appro}
For any fixed $\mu>0$, we have
\begin{align*}
\lim_{n\rightarrow\infty}\|\Phi^{n}_{\mu}-\Phi^{0}_{\mu}\|_{E^{(T,p)}}=0.
 \end{align*}
\end{lemma}
\begin{proof}

Since $H^{1}(\mathbb{R}^d)\hookrightarrow L^q(\mathbb{R}^d)$ for every $q\in [2,\frac{2d}{d-2})$ (c.f.  \cite[Page 206]{Triebel}), we have $H^1(\mathbb{R}^d))\hookrightarrow L^r(\mathbb{R}^d)\cap L^2(\mathbb{R}^d)$, as $r=2\sigma+2$ and $\sigma\in(0,\frac2d)$, and hence $C(0,T;H^1(\mathbb{R}^n))\hookrightarrow L^p(0,T;L^r(\mathbb{R}^d))\cap C(0,T;L^2(\mathbb{R}^n))$. It then follows from Lemma \ref{lemma-D-U} that
\begin{align*}
 \lim_{n\rightarrow\infty} \|\Phi^{n}_{\mu}-\Phi^{0}_{\mu}\|_{E^{(T,p)}}
  \leq
  C  \lim_{n\rightarrow\infty} \sup_{t\in[0,T]} \|\Phi^{n}_{\mu}(t)-\Phi^{0}_{\mu}(t)\|_{H^1(\mathbb{R}^d)}=0,
\end{align*}
from which the required result follows.

\end{proof}

The proof of Proposition \ref{Yu-con} is now complete.

\section{Proof of Proposition \ref{lem LDP 2}}\label{sect-prop-2}

As a preparation for the proof of Proposition \ref{lem LDP 2}, let us first quote the following results.
\begin{lemma}\label{lem-est-nonlinearity}\cite[Lemma 2.17]{Saa15}. Let $T>0$ and $u \in C\left([0, T] ; \mathbb{R}_{+}\right)$, such that
$$
u \leq A+B u^{\theta} \quad \text { on }[0, T],
$$
where $A, B>0, \theta>1, A<\left(1-\frac{1}{\theta}\right)(\theta B)^{-\frac{1}{\theta-1}}$, and $u(0) \leq(\theta B)^{-\frac{1}{\theta-1}}$. Then $$u \leq \frac{\theta}{\theta-1} A\quad\text{ on }[0, T].$$
\end{lemma}

\vskip 0.3cm
For $z\in B$, $u\in\mathbb{C}$ and $0<\varepsilon \leqslant 1$, we denote
\begin{align*}
&G(\varepsilon, z, u):=\Phi^\varepsilon( z, u)-u \\
&H(\varepsilon, z, u):=\Phi^\varepsilon(z, u)-u+\varepsilon\mathrm{i}\sum_{j=1}^{m} z_{j}  g_{j}(u).
\end{align*}

\begin{lemma}\cite[Lemma 3.3]{BLZ} \label{lem-G+H}
There exists $C^{i}_m>0$, $ i=1, \ldots 4$ such that for any $z \in B$ and $y, y_{i} \in \mathbb{C}, i=1,2$, and $0<\varepsilon \leqslant 1$,
\begin{align*}
|G(\varepsilon,z, y)| & \leq C_{m}^{1}\varepsilon |z|_{\mathbb{R}^{m}}|y| \\
\left|G\left(\varepsilon,z, y_{1}\right)-G\left(\varepsilon,z, y_{2}\right)\right| & \leq C_{m}^{2}\varepsilon|z|_{\mathbb{R}^{m}}\left|y_{1}-y_{2}\right|\\
|H(\varepsilon,z, y)| & \leq C_{m}^{3}\varepsilon^2|z|_{\mathbb{R}^{m}}^{2}|y| \\
\left|H\left(\varepsilon,z, y_{1}\right)-H\left(\varepsilon,z, y_{2}\right)\right| &\leq C_{m}^{4}\varepsilon^2|z|_{\mathbb{R}^{m}}^{2}\left|y_{1}-y_{2}\right|.
\end{align*}

\end{lemma}

The following uniform estimate for the controlled stochastic NLS is essential to the proof of Proposition \ref{lem LDP 2}.

\begin{lemma}\label{lem uniform}

For any $\eta\in(0,1)$,
there is a constant $M_\eta>0$ such that
\begin{align*}
      \sup_{\varepsilon\in(0,1)}P\Big[\|X^{\psi_\varepsilon}\|_{L^{p}(0, T; L^{r}(\mathbb{R}^{d}))}\geq M_\eta\Big]\leq \eta.
  \end{align*}
\end{lemma}
\begin{proof}[Proof of Lemma \ref{lem uniform}]

Recall that (see (\ref{EQ4 LDP 2-02}))
 \begin{align*}
X^{\psi_\varepsilon}(t)
=& S_{t} u_{0}-\mathrm{i} \int_{0}^{t } S_{t -s}(\lambda|X^{\psi_\varepsilon}(s)|^{2 \sigma} X^{\psi_\varepsilon}(s)) \mathrm{d} s\\
&+ \int_{0}^{t} \int_{B}  S_{t-s}\left[\Phi^\varepsilon(z, X^{\psi_\varepsilon}(s-))-X^{\psi_\varepsilon}(s-)\right] \tilde{N}^{\varepsilon^{-1}\psi_\varepsilon}(\mathrm{d} s, \mathrm{d} z)\\
&+\varepsilon^{-1}\int_{0}^{t } \int_{B} S_{t -s}\Big[\Phi^\varepsilon(z, X^{\psi_\varepsilon}(s))-X^{\psi_\varepsilon}(s)+\varepsilon\;\mathrm{i} \sum_{j=1}^{m} z_{j} g_{j}(X^{\psi_\varepsilon}(s))\Big]\psi_{\varepsilon}(s,z) \nu(\mathrm{d} z) \mathrm{d} s\\
&-\int_{0}^{t } \int_{B} S_{t -s}\Big[\sum_{j=1}^{m}\mathrm{i}  z_{j} g_{j}(X^{\psi_{\varepsilon}}(s))(\psi_{\varepsilon}(s,z)-1)\Big] \nu(\mathrm{d} z) \mathrm{d} s\\
=&S_{t} u_{0}+I_1(X^{\psi_\varepsilon})(t)+I_2(X^{\psi_\varepsilon})(t)+I_3(X^{\psi_\varepsilon})(t)+I_4(X^{\psi_\varepsilon})(t).
\end{align*}

An argument similar to those that established (\ref{eq Up 02}) implies
\begin{align*}
\|I_1(X^{\psi_\varepsilon})\|_{L^p(0,t;L^r(\mathbb{R}^d))}
\leq C t^{1-\frac{d \sigma}{2}}\|X^{\psi_\varepsilon}\|_{L^p(0,t;L^r(\mathbb{R}^d))}^{2 \sigma+1}.
\end{align*}
Recalling the stochastic Strichartz inequality \eqref{sto-stri-est-1}, Lemma \ref{lem-G+H} and $L^2(\mathbb{R}^d)$-conservation law of $X^{\psi_\varepsilon}$(see (\ref{eq lem uniform 01})), we obtain
\begin{align}\label{Uniform estimate I2}
\mathbb{E}\|I_2(X^{\psi_\varepsilon})\|_{L^{p}\left(0, T ; L^{r}\left(\mathbb{R}^{d}\right)\right)}^p\leq& C_{p} \mathbb{E}\left(\int_{0}^{T} \int_{B}\|\Phi^\varepsilon(z, X^{\psi_\varepsilon}(s))-X^{\psi_\varepsilon}(s)\|_{L^{2}(\mathbb{R}^{d})}^{2}\varepsilon^{-1}\psi_{\varepsilon} (s,z)\nu(\mathrm{d} z) \mathrm{d} s\right)^{\frac{p}{2}}\nonumber\\
&+C_{p} \mathbb{E}\left(\int_{0}^{T} \int_{B}\|\Phi^\varepsilon(z, X^{\psi_\varepsilon}(s))-X^{\psi_\varepsilon}(s)\|_{L^{2}(\mathbb{R}^{d})}^{p}\varepsilon^{-1}\psi_{\varepsilon}(s,z) \nu(\mathrm{d} z) \mathrm{d} s\right)\nonumber\\
\leq&  C_{p,m} \mathbb{E}\left(\int_{0}^{T} \int_{B}\varepsilon|z|_{\mathbb{R}^{m}}^{2}\|X^{\psi_\varepsilon}(s)\|_{L^{2}(\mathbb{R}^{d})}^{2}\psi_{\varepsilon}(s,z) \nu(\mathrm{d} z) \mathrm{d} s\right)^{\frac{p}{2}}\nonumber\\
&+C_{p, m} \mathbb{E}\left(\int_{0}^{T} \int_{B}\varepsilon^{p-1}|z|_{\mathbb{R}^{m}}^{p}\|X^{\psi_\varepsilon}(s)\|_{L^{2}(\mathbb{R}^{d})}^{p}\psi_{\varepsilon}(s,z) \nu(\mathrm{d} z) \mathrm{d} s\right)\nonumber\\
&\leq C_{p, m}K_1(\varepsilon^{p/2}+\varepsilon^{p-1})\|u_0\|_{L^2(\mathbb{R}^d)}^{p},
\end{align}
 where $K_1=\sup_{h\in W^N}\left(\int_0^T\int_{B}|z|_{\mathbb{R}^{m}}^{2}h (s,z) \nu(\mathrm{d} z)\dif s\right)^{\frac{p}{2}}+\sup_{h\in W^N}\int_0^T\int_{B}|z|_{\mathbb{R}^{m}}^{2}h (s,z) \nu(\mathrm{d} z)\dif s<\infty$ as guaranteed by Lemma \ref{lemma-est-ass}.
Again, applying  the Strichartz inequality \eqref{Stri-est-3}, Lemma \ref{lem-G+H} and $L^2(\mathbb{R}^d)$-conservation law of $X^{\psi_\varepsilon}$ yields
 \begin{align}\label{Uniform estimate I3}
&\|I_3(X^{\psi_\varepsilon})\|_{L^p(0,T;L^r(\mathbb{R}^d))}\nonumber\\
&=\bigg\|  \varepsilon^{-1}\int_{0}^{\cdot } \int_{B} S_{\cdot -s}\Big[\Phi^\varepsilon(z, X^{\psi_\varepsilon}(s))-X^{\psi_\varepsilon}(s)+\varepsilon\;\mathrm{i} \sum_{j=1}^{m} z_{j} g_{j}(X^{\psi_\varepsilon}(s))\Big] \psi^{\varepsilon}(s,z)\nu(\mathrm{d} z) \mathrm{d} s  \bigg\|_{L^{p}(0, T ; L^{r}(\mathbb{R}^d))}\nonumber\\
&=\varepsilon^{-1}\left\|\int_{0}^{\cdot }S_{\cdot -s} \int_{B}H(\varepsilon,z, X^{\psi_\varepsilon}(s)) \psi^{\varepsilon}(s,z)\nu(\mathrm{d} z) \mathrm{d} s  \right\|_{L^{p}(0, T ; L^{r}(\mathbb{R}^d))}\nonumber\\
&\leq C\varepsilon^{-1} \int_{0}^{T}\left\|\int_{B}  H(\varepsilon,z, X^{\psi_\varepsilon}(s))\psi^{\varepsilon}(s,z)  \nu(\mathrm{d} z)\right\|_{L^{2}(\mathbb{R}^d)} \mathrm{d} s\nonumber\\
&\leq  \varepsilon^{-1}C \int_{0}^{T} \int_{B}\| H(\varepsilon,z, X^{\psi_\varepsilon}(s)) \|_{L^{2}(\mathbb{R}^d)}\psi^{\varepsilon}(s,z) \nu(\mathrm{d} z) \mathrm{d} s\nonumber\\
&\leq  \varepsilon^{-1}C_m \int_{0}^{T}\left(\int_{B}\varepsilon^2 |z|_{\mathbb{R}^{m}}^{2}\psi^{\varepsilon}(s,z) \nu(\mathrm{d} z)\right)\|X^{\psi_\varepsilon}(s)\|_{L^{2}(\mathbb{R}^d)} \mathrm{d} s \nonumber\\
&\leq \varepsilon C_mK_1  \|u_0\|_{L^{2}(\mathbb{R}^{d})}.
 \end{align}

For the term $I_4$,  using again the Minkowski inequality, the Strichartz inequality \eqref{Stri-est-3}, the Cauchy-Schwartz inequality and Assumption \ref{assm-g} gives
 \begin{align*}
\|I_4(X^{\psi_\varepsilon})\|_{L^{p}(0, T ; L^{r}(\mathbb{R}^d))}&= 
 \bigg\| \int_{0}^{\cdot } \int_{B} S_{\cdot -s}\Big[\Big(\sum_{j=1}^{m}\mathrm{i}  z_{j} g_{j}(X^{\psi_{\varepsilon}}(s))\Big)(\psi_{\varepsilon}(s,z)-1)\Big] \nu(\mathrm{d} z) \mathrm{d} s  \bigg\|_{L^{p}(0, T ; L^{r}(\mathbb{R}^d))}\\
&\leq  C\int_{0}^{T }\bigg\| \int_B \Big(\sum_{j=1}^{m}\mathrm{i}  z_{j} g_{j}(X^{\psi_{\varepsilon}}(s))\Big)(\psi_{\varepsilon}(s,z)-1)\nu(\dif z)\bigg\|_{L^2(\mathbb{R}^d)}\dif s\\
&\leq C\int_{0}^{T } \int_{B}\Big\| \Big(\sum_{j=1}^{m}\mathrm{i}  z_{j} g_{j}(X^{\psi_{\varepsilon}}(s))\Big)(\psi_{\varepsilon}(s,z)-1)\Big\|_{L^2(\mathbb{R}^d)} \nu(\mathrm{d} z) \mathrm{d} s\\
&\leq C \int_0^T \int_B|z|_{\mathbb{R}^m}\Big\|\Big(\sum_{j=1}^m|g_{j}(X^{\psi_{\varepsilon}}(s))|^2\Big)^{\frac12}\Big\|_{L^2(\mathbb{R}^d)}|\psi_{\varepsilon}(s,z)-1| \nu(\mathrm{d} z) \mathrm{d} s\\
&\leq C_m\int_0^T \int_B\big\|X^{\psi_{\varepsilon}}(s)\big\|_{L^2(\mathbb{R}^d)} |z|_{\mathbb{R}^m}|\psi_{\varepsilon}(s,z)-1| \nu(\mathrm{d} z) \mathrm{d} s\\
&\leq C_mK_2\big\|u_0\big\|_{L^2(\mathbb{R}^d)},
 \end{align*}
 where $K_2=\sup_{h\in W^N}\int_0^T\int_{B} |z|_{\mathbb{R}^{m}}|h(s,z) -1|\nu(\mathrm{d} z)\dif s<\infty$ (guaranteed by Lemma \ref{lemma-est-ass} again).

It follows from all of the above that, for any $\varepsilon\in(0,1)$,
\begin{align}\label{uniform 01}
\big\|X^{\psi_{\varepsilon}}\big\|_{L^p(0,t;L^r(\mathbb{R}^d))}&\leq C_{m,K_1,K_2}\|u_0\|_{L^2}+ C t^{1-\frac{d \sigma}{2}}\|X^{\psi_\varepsilon}\|_{L^p(0,t;L^r(\mathbb{R}^d))}^{2 \sigma+1}+\|I_2(X^{\psi_{\varepsilon}})\|_{L^p(0,T;L^2(\mathbb{R}^d))}\nonumber\\
&:=M_{\varepsilon}^T+ C t^{1-\frac{d \sigma}{2}}\|X^{\psi_\varepsilon}\|_{L^p(0,t;L^r(\mathbb{R}^d))}^{2 \sigma+1},
 \end{align}
 where $M_{\varepsilon}^T=  C_{m,K_1,K_2}\|u_0\|_{L^2}+\|I_2(X^{\psi_{\varepsilon}})\|_{L^p(0,T;L^2(\mathbb{R}^d))}$.

  By (\ref{Uniform estimate I2}), for any $\eta>0$, there exists $\Theta_\eta>0$ such that
  \begin{eqnarray}\label{Uniform estimate I2 01}
   \inf_{\varepsilon\in(0,1)}P(M_{\varepsilon}^T\leq \Theta_\eta)\geq 1-\eta.
  \end{eqnarray}
  Now fix $\varepsilon\in(0,1)$ and $\omega\in \{M_{\varepsilon}^T\leq \Theta_\eta\}$. We have
  \begin{align*}
\big\|X^{\psi_{\varepsilon}}(\omega)\big\|_{L^p(0,t;L^r(\mathbb{R}^d))}
\leq
\Theta_\eta+ C t^{1-\frac{d \sigma}{2}}\|X^{\psi_\varepsilon}(\omega)\|_{L^p(0,t;L^r(\mathbb{R}^d))}^{2 \sigma+1}.
 \end{align*}
  Choose $T^0=T\wedge\big( (2\sigma+1)C \big( \frac{2\sigma+1}{\sigma}\Theta_\eta  \big)^{2\sigma}\big)^{-\frac{1}{1-d\sigma/2}}$ which ensures
 $$\Theta_\eta<(1-\frac1{2\sigma+1})((2\sigma+1) C (T^0)^{1-\frac{d \sigma}{2}})^{-\frac1{2\sigma}}.$$
So we can apply Lemma \ref{lem-est-nonlinearity} with $\theta=2\sigma+1$ to deduce
 \begin{align*}
 \big\|X^{\psi_{\varepsilon}}(\omega)\big\|_{L^p(0,T^0;L^r(\mathbb{R}^d))}&\leq \frac{2\sigma+1}{2\sigma}\Theta_\eta.
  \end{align*}
  Using arguments similar to those that led up to (\ref{uniform 01}) and $L^2(\mathbb{R}^d)$-conservation law of $X^{\psi_\varepsilon}$, for any $t\geq T^0$, we infer
  \begin{align}\label{uniform 02}
&\big\|X^{\psi_{\varepsilon}}(\omega)\big\|_{L^p(T^0,t;L^r(\mathbb{R}^d))}\nonumber\\
&\leq C_{m,K_1,K_2}\|u_0\|_{L^2}+ C (t-T^0)^{1-\frac{d \sigma}{2}}\|X^{\psi_\varepsilon}\|_{L^p(T^0,t;L^r(\mathbb{R}^d))}^{2 \sigma+1}+\|I_2(X^{\psi_{\varepsilon}})\|_{L^p(T^0,T;L^2(\mathbb{R}^d))}\nonumber\\
&\leq \Theta_\eta+ C (t-T^0)^{1-\frac{d \sigma}{2}}\|X^{\psi_\varepsilon}\|_{L^p(T^0,t;L^r(\mathbb{R}^d))}^{2 \sigma+1},
 \end{align}
 where the constant $C$ is the same as \eqref{uniform 01}.
  Hence,
  \begin{align*}
 \big\|X^{\psi_{\varepsilon}}(\omega)\big\|_{L^p(T^0,2T^0\wedge T;L^r(\mathbb{R}^d))}&\leq \frac{2\sigma+1}{2\sigma}\Theta_\eta.
  \end{align*}
  Therefore,
  \begin{align*}
 \big\|X^{\psi_{\varepsilon}}(\omega)\big\|_{L^p(0,2T^0\wedge T;L^r(\mathbb{R}^d))}&\leq 2\frac{2\sigma+1}{2\sigma}\Theta_\eta.
  \end{align*}
  Iterating similar estimates on $[kT^0\wedge T,(k+1)T^0\wedge T)   ]$, $1\leq k\leq [\frac{T}{T^0}]$, we find
   \begin{align*}
 \big\|X^{\psi_{\varepsilon}}(\omega)\big\|_{L^p(0,T;L^r(\mathbb{R}^d))}&\leq \Big(  \Big[\frac{T}{T^0}   \Big]+1  \Big)\frac{2\sigma+1}{2\sigma}\Theta_\eta.
  \end{align*}
  Hence we infer
  \begin{eqnarray*}
  \inf_{\varepsilon\in(0,1)}P\Big(\big\|X^{\psi_{\varepsilon}}\big\|_{L^p(0,T;L^r(\mathbb{R}^d))}\leq \Big(  \Big[\frac{T}{T^0}   \Big]+1  \Big)\frac{2\sigma+1}{2\sigma}\Theta_\eta\Big)\geq
  \inf_{\varepsilon\in(0,1)}P(M_{\varepsilon}^T\leq \Theta_\eta)\geq 1-\eta,
  \end{eqnarray*}
  which proves Lemma \ref{lem uniform}.

\end{proof}

\begin{proof}[Proof of Proposition \ref{lem LDP 2}]
Given $\eta\in(0,1)$, let $M_\eta$ be the constant appearing in Lemma \ref{lem uniform}.
Define $$\tau_{M_\eta}^{\varepsilon}=\inf \Big\{t \geqslant 0 ,\big\|X^{\psi_{\varepsilon}}\big\|_{L^{p}(0, t;L^{r}(\mathbb{R}^{d}))} \geqslant M_\eta\Big\}\wedge T.$$
Note that
 \begin{align}
 \begin{split}\label{prop-eq-X-psi}
X^{\psi_\varepsilon}(t)
=& S_{t} X^{\psi_\varepsilon}(0)-\mathrm{i} \int_{0}^{t } S_{t -s}(\lambda|X^{\psi_\varepsilon}(s)|^{2 \sigma} X^{\psi_\varepsilon}(s) \mathrm{d} s\\
&+ \int_{0}^{t} \int_{B}  S_{t-s}\left[\Phi^\varepsilon(z, X^{\psi_\varepsilon}(s ))-X^{\psi_\varepsilon}(s  )\right] \tilde{N}^{\varepsilon^{-1}\psi_\varepsilon}(\mathrm{d} s, \mathrm{d} z)\\
&+\varepsilon^{-1}\int_{0}^{t } \int_{B} S_{t -s}\Big[\Phi^\varepsilon(z, X^{\psi_\varepsilon}(s))-X^{\psi_\varepsilon}(s)+\varepsilon\;\mathrm{i} \sum_{j=1}^{m} z_{j} g_{j}(X^{\psi_\varepsilon}(s))\Big]\psi_{\varepsilon}(s,z) \nu(\mathrm{d} z) \mathrm{d} s\\
&-\int_{0}^{t } \int_{B} S_{t -s}\Big[\mathrm{i}\sum_{j=1}^{m}  z_{j} g_{j}(X^{\psi_{\varepsilon}}(s))(\psi_{\varepsilon}(s,z)-1)\Big] \nu(\mathrm{d} z) \mathrm{d} s,\ \ t\in[0,T],
\end{split}
\end{align}
and
\begin{align}
Y^{\psi^{\varepsilon}}(t)=& S_{t} Y^{\psi^{\varepsilon}}(0)-\mathrm{i} \int_{0}^{t } S_{t -s}(\lambda|Y^{\psi^{\varepsilon}} (s)|^{2 \sigma} Y^{\psi^{\varepsilon}}(s) \mathrm{d} s \nonumber\\
&-\int_0^t\!\!\int_{B}S_{t -s}\Big[\sum_{j=1}^{m}\mathrm{i}  z_{j} g_{j}(Y^{\psi^{\varepsilon}}(s))(\psi^{\varepsilon}(s,z)-1)\Big]\nu(\dif z)\dif s,\ \ t\in[0,T].\label{prop-eq-Y-psi}
\end{align}
Taking the difference of the above two equations gives
 \begin{align*}
&X^{\psi_\varepsilon}(t)- Y^{\psi^{\varepsilon}}(t)\\
&=S_{t} \big(X^{\psi_\varepsilon}(0)-Y^{\psi^{\varepsilon}}(0)\big)-\Big[\mathrm{i} \int_{0}^{t } S_{t -s}(\lambda|X^{\psi_\varepsilon}(s)|^{2 \sigma} X^{\psi_\varepsilon}(s) \mathrm{d} s-\mathrm{i} \int_{0}^{t } S_{t -s}(\lambda|Y^{\psi^{\varepsilon}} (s)|^{2 \sigma} Y^{\psi^{\varepsilon}}(s) \mathrm{d} s\Big]\\
&\hspace{0.5cm}+\int_{0}^{t} \int_{B}  S_{t-s}\left[\Phi^\varepsilon(z, X^{\psi_\varepsilon}(s ))-X^{\psi_\varepsilon}(s  )\right] \tilde{N}^{\varepsilon^{-1}\psi_\varepsilon}(\mathrm{d} s, \mathrm{d} z)\\
&\hspace{0.5cm}+\varepsilon^{-1}\int_{0}^{t } \int_{B} S_{t -s}\Big[\Phi^\varepsilon(z, X^{\psi_\varepsilon}(s))-X^{\psi_\varepsilon}(s)+\varepsilon\;\mathrm{i} \sum_{j=1}^{m} z_{j} g_{j}(X^{\psi_\varepsilon}(s))\Big] \psi_{\varepsilon}(s,z)\nu(\mathrm{d} z) \mathrm{d} s\\
&\hspace{0.5cm}-\int_{0}^{t } \int_{B} S_{t -s}\Big[\Big(\sum_{j=1}^{m}\mathrm{i}  z_{j} g_{j}(X^{\psi_{\varepsilon}}(s))-\sum_{j=1}^{m}\mathrm{i}  z_{j} g_{j}(Y^{\psi_{\varepsilon}}(s))\Big)(\psi_{\varepsilon}(s,z)-1)\Big] \nu(\mathrm{d} z) \mathrm{d} s\\
&:=S_{t} \big(X^{\psi_\varepsilon}(0)-Y^{\psi^{\varepsilon}}(0)\big)+\Pi_1^{\psi^{\varepsilon}}(t)+\Pi_2^{\psi^{\varepsilon}}(t)+\Pi_3^{\psi^{\varepsilon}}(t)+\Pi_4^{\psi^{\varepsilon}}(t).
\end{align*}
Recalling Corollary \ref{coro-uni-bou}, we have
\begin{align}\label{Y-psi-boundedness}
\sup_{\varepsilon\in(0,1)}\sup_{\omega\in\Omega}\Big(\sup_{0\leq t\leq T}\|Y^{\psi_{\varepsilon}}(t) \|_{L^2(\mathbb{R}^d)}^2+ \|Y^{\psi_{\varepsilon}}\|_{L^p(0,T;L^r(\mathbb{R}^d))}^p\Big)
\leq C_N<\infty.
\end{align}
By (\ref{Stri-est-1}), we find
\begin{eqnarray*}
&&\sup_{t\in[0,T]}\|S_{t} \Big(X^{\psi_\varepsilon}(0)-Y^{\psi^{\varepsilon}}(0)\Big)\|_{L^2(\mathbb{R}^d)}
+
\|S_{\cdot} \Big(X^{\psi_\varepsilon}(0)-Y^{\psi^{\varepsilon}}(0)\Big)\|_{L^p(0,T;L^r(\mathbb{R}^d))}\\
&&\leq
C\|X^{\psi_\varepsilon}(0)-Y^{\psi^{\varepsilon}}(0)\|_{L^2(\mathbb{R}^d)}.
\end{eqnarray*}
 Using arguments similar to (\ref{eq Up 03}), we infer
 \begin{align*}
& \sup_{0\leq s\leq t\wedge\tau_{M_\eta}^{\varepsilon}}\|\Pi_1^{\psi^{\varepsilon}}(s)\|_{L^2(\mathbb{R}^d)} +\|\Pi_1^{\psi^{\varepsilon}}\|_{L^p(0,t\wedge\tau_{M_\eta}^{\varepsilon};L^r(\mathbb{R}^d))}
\\
&\leq C|  \lambda| t^{1-\frac{d\sigma}2}(\|X^{\psi_\varepsilon}\|_{L^{p}(0,t\wedge\tau_{M_\eta}^{\varepsilon}; L^{r}(\mathbb{R}^{d}))}+\|Y^{\psi_\varepsilon}\|_{L^{p}(0,t\wedge\tau_{M_\eta}^{\varepsilon}; L^{r}(\mathbb{R}^{d}))})^{2 \sigma}\|X^{\psi_\varepsilon}-Y^{\psi_\varepsilon}\|_{L^{p}(0,t\wedge\tau_{M_\eta}^{\varepsilon}; L^{r}(\mathbb{R}^{d}))}\\
&\leq C|  \lambda| t^{1-\frac{d\sigma}2}(M_\eta+C_N)^{2\sigma}\|X^{\psi_\varepsilon}-Y^{\psi_\varepsilon}\|_{L^{p}(0,t\wedge\tau_{M_\eta}^{\varepsilon}; L^{r}(\mathbb{R}^{d}))},
\end{align*}
where have used \eqref{Y-psi-boundedness} in the last inequality.

By the stochastic Strichartz inequality \eqref{sto-stri-est-1}, a similar
argument as proving (\ref{Uniform estimate I2}) shows
 \begin{align}
 \begin{split}\label{prop-proof-est-I2}
   & \mathbb{E} \|\Pi_2^{\psi^{\varepsilon}}\|_{ L^{\infty}(0,T; L^2(\mathbb{R}^d)) \cap L^{p}(0, T; L^r(\mathbb{R}^d))}^p\\
  & \leq C_p \mathbb{E}\sup_{0\leq s\leq T} \bigg\|  \int_{0}^{s} \int_{B}  S_{s-r}\Big[\Phi^\varepsilon(z, X^{\psi_\varepsilon}(r))-X^{\psi_\varepsilon}(r  )\Big] \tilde{N}^{\varepsilon^{-1}\psi_\varepsilon}(\mathrm{d} r, \mathrm{d} z)   \bigg\|^p_{L^2(\mathbb{R}^d)}\\
&\hspace{0.5cm}+ C_p \mathbb{E}\left\|   \int_{0}^{\cdot} \int_{B}  S_{\cdot-r}\left[\Phi^\varepsilon(z, X^{\psi_\varepsilon}(r ))-X^{\psi_\varepsilon}(r  )\right] \tilde{N}^{\varepsilon^{-1}\psi_\varepsilon}(\mathrm{d} r, \mathrm{d} z)   \right\|^p_{L^p(0,T;L^r(\mathbb{R}^d))}\\
&\leq C_{p, m} \mathbb{E}\left(\int_{0}^{T} \int_{B}\varepsilon^2|z|_{\mathbb{R}^{m}}^{2}\|X^{\psi_\varepsilon}(r )\|_{L^{2}\left(\mathbb{R}^{d}\right)}^{2}\varepsilon^{-1}\psi_\varepsilon (r,z) \nu(\mathrm{d} z) \mathrm{d} r\right)^{\frac{p}{2}} \\
&\hspace{0.5cm}+C_{p, m} \mathbb{E}\left(\int_{0}^{T} \int_{B}\varepsilon^p|z|_{\mathbb{R}^{m}}^{p}\|X^{\psi_\varepsilon}(r )\|_{L^{2}\left(\mathbb{R}^{d}\right)}^{p}\varepsilon^{-1}\psi_\varepsilon (r,z) \nu(\mathrm{d} z) \mathrm{d} r\right) \\
&\leq C_{p,m}K_1 \big(\varepsilon^{\frac{p}2}+\varepsilon^{p-1}\big)\|u_0\|_{L^2(\mathbb{R}^d)}^p.
\end{split}
 \end{align}
Again, by applying the Strichartz inequality (\ref{Stri-est-2}) and (\ref{Stri-est-3}) and using calculations similar to those in (\ref{Uniform estimate I3}), we have
 \begin{align*}
& \sup_{0\leq s\leq T}\|\Pi_3^{\psi^{\varepsilon}}(s)\|_{L^2(\mathbb{R}^d)} +\|\Pi_3^{\psi^{\varepsilon}}\|_{L^p(0,T;L^r(\mathbb{R}^d))}\\
&\leq  \varepsilon^{-1}C \int_{0}^{T} \int_{B}\| H(\varepsilon,z, X^{\psi_\varepsilon}(s)) \|_{L^{2}(\mathbb{R}^d)}\psi^{\varepsilon}(s,z) \nu(\mathrm{d} z) \mathrm{d} s\nonumber\\
&\leq  \varepsilon^{-1}C_m \int_{0}^{T}\left(\int_{B}\varepsilon^2 |z|_{\mathbb{R}^{m}}^{2}\psi^{\varepsilon}(s,z) \nu(\mathrm{d} z)\right)\|X^{\psi_\varepsilon}(s)\|_{L^{2}(\mathbb{R}^d)} \mathrm{d} s \nonumber\\
&\leq \varepsilon C_mK_2 \|u_0\|_{L^{2}(\mathbb{R}^{d})}.
 \end{align*}

Now consider $\Pi_4^{\psi^{\varepsilon}}$. Applying the Minkowski inequality, Strichartz's estimate and the Cauchy-Schwartz inequality, we obtain
 \begin{align*}
& \sup_{0\leq s\leq t\wedge\tau_{M_\eta}^{\varepsilon}}\|\Pi_4^{\psi^{\varepsilon}}(s)\|_{L^2(\mathbb{R}^d)} +\|\Pi_4^{\psi^{\varepsilon}}\|_{L^p(0,t\wedge\tau_{M_\eta}^{\varepsilon};L^r(\mathbb{R}^d))}\\
&= \sup_{0\leq s\leq t\wedge\tau_{M_\eta}^{\varepsilon}} \bigg\| \int_{0}^{s } \int_{B} S_{s -r}\Big[\Big(\sum_{j=1}^{m}\mathrm{i}  z_{j} g_{j}(X^{\psi_{\varepsilon}}(r))-\sum_{j=1}^{m}\mathrm{i}  z_{j} g_{j}(Y^{\psi_{\varepsilon}}(r))\Big)(\psi_{\varepsilon}(r,z)-1)\Big] \nu(\mathrm{d} z) \mathrm{d} r \bigg\|_{L^2(\mathbb{R}^d)}\\
&\hspace{0.5cm}+ \bigg\| \int_{0}^{\cdot } \int_{B} S_{\cdot -r}\Big[\Big(\sum_{j=1}^{m}\mathrm{i}  z_{j} g_{j}(X^{\psi_{\varepsilon}}(r))-\sum_{j=1}^{m}\mathrm{i}  z_{j} g_{j}(Y^{\psi_{\varepsilon}}(r))\Big)(\psi_{\varepsilon}(r,z)-1)\Big] \nu(\mathrm{d} z) \mathrm{d} r  \bigg\|_{L^p(0,t\wedge\tau_{M_\eta}^{\varepsilon};L^r(\mathbb{R}^d))}\\
&\leq \int_{0}^{t\wedge\tau_{M_\eta}^{\varepsilon} } \int_{B}\Big\| \Big(\sum_{j=1}^{m}\mathrm{i}  z_{j} g_{j}(X^{\psi_{\varepsilon}}(r))-\sum_{j=1}^{m}\mathrm{i}  z_{j} g_{j}(Y^{\psi_{\varepsilon}}(r))\Big)(\psi_{\varepsilon}(r,z)-1)\Big\|_{L^2(\mathbb{R}^d)} \nu(\mathrm{d} z) \mathrm{d} r\\
&\hspace{0.5cm}+  C\int_{0}^{t\wedge\tau_{M_\eta}^{\varepsilon} }\bigg\| \int_B \Big(\sum_{j=1}^{m}\mathrm{i}  z_{j} g_{j}(X^{\psi_{\varepsilon}}(r))-\sum_{j=1}^{m}\mathrm{i}  z_{j} g_{j}(Y^{\psi_{\varepsilon}}(r))\Big)(\psi_{\varepsilon}(r,z)-1)\nu(\dif z)\bigg\|_{L^2(\mathbb{R}^d)}\dif r\\
&\leq C \int_{0}^{t\wedge\tau_{M_\eta}^{\varepsilon} } \int_{B}\Big\| \Big(\sum_{j=1}^{m}\mathrm{i}  z_{j} g_{j}(X^{\psi_{\varepsilon}}(r))-\sum_{j=1}^{m}\mathrm{i}  z_{j} g_{j}(Y^{\psi_{\varepsilon}}(r))\Big)\Big\|_{L^2(\mathbb{R}^d)}
|\psi_{\varepsilon}(r,z)-1| \nu(\mathrm{d} z) \mathrm{d} r\\
&\leq C\int_0^{t\wedge\tau_{M_\eta}^{\varepsilon}} \int_B|z|_{\mathbb{R}^m}\Big\|\Big(\sum_{j=1}^m|g_{j}(X^{\psi_{\varepsilon}}(r))-g_{j}(Y^{\psi_{\varepsilon}}(r))|^2\Big)^{\frac12}\Big\|_{L^2(\mathbb{R}^d)}|\psi_{\varepsilon}(r,z)-1| \nu(\mathrm{d} z) \mathrm{d} r\\
&\leq C_m\int_0^{t\wedge\tau_{M_\eta}^{\varepsilon}} \int_B\big\|X^{\psi_{\varepsilon}}(r)-Y^{\psi_{\varepsilon}}(r)\big\|_{L^2(\mathbb{R}^d)} |z|_{\mathbb{R}^m}|\psi_{\varepsilon}(r,z)-1| \nu(\mathrm{d} z) \mathrm{d} r\\
&:= C_m\int_0^{t\wedge\tau_{M_\eta}^{\varepsilon}} \big\|X^{\psi_{\varepsilon}}(r)-Y^{\psi_{\varepsilon}}(r)\big\|_{L^2(\mathbb{R}^d)} k_{\varepsilon}(r) \dif r,
 \end{align*}
where $ k_{\varepsilon}(s) = \int_B |z|_{\mathbb{R}^m}|\psi_{\varepsilon}(s,z)-1| \nu(\mathrm{d} z)$.

Combining the above estimates, we get, for any $l\in[0,t]$,
\begin{align*}
&\big\|X^{\psi_{\varepsilon}}-Y^{\psi_{\varepsilon}}\big\|_{ L^{\infty}(0,l\wedge\tau_{M_\eta}^{\varepsilon}; L^2(\mathbb{R}^d)) \cap L^{p}(0, l\wedge\tau_{M_\eta}^{\varepsilon} ; L^r(\mathbb{R}^d))}\\
& \leq  C\|X^{\psi_\varepsilon}(0)-Y^{\psi^{\varepsilon}}(0)\|_{L^2(\mathbb{R}^d)}\\
 &\hspace{0.5cm}+C|  \lambda| t^{1-\frac{d\sigma}2} ({M_\eta}+C_N)^{2\sigma}\|X^{\psi_\varepsilon}-Y^{\psi_\varepsilon}\|_{L^{\infty}(0,t\wedge\tau_{M_\eta}^{\varepsilon}; L^2(\mathbb{R}^d)) \cap L^{p}(0, t\wedge\tau_{M_\eta}^{\varepsilon} ; L^r(\mathbb{R}^d))}\\
&\hspace{0.5cm}+ \|\Pi_2^{\psi^{\varepsilon}}\|_{L^{\infty}(0,T; L^2(\mathbb{R}^d)) \cap L^{p}(0, T; L^r(\mathbb{R}^d))}+\varepsilon C_mK_2 \|u_0\|_{L^{2}(\mathbb{R}^{d})}\\
&\hspace{0.5cm}+C_m\int_0^l \big\|X^{\psi_{\varepsilon}}-Y^{\psi_{\varepsilon}}\big\|_{L^{\infty}(0,s\wedge\tau_{M_\eta}^{\varepsilon}; L^2(\mathbb{R}^d)) \cap L^{p}(0, s\wedge\tau_{M_\eta}^{\varepsilon} ; L^r(\mathbb{R}^d))} k_{\varepsilon}(s) \dif s.
\end{align*}
Applying Gronwall's Lemma, and noting that, according to Lemma \ref{lemma-est-ass},
\begin{align*}
     k_{\varepsilon}\in L^1(0,T;\mathbb{R}^+)\quad\text{and}\quad \int_0^T k_{\varepsilon}(s)  \dif s\leq C_N,
\end{align*}
we infer, for any $t\in[0,T]$,
\begin{align*}
&\big\|X^{\psi_{\varepsilon}}-Y^{\psi_{\varepsilon}}\big\|_{ L^{\infty}(0,t\wedge\tau_{M_\eta}^{\varepsilon}; L^2(\mathbb{R}^d)) \cap L^{p}(0, t\wedge\tau_{M_\eta}^{\varepsilon} ; L^r(\mathbb{R}^d))}\\
&\leq \Big(C\|X^{\psi_\varepsilon}(0)-Y^{\psi^{\varepsilon}}(0)\|_{L^2(\mathbb{R}^d)}\\
&\hspace{0.5cm}+
C|  \lambda| t^{1-\frac{d\sigma}2} ({M_\eta}+C_N)^{2\sigma}\|X^{\psi_\varepsilon}-Y^{\psi_\varepsilon}\|_{L^{\infty}(0,t\wedge\tau_{M_\eta}^{\varepsilon}; L^2(\mathbb{R}^d)) \cap L^{p}(0, t\wedge\tau_{M_\eta}^{\varepsilon} ; L^r(\mathbb{R}^d))}\\
&\hspace{0.5cm}+ \|\Pi_2^{\psi^{\varepsilon}}\|_{L^{\infty}(0,T; L^2(\mathbb{R}^d)) \cap L^{p}(0, T; L^r(\mathbb{R}^d))}+\varepsilon C_mK_2 \|u_0\|_{L^{2}(\mathbb{R}^{d})}
\Big)e^{C_m\int_0^T  k_{\varepsilon}(s) \dif s}\\
&\leq C_N\Big(\|X^{\psi_\varepsilon}(0)-Y^{\psi^{\varepsilon}}(0)\|_{L^2(\mathbb{R}^d)}\\
&\hspace{0.5cm}+
|  \lambda| t^{1-\frac{d\sigma}2} ({M_\eta}+C_N)^{2\sigma}\|X^{\psi_\varepsilon}-Y^{\psi_\varepsilon}\|_{L^{\infty}(0,t\wedge\tau_{M_\eta}^{\varepsilon}; L^2(\mathbb{R}^d)) \cap L^{p}(0, t\wedge\tau_{M_\eta}^{\varepsilon} ; L^r(\mathbb{R}^d))}\\
&\hspace{0.5cm}+ \|\Pi_2^{\psi^{\varepsilon}}\|_{L^{\infty}(0,T; L^2(\mathbb{R}^d)) \cap L^{p}(0, T; L^r(\mathbb{R}^d))}+\varepsilon C_mK_2 \|u_0\|_{L^{2}(\mathbb{R}^{d})}
\Big).
\end{align*}
By choosing $t_{\eta,N}\in(0,T]$ small enough, depending on $M_\eta$ and $N$ and independing on $\varepsilon$, such that $ C_N|  \lambda| t_{\eta,N}^{1-\frac{d\sigma}2} ({M_\eta}+C_1)^{2\sigma}<\frac12$, we infer from the above inequality
\begin{align}
&\big\|X^{\psi_{\varepsilon}}-Y^{\psi_{\varepsilon}}\big\|_{ L^{\infty}(0,t_{\eta,N}\wedge\tau_{M_\eta}^{\varepsilon}; L^2(\mathbb{R}^d)) \cap L^{p}(0, t_{\eta,N}\wedge\tau_{M_\eta}^{\varepsilon} ; L^r(\mathbb{R}^d))}\nonumber\\
&\leq 2C_N\Big(\|X^{\psi_\varepsilon}(0)-Y^{\psi^{\varepsilon}}(0)\|_{L^2(\mathbb{R}^d)}
+ \hspace{0.5cm}\|\Pi_2^{\psi^{\varepsilon}}\|_{L^{\infty}(0,T; L^2(\mathbb{R}^d)) \cap L^{p}(0, T; L^r(\mathbb{R}^d))}+\varepsilon \|u_0\|_{L^{2}(\mathbb{R}^{d})}
\Big)\nonumber\\
&= 2C_N\Big(\|\Pi_2^{\psi^{\varepsilon}}\|_{L^{\infty}(0,T; L^2(\mathbb{R}^d)) \cap L^{p}(0, T; L^r(\mathbb{R}^d))}+\varepsilon C_m K_2\|u_0\|_{L^{2}(\mathbb{R}^{d})}\Big).\label{proof-prof-2-eq-10}
\end{align}
To obtain the last equality, we have used $X^{\psi_\varepsilon}(0)=Y^{\psi^{\varepsilon}}(0)=u_0$.
Similaly, we have
\begin{align*}
&\big\|X^{\psi_{\varepsilon}}-Y^{\psi_{\varepsilon}}\big\|_{ L^{\infty}(t_{\eta,N}\wedge\tau_{M_\eta}^{\varepsilon},(2t_{\eta,N})\wedge\tau_{M_\eta}^{\varepsilon}; L^2(\mathbb{R}^d)) \cap L^{p}(t_{\eta,N}\wedge\tau_{M_\eta}^{\varepsilon}, (2t_{\eta,N})\wedge\tau_{M_\eta}^{\varepsilon} ; L^r(\mathbb{R}^d))}\\
&\leq 2C_N\Big(\|X^{\psi_\varepsilon}(t_{\eta,N}\wedge\tau_{M_\eta}^{\varepsilon})-Y^{\psi^{\varepsilon}}(t_{\eta,N}\wedge\tau_{M_\eta}^{\varepsilon})\|_{L^2(\mathbb{R}^d)}
+ \|\Pi_2^{\psi^{\varepsilon}}\|_{L^{\infty}(0,T; L^2(\mathbb{R}^d)) \cap L^{p}(0, T; L^r(\mathbb{R}^d))}\\
&\hspace{1cm}+\varepsilon C_m K_2 \|u_0\|_{L^{2}(\mathbb{R}^{d})}
\Big)\\
&\leq 2C_N(2C_N+1)\Big(\|\Pi_2^{\psi^{\varepsilon}}\|_{L^{\infty}(0,T; L^2(\mathbb{R}^d)) \cap L^{p}(0, T; L^r(\mathbb{R}^d))}+\varepsilon C_m K_2 \|u_0\|_{L^{2}(\mathbb{R}^{d})}
\Big),
\end{align*}
where the last inequality follows from \eqref{proof-prof-2-eq-10}.

Iterating similar estimates for $1\leq k\leq [\frac{T}{t_{\eta,N}}]$, there exists a constant $C_{k,N}$ such that
\begin{align*}
&\big\|X^{\psi_{\varepsilon}}-Y^{\psi_{\varepsilon}}\big\|_{ L^{\infty}((kt_{\eta,N})\wedge\tau_{M_\eta}^{\varepsilon},((k+1)t_{\eta,N})\wedge\tau_{M_\eta}^{\varepsilon}; L^2(\mathbb{R}^d)) \cap L^{p}(t_{\eta,N}\wedge\tau_{M_\eta}^{\varepsilon}, (2t_{\eta,N})\wedge\tau_{M_\eta}^{\varepsilon} ; L^r(\mathbb{R}^d))}\\
&\leq 2C_N\Big(\|X^{\psi_\varepsilon}((kt_{\eta,N})\wedge\tau_{M_\eta}^{\varepsilon})-Y^{\psi^{\varepsilon}}((kt_{\eta,N})\wedge\tau_{M_\eta}^{\varepsilon})\|_{L^2(\mathbb{R}^d)}
+ \|\Pi_2^{\psi^{\varepsilon}}\|_{L^{\infty}(0,T; L^2(\mathbb{R}^d)) \cap L^{p}(0, T; L^r(\mathbb{R}^d))}\\
&\hspace{1cm}+\varepsilon C_m K_2\|u_0\|_{L^{2}(\mathbb{R}^{d})}
\Big)\\
&\leq C_{k,N}\Big(\|\Pi_2^{\psi^{\varepsilon}}\|_{L^{\infty}(0,T; L^2(\mathbb{R}^d)) \cap L^{p}(0, T; L^r(\mathbb{R}^d))}+\varepsilon C_mK_2\|u_0\|_{L^{2}(\mathbb{R}^{d})}
\Big).
\end{align*}
Hence, we can find a constant $C_{\frac{T}{t_{\eta,N}},N}$ such that
\begin{align*}
&\big\|X^{\psi_{\varepsilon}}-Y^{\psi_{\varepsilon}}\big\|_{ L^{\infty}(0,T\wedge\tau_{M_\eta}^{\varepsilon}; L^2(\mathbb{R}^d)) \cap L^{p}(0, T\wedge\tau_{M_\eta}^{\varepsilon}; L^r(\mathbb{R}^d))}\\
&\leq C_{\frac{T}{t_{\eta,N}},N}\Big(\|\Pi_2^{\psi^{\varepsilon}}\|_{L^{\infty}(0,T; L^2(\mathbb{R}^d)) \cap L^{p}(0, T; L^r(\mathbb{R}^d))}+\varepsilon C_mK_2 \|u_0\|_{L^{2}(\mathbb{R}^{d})}
\Big).
\end{align*}
Integrate over $\Omega$ and apply \eqref{prop-proof-est-I2} to obtain
\begin{align*}
&\mathbb{E}\big\|X^{\psi_{\varepsilon}}-Y^{\psi_{\varepsilon}}\big\|_{ L^{\infty}(0,T\wedge\tau_{M_\eta}^{\varepsilon}; L^2(\mathbb{R}^d)) \cap L^{p}(0, T\wedge\tau_{M_\eta}^{\varepsilon} ; L^r(\mathbb{R}^d))}\\
    &\leq  C_{\frac{T}{t_{\eta,N}},N}\big( C_{p,m}K_1^{\frac1p}\big(\varepsilon^{\frac{1}2}+\varepsilon^{\frac{p-1}{p}}\big)+ \varepsilon C_mK_2 \big)\|u_0\|_{L^{2}(\mathbb{R}^{d})}
    \rightarrow 0,\quad \quad \text{as }\varepsilon\rightarrow0.
\end{align*}
It follows that, for any $\delta>0$, $\eta>0$ and $M_\eta$(appearing in Lemma \ref{lem uniform}),
\begin{align*}
&\limsup_{\varepsilon\rightarrow0}P\Big(\big\|X^{\psi_{\varepsilon}}-Y^{\psi_{\varepsilon}}\big\|_{ E^{(T,p)}}\geq \delta \Big)\\
&=\limsup_{\varepsilon\rightarrow0}P\Big(\big\|X^{\psi_{\varepsilon}}-Y^{\psi_{\varepsilon}}\big\|_{ L^{\infty}(0,T; L^2(\mathbb{R}^d)) \cap L^{p}(0, T ; L^r(\mathbb{R}^d))}\geq \delta \Big)\\
&\leq \limsup_{\varepsilon\rightarrow0}P\Big(\big\|X^{\psi_{\varepsilon}}-Y^{\psi_{\varepsilon}}\big\|_{ L^{\infty}(0,T\wedge\tau_{M_\eta}^{\varepsilon}; L^2(\mathbb{R}^d)) \cap L^{p}(0, T\wedge\tau_{M_\eta}^{\varepsilon} ; L^r(\mathbb{R}^d))}\geq \delta,\;\tau_{M_\eta}^{\varepsilon}=T \Big)\\
&\quad+ \limsup_{\varepsilon\rightarrow0}P\Big(\tau_{M_\eta}^{\varepsilon}<T \Big)\\
&\leq \limsup_{\varepsilon\rightarrow0}P\Big(\big\|X^{\psi_{\varepsilon}}-Y^{\psi_{\varepsilon}}\big\|_{ L^{\infty}(0,T\wedge\tau_{M_\eta}^{\varepsilon}; L^2(\mathbb{R}^d)) \cap L^{p}(0, T\wedge\tau_{M_\eta}^{\varepsilon} ; L^r(\mathbb{R}^d))}\geq \delta,\;\tau_{M_\eta}^{\varepsilon}=T \Big)\\
&\quad+ \limsup_{\varepsilon\rightarrow0}P\Big(\big\|X^{\psi_{\varepsilon}}\big\|_{L^{p}(0, T;L^{r}(\mathbb{R}^{d}))} \geqslant {M_\eta}\Big)\\
&\leq \limsup_{\varepsilon\rightarrow0}\frac{1}{\delta}\mathbb{E}\big\|X^{\psi_{\varepsilon}}-Y^{\psi_{\varepsilon}}\big\|_{ L^{\infty}(0,T\wedge\tau_{M_\eta}^{\varepsilon}; L^2(\mathbb{R}^d)) \cap L^{p}(0, T\wedge\tau_{M_\eta}^{\varepsilon} ; L^r(\mathbb{R}^d))}+\eta\\
&\leq\eta,
\end{align*}
which completes the proof of Proposition \ref{lem LDP 2}.

\end{proof}
%------------------------------------

\end{document}